\documentclass[11pt,reqno,tbtags]{amsart}
\usepackage{amssymb,bbm}
\numberwithin{equation}{section}

\makeatletter
\newtheorem*{rep@theorem}{\rep@title}
\newcommand{\newreptheorem}[2]{%
\newenvironment{rep#1}[1]{%
 \def\rep@title{#2 \ref{##1}}%
 \begin{rep@theorem}}%
 {\end{rep@theorem}}}
\makeatother

\newtheorem{theorem}{Theorem}[section]
\newreptheorem{theorem}{Theorem}
\newtheorem{lemma}[theorem]{Lemma}
\newtheorem{proposition}[theorem]{Proposition}
\newtheorem{corollary}[theorem]{Corollary}

\theoremstyle{definition}

\theoremstyle{remark}

\newcounter{thmenumerate}

\newcounter{xenumerate}

\newcommand\E{\operatorname{\mathbb E{}}}
\renewcommand\Pr{\operatorname{\mathbb P{}}}

\newcommand\eps{\varepsilon}

\renewcommand\phi{\varphi}

\newcommand\la{\lambda}

\newcommand\cA{\mathcal A}
\newcommand\cB{\mathcal B}
\newcommand\cC{\mathcal C}
\newcommand\cD{\mathcal D}
\newcommand\cE{\mathcal E}
\newcommand\cF{\mathcal F}
\newcommand\cG{\mathcal G}
\newcommand\cH{\mathcal H}
\newcommand\cI{\mathcal I}

\newcommand\cK{\mathcal K}
\newcommand\cL{\mathcal L}
\newcommand\cM{\mathcal M}
\newcommand\cN{\mathcal N}
\newcommand\cP{\mathcal P}
\newcommand\cR{\mathcal R}

\newcommand\cS{\mathcal S}

\newcommand\cX{\mathcal X}

\newcommand\N{{\mathbb N}}
\newcommand\Z{{\mathbb Z}}
\newcommand\I{{\mathbb I}}
\newcommand\R{{\mathbb R}}

\newcommand\1{{\mathbbm 1}}

\begin{document}
\title{The Supermarket Model with Arrival Rate Tending to One}

%\date{16 May 2011} % (typeset \today)}

\author{Graham Brightwell}
\address{Department of Mathematics, London School of Economics,
Houghton Street, London WC2A 2AE, United Kingdom}
\email{g.r.brightwell@lse.ac.uk}
\urladdr{http://www.maths.lse.ac.uk/Personal/graham/}

\author{Malwina J. Luczak} \thanks{The research of Malwina Luczak is supported by an EPSRC Leadership Fellowship, grant reference EP/J004022/1.}
\address{School of Mathematics and Statistics, University of Sheffield}
\email{m.luczak@sheffield.ac.uk}

\subjclass[2000]{}

\begin{abstract}
In the supermarket model, there are $n$ queues, each with a single server.  Customers arrive in a Poisson process
with arrival rate $\la n$, where $\la = \la(n) \in (0,1)$.  Upon arrival, a customer selects $d=d(n)$ servers
uniformly at random, and joins the queue of a least-loaded server amongst those chosen.  Service times are independent
exponentially distributed random variables with mean~1.
In this paper, we analyse the behaviour of the supermarket model in a regime where $\lambda(n)$
tends to~1, and $d(n)$ tends to infinity, as $n \to \infty$.  For suitable triples $(n,d,\la)$, we identify a
subset $\cN$ of the state space where the process remains for a long time in equilibrium.  We further show that the
process is rapidly mixing when started in $\cN$, and give bounds on the speed of mixing for more general initial
conditions.
\end{abstract}

\maketitle

\section{Introduction}\label{Sintro}

The supermarket model is a Markov chain model for a dynamic load-balancing process.  There are $n$ servers, and customers arrive according to a Poisson process
with rate $\la = \la (n)< 1$.  On arrival, a customer inspects $d = d(n)$ queues, chosen uniformly at random with replacement, and joins
a shortest queue among those inspected (in case of a tie, the first shortest queue in the list is joined).  Each server serves one customer at a time,
and service times are iid random variables, with an exponential distribution of mean~1.

A number of authors~\cite{mitz96a,mitz96b,vdk96,gr00,gr04,ln05,lmcd06,lmcd07,luc08,f11} have studied the supermarket model, as well as various
extensions, e.g., to the setting of a Jackson network~\cite{ms99} and to a version with one queue saved in memory~\cite{mps02, ln12}.

Previous work has concentrated on the case where $\la$ and $d$ are held fixed as $n$ tends to infinity.  As with other related models, there is a
dramatic change when $d$ is increased from~1 to~2: if $d=1$, the maximum queue length in equilibrium is of order $\log n$, while if
$d$ is a constant at least~2, then the maximum queue length in equilibrium is of order $\log \log n / \log d$.

Luczak and McDiarmid~\cite{lmcd06} prove that, for fixed $\la$ and $d$, the sequence of Markov chains indexed by $n$ is rapidly mixing: as $n \to \infty$,
the time for the system to converge to equilibrium is of order $\log n$, provided the initial state has not too many customers and no very long queue.
Also, they show that, for $d \ge 2$, with probability tending to~1 as $n \to \infty$, in the equilibrium distribution the
maximum queue length takes one of at most~2 values, and that these values are $\log \log n / \log d + O(1)$.

Consider the infinite system of differential equations
\begin{equation}
\label{eq.diff-eq}
\frac{dv_k(t)}{dt} = \la (v_{k-1}(t)^d - v_k(t)^d) - (v_k(t) - v_{k+1}(t)), \quad \quad \quad k\ge 1,
\end{equation}
where $v_0(t) = 1$ for all $t$. For an initial condition $v(0)$ such that $1 \ge v_1(0) \ge v_2 (0) \ge \ldots \ge 0$ and $v_k(0) \to 0$ as $k \to \infty$,
there is a unique solution $v(t)$ ($t \ge 0$), with $v(t) = (v_k(t))_{k\ge 1}$, which is such that $1 \ge v_1(t) \ge v_2 (t) \ge \ldots \ge 0$ and
$v_k(t) \to 0$ as $k \to \infty$, for each $t \ge 0$.
It follows from earlier work~\cite{vdk96,gr00,gr04,ln05,lmcd07} that, with high probability, for each $k$, the proportion of queues of length at least
$k$ at time $t$ stays ``close to'' $v_k(t)$ over a bounded time interval (or an interval whose length tends to infinity at most polynomially with $n$),
assuming this is the case at time 0.

The system~(\ref{eq.diff-eq}) has a unique, attractive, fixed point $\pi = (\pi (k))_{k \ge 1}$, such that $\pi (k) \to 0$ as $k \to \infty$, given by
\begin{equation} \label{eq:fixed-point}
\pi (k) = \la^{1 + \cdots + d^{k-1}}, \quad \quad \quad k \ge 1.
\end{equation}
In equilibrium, with high probability, the proportion of queues of length at least $k$ is close to $\pi (k)$
for each $k \ge 1$, over time intervals of length polynomial in $n$; see~\cite{gr00,gr04,lmcd06,lmcd07}.

%\medbreak

In this paper, we extend the above results about equilibrium behaviour and rapid mixing to some regimes where $\la(n) \to 1$ and $d(n) \to \infty$ as
$n \to \infty$.  For $\lambda$ and $d$ functions of $n$, there is no single limiting differential equation (\ref{eq.diff-eq}), but rather
a sequence of approximating differential equations, each with their own solutions and fixed points.  In this paper, we do not address the question of
whether such approximations to the evolution of the process are valid in generality, focussing solely on equilibrium behaviour and the time to reach
equilibrium.  We show that, for a wide range of triples $(n,d,\la)$, the maximum queue length in equilibrium is equal to
$$
k = k^{\la,d} := \left \lceil \frac{\log(1-\la)^{-1}}{\log d} \right\rceil
$$
over long stretches of time, with high probability.  Moreover, in equilibrium, with high probability, most queues have length exactly $k$, and we are
able to estimate precisely the numbers of queues of each smaller length.  In other words, this is a regime where we have ``nearly exact'' load balancing
between the servers.  As we shall discuss later, this is still consistent with the principle that the proportion of queues of length at least $k$
is close to $\pi(k)$ for each~$k$.  We further prove that the mixing time from a ``good'' state is at most of order $k d^{k-1} \log n$, and we show that
this is roughly best possible.  We also prove general bounds on the mixing time, in terms of the initial number of customers and the initial maximum
queue length, and show that these bounds are also roughly best possible.

%\medbreak

We will shortly state our main results precisely, but first we describe the supermarket model more carefully.  In fact, we describe a natural
discrete-time version of the process, which we shall work with throughout; as is standard, one may convert results about the discrete time version to the
continuous model, with the understanding that one unit of time in the continuous model corresponds to about $(1+\la) n$ steps of the
discrete model.

A {\em queue-lengths vector} is an $n$-tuple $(x(1), \dots, x(n))$ whose entries are non-negative integers.  If $x(j)=i$, we say that queue~$j$ has
{\em length}~$i$, or that there are $i$ {\em customers} in queue $j$; we think of these customers as in {\em positions} $1, \dots, i$ in the queue.
We use similar terminology throughout; for instance, to say that a customer
{\em arrives} and {\em joins queue $j$} means that $x(j)$ increases by~1, and to say that a customer in queue~$j$ {\em departs} or {\em is served}
means that $x(j)$ decreases by~1.  Given a queue-lengths vector $x$, we write
$\|x\|_1 = \sum_{j=1}^n x(j)$ to denote the total number of customers in state $x$, and $\|x\|_\infty = \max x(j)$ to denote the maximum queue length
in state $x$.

For each $i \ge 0$, and each $x \in \Z_+^n$, we define $u_i(x)$ to be the proportion of queues in $x$ with length at least $i$.
So $u_0(x) = 1$ for all $x$, and, for each fixed $x$, the $u_i(x)$ form a non-increasing sequence of multiples of $1/n$, such that $u_i(x) = 0$
eventually.  The sequence $(u_i(x))_{i\ge 0}$ captures the ``profile'' of a queue-lengths vector $x$, and we shall describe various sets of
queue-lengths vectors, and functions of the queue-lengths vector, in terms of the $u_i(x)$.

For positive integers $n$ and $d$, and $\la \in (0,1)$, we now define the {\em $(n,d,\la)$-supermarket process}.
This process is a discrete-time Markov chain $(X_t)$, whose state space is the set $\Z_+^n$ of queue-lengths vectors, and where
transitions occur at non-negative integer times.  Each transition is either a customer {\em arrival}, with probability $\lambda/(1 + \lambda)$,
or a {\em potential departure}, with probability $1/(1+ \lambda)$.  If there is a potential departure, then a queue $K$ is selected uniformly at
random from $\{1,\dots, n\}$: if there is a customer in queue $K$, then they are served and depart the system.  If there is an arrival, then
$d$ queues are selected uniformly at random, with replacement, from $\{1,\dots,n\}$, and the arriving customer joins a shortest queue among those
selected.  To be precise, a $d$-tuple $(K_1, \dots, K_d)$ is selected, and the customer joins queue $k =K_j$, where $j$ is the least index such
that $x(K_j)$ is minimal among $\{ x(K_1), \dots, x(K_d)\}$.

For $x \in \Z_+^n$, $(X_t^x)$ denotes a copy of the $(n,d,\la)$-supermarket process $(X_t)$ where $X_0 = x$ a.s., although when it is clear from the
context we shall prefer to use the simpler notation $(X_t)$.  Throughout, we let $(Y_t)$ denote a copy of the process in equilibrium.  Of course, the
processes depend on the parameters $(n,d,\la)$, but we suppress this dependence in the notation.  Throughout the paper, we use $(\cF_t)$ to denote the
natural filtration of the process $(X_t)$.  We use the notation $\Pr(\cdot)$ freely to denote probability in whatever space we are working in.
As before, for $\la \in (0,1)$ and $d \in \N$, we set $k^{\la,d} = \left \lceil \log(1-\la)^{-1}/\log d \right \rceil$.

We now state our main results.  First, we describe sets of queue-lengths vectors $\cN^\eps(n,d,\la,k)$: our aim is to prove that, for
suitable values of $n$, $d$ and $\la$, $k = k^{\la,d}$, and appropriately small $\eps$, an equilibrium copy of the $(n,d,\la)$-supermarket process
spends almost all of its time in the set $\cN^\eps(n,d,\la,k)$.

For $\eps \in (0,1/10]$, $\la \in (0,1)$, and positive integers $n$, $d$ and $k$, let $\cN^\eps = \cN^\eps(n,d,\la,k)$ be the set of all queue-lengths
vectors $x$ such that: $u_{k+1}(x) = 0$ and, for $1\le j \le k$,
$$
(1-5\eps) (1-\la) (\la d)^{j-1} \le 1- u_j(x) \le (1+5\eps) (1-\la) (\la d)^{j-1}.
$$
So, for $x \in \cN^\eps$, we have the following.
\begin{itemize}
\item[(a)] There are no queues of length $k+1$ or greater.
\item[(b)] For $1 \le j \le k$, the number of queues of length less than $j$ is $n(1-u_j(x))$, which lies between
$(1-5\eps) n (1-\la) (\la d)^{j-1}$ and $(1+5\eps) n (1-\la) (\la d)^{j-1}$.
\item[(c)] In particular, the number of queues of length less than $k$ is at most $(1+5\eps) n (1-\la) (\la d)^{k-1}$.  We shall work under
assumptions guaranteeing that $d^{k-1}(1-\la) \to 0$ as $n \to \infty$, so that the number of queues of length less than $k$ is $o(n)$, and
so the proportion of queues of length exactly~$k$ tends to~1 as $n \to \infty$.
\item[(d)] For $1\le j\le k-1$, the number of queues of length exactly $j$ is $n(u_j(x) - u_{j+1}(x))$.  Provided $\eps \la d \ge 2$, this
quantity lies between $(1-6\eps) n (1-\la) (\la d)^j$ and $(1+6\eps) n (1-\la) (\la d)^j$.
\end{itemize}

\begin{theorem} \label{thm.main}
For each $\eps > 0$, there exists $C$ such that the following holds for all sufficiently large $n$, and all $\la$ with
$$
\frac{C \log^2 n}{\sqrt n} \le 1-\la \le \frac{1}{C\log^2 n}.
$$
For each positive integer $d$ with $(1-\la)^{-1} > d \ge 4C k_{\la,d} \log^2 n$ and
\begin{equation} \label{eq:d}
\left( 2 (1-\la)^{-1} \log^2 n \right)^{1/k_{\la,d}} \le d \le \left( \frac{ (1-\la)^{-1} }{Ck_{\la,d}} \right)^{1/(k_{\la,d}-1)},
\end{equation}
a copy $(Y_t)$ of the $(n,d,\la)$-supermarket process in equilibrium satisfies
$$
\Pr\left( \exists t \in [0,e^{\frac14 \log^2 n}], Y_t \notin \cN^\eps(n,d,\la,k_{\la,d}) \right) \le e^{-\frac14 \log^2 n}.
$$
\end{theorem}

Observe that, from the definition of $k_{\la,d}$, we have $d^{k_{\la, d}-1} < (1-\la)^{-1} \le d^{k_{\la,d}}$.  We interpret
(\ref{eq:d}) as saying that these inequalities are true with something to spare.  In other words, (\ref{eq:d}) will hold whenever
$\log(1-\la)^{-1}/\log d$ is not too close to an integer.  If that is the case,
then the conclusion of the theorem implies that, for the $(n,d,\la)$-supermarket process in equilibrium, with high probability the
maximum queue length is $k_{\la,d}$ and most queues have exactly this length.

%\bigbreak

In fact, we shall prove the following result, which -- as we shall show -- implies Theorem~\ref{thm.main}.

\begin{theorem} \label{thm.technical}
Suppose the natural numbers $n$, $d$ and $k$, and the real numbers $\la$ and $\eps$ in $(0,1)$ satisfy: $k \ge 2$, $d^k(1-\la) \ge 2 \log^2 n$,
$$
\frac{1}{10} \ge \eps \ge \max \left\{ \frac{150 k}{\sqrt d}, 100 k (1-\la) d^{k-1}, \frac{10\sqrt 6 k \log n (1-\la)^{-1}}{\sqrt{n d}} \right\}.
$$
Then a copy $(Y_t)$ of the $(n,d,\la)$-supermarket process in equilibrium satisfies
$$
\Pr\left( \exists t \in [0,e^{\frac14 \log^2 n}], Y_t \notin \cN^\eps(n,d,\la,k) \right) \le e^{-\frac14 \log^2 n}.
$$
\end{theorem}

On the surface, this theorem may appear to apply for every value of $n$; however, the conditions above can only be satisfied if $n$ is at least
$10^{15}$.  This, as well as other consequences of the assumptions above, is shown in Lemma~\ref{lem.inequalities}.

\begin{proof} [Proof of Theorem~\ref{thm.main}.]
In order to prove Theorem~\ref{thm.main}, we may assume that $\eps \le 1/10$.
We shall show that, given $\eps \in (0,1/10]$, a suitable $C$ can be found so that, whenever $n$, $d$ and $\la$ satisfy the assumptions of
Theorem~\ref{thm.main}, then $(n,d,k_{\la,d},\la,\eps)$ satisfies all the conditions of Theorem~\ref{thm.technical}.

The conditions $d^{k_{\la,d}}(1-\la) \ge 2 \log^2 n$ and $k_{\la,d} \ge 2$ (which is equivalent to $(1-\la)^{-1} > d$) follow automatically from the
assumptions of Theorem~\ref{thm.main}.  We also have, directly from the second inequality in (\ref{eq:d}), that
$\eps \ge 100 k_{\la,d} (1-\la) d^{k_{\la,d}-1}$, provided we choose $C \ge 100 \eps^{-1}$.

To see the remaining conditions, we first note that $k_{\la,d} \le \log (1-\la)^{-1} \le \log n$, for $n$ sufficiently large: the first
inequality is from the definition of $k_{\la,d}$ and the second is from the lower bound on $1-\la$, for $n$ sufficiently large.
Now $\eps \sqrt d \ge 2\eps \sqrt{C k_{\la,d}} \log n \ge 2 \eps \sqrt C k_{\la,d}^{3/2}$, using the lower bound on $d$.  We now
deduce that $\eps \sqrt d \ge 150 k$, provided $\sqrt C \ge 75\eps^{-1}$.

Finally $\frac{(1-\la)^{-1}}{\sqrt n} \le \frac{1}{C \log^2 n}$, and $k_{\la,d} \le \log n$, so
$$
\frac{10\sqrt 6 k_{\la,d} \log n (1-\la)^{-1}}{\sqrt{n d}} \le \frac{1}{C \log^2 n} \log n \frac{10 \sqrt 6\log n}{\sqrt d}
= \frac{10\sqrt 6}{C \sqrt d} \le \eps,
$$
provided $C \ge 10\sqrt 6 \eps^{-1}$.

In summary, provided we choose $C \ge 75^2 \eps^{-2}$, all the conditions of Theorem~\ref{thm.technical} are satisfied.
\end{proof}

%\medbreak

The conditions $d^k(1-\la) \ge 2\log^2 n$ and $100k (1-\la) d^{k-1} \le \eps$ imply that $d\ge 200 \eps^{-1} k \log^2 n$.
Provided $(1-\la)^{-1}$ is large compared with $\log^2 n$, ``most'' values of $d$ above this minimum fall into one of the ranges between
$\displaystyle \left( 2 (1-\la)^{-1} \log^2 n \right)^{1/k}$ and $\displaystyle \left( \frac{ (1-\la)^{-1} }{Ck} \right)^{1/(k-1)}$,
with only small transitional ranges around $(1-\la)^{-1/k}$, for $k$ an integer, left uncovered by the result.

For values of $d$ in these transitional ranges, our results say nothing.  However, in these ranges, we can make use of a coupling result in~\cite{tur98}
(see also~\cite{gr00}).  For $d < d'$, there is a coupling of the $(n,d',\la)$-supermarket process and the $(n,d,\la)$-supermarket process such that, for all
times $t \ge 0$, and for each $j$, the number of customers in position at least $j$ in their queue at time $t$ in the $(n,d',\la)$-supermarket
process is at most the corresponding number in the $(n,d,\la)$-supermarket process, provided this is true at time~0.  This implies that, in
equilibrium, the number of customers in position at least $j$ in the $(n,d',\la)$-supermarket
process in equilibrium is stochastically at most the corresponding number in the $(n,d,\la)$-supermarket process.  For instance, if in the equilibrium
$(n,d,\la)$-supermarket process, the maximum queue length is at most $k$ with high probability, then the same is true for the equilibrium
$(n,d',\la)$-supermarket process.

Also, if $\la < \la'$, then there is a coupling of the $(n,d,\la)$- and $(n,d,\la')$-supermarket processes, so that at each time, each queue in
the $(n,d,\la)$-supermarket process is no longer than in the $(n,d,\la')$-supermarket process, provided this is true at time~0.  So, for instance,
if at a given time there are at least $m$ queues with length $k$ in the $(n,d,\la)$-supermarket process, then there are also at least $m$ queues with
length at least $k$ in the $(n,d,\la')$-supermarket process.

%\medbreak

To illustrate our results in some special cases, first suppose $\la = \la(n) = 1 - n^{-\alpha}$, and $d = n^\beta$, for some real numbers
$\alpha, \beta \in (0,1)$; then $k_{\la,d} = \lceil \alpha/\beta \rceil$.  It is easy to check that the conditions of Theorem~\ref{thm.technical} are
satisfied with $k = k_{\la,d}$ and $\eps = n^{-\delta}$, provided: $\alpha/\beta$ is not an integer, $\alpha > \beta$ (so that $k\ge 2$),
$2\alpha < 1 + \beta$ (so that $(1-\la)^{-1}/\sqrt{nd} = n^{\alpha -(1+\beta)/2}$ tends to zero), and
$$
\delta < \min \left( \frac{\beta}{2}, \left\lfloor \frac{\alpha}{\beta} \right\rfloor \beta - \alpha, \alpha - \frac{1+\beta}{2} \right),
$$
provided $n$ is sufficiently large.
The conclusions of Theorem~\ref{thm.technical} then hold, so in equilibrium the process spends almost all of the time in
$\cN^\eps = \cN^\eps(n,d,\la,k)$.  For $x \in \cN^\eps$, the maximum queue-length is $k$, and the number of queues of
length less than $k$ is given by $n (1-u_k(x))$, which lies between
$$
(1-5n^{-\delta}) n^{1-\alpha + \lfloor \alpha/\beta \rfloor \beta} \quad \mbox{ and } \quad (1+5n^{-\delta}) n^{1-\alpha + \lfloor \alpha/\beta \rfloor \beta}
$$
If $\alpha/\beta$ is equal to an integer $k \ge 2$, then we cannot expect as strong a conclusion to hold.  However, by comparing with the process for
slightly lower, and slightly higher, values of $\la$, we see that the maximum queue length in equilibrium is a.s.\ either $k$ or $k+1$, and that
most queues have length either $k$ or $k+1$.

%\medbreak

If we take $\la = 1 - n^{-\alpha}$ for some constant $\alpha \in (0,1/2)$, and $d$ tending to infinity more slowly than a power of $n$,
but with $d \ge C \log^3 n /\log\log n$ for a suitably large constant $C$, then again Theorem~\ref{thm.technical} applies, with
$k = k_{\la,d} = \lceil \log(1-\la)^{-1} / \log d \rceil = \lceil \alpha \log n / \log d \rceil$, provided that
$\alpha \log n/ \log d$ is not too close to an integer.  In this case,
$k_{\la,d}$ tends to infinity with $n$, and can be as large as
$\left(\alpha/3 - \delta\right) \log n /\log \log n$, for $\delta$ any positive constant and $n$ sufficiently large.
Specifically, suppose that $k = \gamma \log n / \log \log n$, with $\gamma < \alpha/3$.  Then it is straightforward to check that the conditions of
Theorem~\ref{thm.technical} are satisfied if
$d = (\log n)^{(\alpha + \theta \log \log n / \log n)/\gamma}$, which is equivalent to $d^k = n^\alpha (\log n)^\theta$,
for $2 < \theta < \alpha/\gamma - 1$.  Here we can take $\eps$ to be $(\log n)^{-\delta}$ for a small enough $\delta$.  In this range, even
though $k$ is fairly large, the conclusion is that, with high probability in equilibrium, almost all queues have length exactly $k$, and there are
no longer queues.
%The intermediate range of $d$, where we are unable to say whether the maximum queue length is $k$ or $k-1$, is roughly between
%$\displaystyle d = (\log n)^{\frac{1}{\gamma}(\alpha + (\alpha/\gamma-1) \log \log n / \log n)}$ and
%$\displaystyle d = (\log n)^{\frac{1}{\gamma}(\alpha + (\alpha/\gamma+2) \log \log n / \log n)}$.

%\medbreak

Next, suppose $\la = 1 - (\log n)^{-\alpha}$, for some fixed $\alpha > 2$.  For such a value of $\la$, Theorem~\ref{thm.technical} requires that
$d \ge 200 \eps^{-1} k \log^2 n \ge \log^2 n$, and that $d^{k-1} \le \eps (1-\la)^{-1} / 100 k \le (1-\la)^{-1} = (\log n)^\alpha$, which implies
that $\alpha / (k-1) < 2$, or $k < \frac12 \alpha + 1$.

In other words, for such values of $\la$, we only have results giving conditions under which the maximum queue length
$k = k_{\la,d}$ is a constant, with $2 \le k < \frac12 \alpha +1$.  If $d = (\log n)^\beta$, where
$(\alpha + 2)/k < \beta < \alpha/(k-1)$, then the conditions of Theorem~\ref{thm.technical} do apply for sufficiently
large $n$, where we may take $\eps = (\log n)^{-\delta}$ for a suitably small $\delta$.

If $\eps \le 1/10$ and $\la = 1 - \frac{1}{C \log^2 n}$, for $C > 160000 \eps^{-2}$, and $d = c \log^2 n$, for $\sqrt {2C} \le c \le \eps C / 200$,
then the conditions of Theorem~\ref{thm.technical} apply for $k=2$ and this value of $\eps$, for sufficiently large $n$.  This is the range of applicability
of our results with the slowest rate of convergence of $\la$ to 1.

As mentioned earlier, and explained in more detail in Section~\ref{sec.heuristics}, our results are in line with a more general hypothesis:
for a very wide range of parameter values, the maximum queue length of the $(n,d,\la)$-supermarket model in equilibrium is within~1 of the largest $k$
such that
$$
\la^{1 + d + \cdots + d^{k-1}} > \frac{1}{n}.
$$
This general hypothesis holds when $\la$ and $d$ are constants: see~\cite{lmcd06}.  It is also valid for the range where $\la$ is fixed and
$d \to \infty$: see~\cite{f11}.

%\medbreak

Another range not covered by Theorem~\ref{thm.main} is that where $d \ge (1-\la)^{-1}$, where one should expect the maximum queue length $k$ to be equal to~1.
For this range, our techniques can be used, but there are several places where we would need to pick out $k=1$ as a special case and treat it
separately.  Rather than do this, we refer the interested reader to the PhD thesis~\cite{f11} of Marianne Fairthorne, which contains a detailed
treatment of this case.  The authors, with Fairthorne, intend to write this result up for publication elsewhere.

%\medbreak

We also prove various rapid mixing results.
For $x \in \Z_+^n$, let $\cL(X^x_t)$ denote the law at time $t$ of the $(n,d,\la)$-supermarket process $(X_t^x)$ started in state $x$.  Also let $\Pi$
denote the stationary distribution of the $(n,d,\la)$-supermarket process.

\begin{theorem} \label{thm.rapid-mixing}
Suppose that $n$, $d$, $k$, $\la$ and $\eps = \frac{1}{60}$ satisfy the conditions of Theorem~\ref{thm.technical}.
Let $x$ be a queue-lengths vector in $\cN^\eps(n,d,\la,k)$.  Then, for $t \ge 0$,
$$
d_{TV}(\cL(X^x_t),\Pi) \le n \left( 2 e^{-\frac14 \log^2 n} + 4\exp \left( - \frac{t}{1600 k d^{k-1} n}\right) \right).
$$
\end{theorem}

In other words, for a copy of the process started in a state in $\cN^\eps(n,d,\la,k)$, with $\eps = \frac{1}{60}$, the mixing time is of order
$k d^{k-1} n \log n \le (1-\la)^{-1} n \log n \le n^2$ (see (\ref{ineq:7}) and (\ref{ineq:115})).  Formally, rapid mixing is often defined to be
mixing in $O(n\log n)$ steps, and this does not meet that criterion, but in fact this result is nearly best possible: we show that mixing,
starting from states in $\cN^\eps(n,d,\la,k)$, requires $\Omega(d^{k-1} n)$ steps.

From states not in $\cN^\eps$, we cannot expect to have rapid mixing in general.  For instance, suppose we start from a state $x$ with number
of customers $\|x\|_1 =g n$, where $g$ is much larger than $k_{\la,d}$.
The expected decrease in the number of customers at each step of the chain is at most $\frac{1-\la}{1+\la}$, so mixing takes
$\Omega( (1-\la)^{-1} g n)$ steps.  Similarly, if we start with one long queue, of length $\|x\|_\infty = \ell$ much greater than $k_{\la,d}$, then mixing
takes $\Omega(\ell n)$ steps, to allow time for the long queue to empty out.  We prove the following result, giving a nearly best-possible mixing time
for $(X_t^x)$ in terms of $\|x\|_1$ and $\|x\|_\infty$.

\begin{theorem} \label{thm.mixing2}
Suppose that $(n,d,\la,k,\frac{1}{60})$ satisfies the hypotheses of Theorem~\ref{thm.technical}, and let $x$ be any queue-lengths vector.
Let
$$
q = (6000 k n + 4320 \|x\|_1) (1-\la)^{-1} + 8 n \| x\|_\infty
$$
and suppose that $q \le \frac12 e^{\frac13 \log^2 n}$.  Then, for $t \ge 2q$, we have
$$
d_{TV}(\cL(X^x_t),\Pi) \le 2 \|x\|_1 \left( 2 e^{-\frac14 \log^2 n} + 4\exp \left( - \frac{t}{3200 k d^{k-1} n}\right) \right).
$$
\end{theorem}

%\bigbreak

The supermarket model is an instance of a model whose behaviour has been fully analysed even though there are an unbounded
number of variables that need to be tracked -- namely, the proportions $u_i(X_t)$.  While what we achieve in this paper is
similar to what is achieved by Luczak and McDiarmid in~\cite{lmcd06} for the case where $\la$ and $d$ are fixed as $n \to \infty$,
only some of the techniques of that paper can be used here, as we now explain.

The proofs in~\cite{lmcd06} rely on a coupling of copies of the supermarket process where the distance between coupled copies does not increase in time.
This coupling is, in particular, used to establish concentration of measure, over a long time period, for Lipschitz functions of the queue-lengths
vector; this result is valid for any values of $(n,d,\la)$, and in particular in our setting.  Fast coalescence of coupled copies, and hence rapid mixing,
is shown by comparing the behaviour of the $(n,d,\la)$-process with the $(n,1,\la)$-process, which is easy to analyse.
This then also implies concentration of measure for Lipschitz functions in equilibrium, and that the profile of the equilibrium
process is well concentrated around the fixed point $\pi$ of the equations (\ref{eq.diff-eq}).

The coupling from~\cite{lmcd06} also underlies the proofs in the present paper. However,
in our regime, comparisons with the $(n,1,\la)$-process are too crude. Thus we cannot show that the coupled copies
coalesce quickly enough, until we know something about the profiles of the copies, in particular that their maximum queue lengths are small.
Our approach is to investigate the equilibrium distribution first, as well as the time for a copy of the process from a fairly general starting state to
reach a ``good'' set of states in which the equilibrium copy spends most of its time.  Having done this, we then prove rapid mixing in a very similar
way to the proof in~\cite{lmcd06}.

To show anything about the equilibrium distribution, we would like to examine the trajectory of the vector $u(X_t)$, whose components are the
$u_i(X_t)$ for $i\ge 1$.  This seems difficult to do directly, but we perform a change of variables and analyse instead a collection of just
$k= k_{\la,d}$ functions $Q_1(X_t), \dots, Q_k(X_t)$.  These are linear functions of $u_1(X_t), \dots, u_k(X_t)$, with the property that the
drift of each $Q_j(X_t)$ can be written, approximately, in terms of $Q_j(X_t)$ and $Q_{j+1}(X_t)$ only.  Exceptionally, the drift of $Q_k(X_t)$
is written in terms of $Q_k(X_t)$ and $u_{k+1}(X_t)$.  The particular forms of the $Q_j$ are chosen by considering the Perron-Frobenius eigenvalues
of certain matrices $M_k$ derived from the drifts of the $u_j(x)$.  Making this change of variables allows us to consider one function $Q_j(X_t)$ at
a time, and show that each in turn drifts towards its equilibrium mean (which is derived from the fixed point $\pi$ of (\ref{eq.diff-eq})), and we are
thus able to prove enough about the trajectory of the $Q_j(X_t)$ to show that, starting from any reasonable state, with high probability the chain
soon enters a good set of states where, in particular, $u_{k+1}(X_t) = 0$, and so the maximum queue length is at most $k$.  We also show that, with
high probability, the chain remains in this good set of states for a long time, which implies that the equilibrium copy spends the vast majority of its
time in this set.  The argument from~\cite{lmcd06} about coalescence of coupled copies can be used to show rapid mixing from this good set of states.
The drift of the function $Q_k$ to its equilibrium is slower than that of any other $Q_j$, and its drift rate, $1/(\la d)^{k-1}n$, is approximately the
spectral gap of the Markov chain $(X_t)$, and hence determines the speed of mixing.

The structure of the paper is as follows.  In Section~\ref{sec.heuristics}, we expand on the discussion above, and motivate the definitions of
the functions $Q_j: \Z_+^n \to \R$, which are fundamental to the proof.
In Section~\ref{sec.prelim}, we give a number of results about the long-term behaviour of random walks with drifts, including
several variants on results from~\cite{lmcd06}.
%; we advise the reader to skim or skip this section on a first reading.
In Section~\ref{sec.coupling}, we describe the key coupling from~\cite{lmcd06}, and use it to prove some results about the maximum queue length
and number of customers.
In Section~\ref{sec.drifts}, we discuss in detail the drifts of the functions $Q_j$.
The proof of Theorem~\ref{thm.technical} starts in Section~\ref{sec.hit-and-exit}, where we show how to
derive a closely related result from a sequence of lemmas.  These lemmas are proved in Sections~\ref{sec.BCD}--\ref{sec.GH}.  In Section~\ref{sec.H-I-N}, we
complete the proof of Theorem~\ref{thm.technical}.  We prove our results on mixing times in Section~\ref{sec.mixing}.

\section{Heuristics} \label{sec.heuristics}

In this section, we set out the intuition behind our results and proofs.  As before, let $(Y_t)$ be an equilibrium copy of the $(n,d,\la)$-supermarket
process.  Guided by the results in~\cite{f11,lmcd06}, we start by supposing that, for each $i\ge 1$, $u_i(Y_t)$ is well-concentrated around its
expectation $u_i$, and seeing what that implies about the $u_i$.  We have
\begin{eqnarray} \label{eq.drift-u}
\Delta u_i(Y_t) &=& \E[u_i(Y_{t+1}) - u_i(Y_t) \mid Y_t] \\
&=& \frac{1}{n(1+\la)} [\la u_{i-1} (Y_t)^d - \la u_i (Y_t)^d - u_i (Y_t) + u_{i+1}(Y_t)] .
\end{eqnarray}
To see this, observe that, for $i\ge 1$, conditioned on $Y_t$, the probability that the event at time $t+1$ is an arrival to a queue
of length exactly $i-1$, increasing $u_i$ by $1/n$, is $\frac{\la}{1+\la} \left( u_{i-1}(Y_t)^d - u_i(Y_t)^d \right)$, while
the probability that the event is a departure from a queue of length exactly $i$, decreasing $u_i$ by $1/n$, is
$\frac{1}{1+\la} \left( u_i(Y_t) - u_{i+1}(Y_t) \right)$.  Note that $u_0$ is identically equal to~1.

Taking expectations on both sides, and setting them to 0, we see that, since $(Y_t)$ is in equilibrium,
\begin{eqnarray} \label{eq.equilibrium}
0 &=& \E[u_i(Y_{t+1}) - u_i(Y_t)] \\
&=& \frac{1}{n(1+\la)} \E [\la u_{i-1} (Y_t)^d - \la u_i (Y_t)^d - u_i (Y_t) + u_{i+1}(Y_t)] \\
&\simeq& \frac{1}{n(1+\la)} [\la u_{i-1}^d - \la u_i^d - u_i + u_{i+1}],
\end{eqnarray}
where the approximations $\E u_i(Y_t)^d \simeq u_i^d$ and $\E u_{i-1}(Y_t)^d \simeq u_{i-1}^d$ are justified because of our assumption that
$u_i(Y_t)$ and $u_{i-1}(Y_t)$ are  well-concentrated around their respective means $u_i$ and $u_{i-1}$.

The system of equations
\begin{eqnarray}
0 &=& \la \hat{u}_{i-1}^d - \la \hat{u}_i^d - \hat{u}_i + \hat{u}_{i+1} \quad (i=1,2,\dots) \label{eq:hat}\\
1 &=& \hat{u}_0 \notag
\end{eqnarray}
has a unique solution with $\hat{u}_i \to 0$ as $i \to \infty$, namely:
$$
\hat{u}_i = \la^{1+ \cdots + d^{i-1}} \quad (i=0,1,\dots).
$$
See~\cite{lmcd06} and the references therein for details.

By analogy with~\cite{lmcd06}, and motivated by (\ref{eq.equilibrium}), if the $u_i(Y_t)$ are well concentrated, we expect that
$u_i \approx \hat{u}_i$, for each $i$, and moreover that the values of $u_i(Y_t)$ remain close to the corresponding $\hat{u}_i$ for long periods of time.
In the regime of Theorem~\ref{thm.technical},
\begin{eqnarray*}
\log \hat{u}_i &=& \log (1-(1-\la)) (1+ \cdots + d^{i-1}) \\
&=& - (1-\la)d^{i-1} (1 + O(1-\la)) (1+ O(1/d)) \\
&=&  - (1-\la)d^{i-1} ( 1+ O(1/\log^2 n)),
\end{eqnarray*}
for each $i \ge 1$.  In particular, $\hat{u}_{k+1}$ is much smaller than $1/n$ -- recall that $d^k(1-\la) \ge 2 \log^2 n$.  One part of our goal is to
show that indeed, in equilibrium, for a long period of time there is no queue of length greater than~$k$.

On the other hand, for $i\le k$, our assumptions on $\lambda$ and $d$ imply that $\hat{u}_i$ is close to~1 -- recall that $d^{k-1}(1-\la) \le \eps/100k$.
This suggests that, in equilibrium, most queues have length exactly~$k$.  Moreover, $\hat{u}_i^d = 1-o(1)$ for $i<k$, so that
$1-\hat{u}_i^d \approx d (1-\hat{u}_i)$, whereas $\hat{u}_k^d = o(1)$.  We then obtain the following linear approximation to the equations (\ref{eq:hat}),
written in terms of variables $1-\tilde{u}_1, \dots, 1-\tilde{u}_k$:
\begin{eqnarray*}
0 &=& \la d (1-\tilde{u}_1) + (1-\tilde{u}_1) - (1-\tilde{u}_2), \\
0 &=& -\la d (1-\tilde{u}_{i-1}) + \la d (1-\tilde{u}_i) + (1-\tilde{u}_i) - (1-\tilde{u}_{i+1})\\
&&\mbox{} \hspace{3.2truein} (2\le i \le k-1),\\
0 &=& -\la d (1-\tilde{u}_{k-1}) + (1-\tilde{u}_k) - (1 -\la) .
\end{eqnarray*}
These linear equations have solution $\tilde u$ given by
$$
1-\tilde{u}_i = (1-\la) ( 1+ (\la d) + \cdots + (\la d)^{i-1} ),
$$
for $i=1,\dots, k$.  We then have the further approximation
$$
1-\tilde{u}_i \approx (1-\la) (\la d)^{i-1},
$$
for $i=1, \dots, k$.

Ideally, we would seek a single function of the $u_i(x)$, which is small when $u_i(x) \approx \tilde{u}_i$ for each $i$, and larger
otherwise, and which has a downward drift outside of a small neighbourhood of $\tilde u$: we could then analyse the trajectory of this
function to show that $(u_1(x), \dots, u_k(x))$ stays close to $\tilde u$ for a long period.  We have been unable to find such a function,
and indeed analysing the evolution of the $u_i(X_t)$ directly appears to be challenging.
Instead, we work with a sequence of functions $Q_j(x)$, $j=1, \dots, k$, each of the form $Q_j(x) = n \sum_{i=1}^j \gamma_{j,i} (1-u_i(x))$,
where the $\gamma_{j,i}$ are positive real coefficients.  This sequence of functions has the property that the drift of
each $Q_j(x)$ can be written (approximately) in terms of $Q_j(x)$ itself and $Q_{j+1}(x)$.

Let us see how these coefficients should be chosen, starting with the special case $j=k$, where we write $\beta_i$ for $\gamma_{k,i}$.
Consider a function of the form $Q_k(x) = n \sum_{i=1}^k \beta_i (1-u_i(x))$.  As in the argument leading to (\ref{eq.drift-u}), we have that the
drift of this function satisfies
\begin{eqnarray*}
\lefteqn{(1+\la) \Delta Q_k(x)} \\
&=&  - (1+\la) n \sum_{i=1}^k \beta_i \Delta u_i(x) \\
&=& - \sum_{i=1}^k \beta_i [\la u_{i-1}(x)^d - \la u_i(x)^d - u_i(x) + u_{i+1}(x)] \\
&=& \sum_{i=1}^k \beta_i [\la (1-u_{i-1}(x)^d) - \la (1-u_i(x)^d) - (1-u_i(x)) + (1-u_{i+1}(x))].
\end{eqnarray*}
Making the approximations $u_{k+1}(x) \simeq 0$, $u_k(x)^d \simeq 0$, and $1 - u_i(x)^d \simeq d(1-u_i(x))$ for $i=1,\dots, k-1$, and
rearranging, we arrive at
\begin{eqnarray*}
(1+\la) \Delta Q_k(x) &\simeq& \beta_k (1-\la) + (\beta_{k-1} - \beta_k) (1-u_k(x)) \\
&&\mbox{} + \sum_{i=1}^{k-1} [ \la d (\beta_{i+1} - \beta_i) - \beta_i + \beta_{i-1}] (1-u_i(x)).
\end{eqnarray*}
We set $\beta_0 =0$ for convenience of writing the above expression.  This calculation is done carefully, with precise inequalities,
in Lemma~\ref{lem.qk-drift} below.  We would like to choose the $\beta_i$ so that the vector
$$
\big( \la d (\beta_2 - \beta_1) - \beta_1 + \beta_0, \dots, \la d (\beta_k - \beta_{k-1}) - \beta_{k-1} + \beta_{k-2},
\beta_{k-1} - \beta_k \big)
$$
%$$
%\big( (\beta_{k-1} - \beta_k) , \la d (\beta_k - \beta_{k-1}) - \beta_{k-1} + \beta_{k-2}, \dots ,
%\la d (\beta_2 - \beta_1) - \beta_1 + \beta_0 \big)
%$$
is a (negative) multiple of $\big( \beta_1, \dots, \beta_{k-1}, \beta_k \big)$.  This would entail
$$
(1+\la) \Delta Q_k(x) \simeq \beta_k (1-\la) - \mu Q_k(x),
$$
for some (positive) $\mu$, which in turn would mean that $Q_k$ drifts towards a value of $\beta_k (1-\la) / \mu$, which should be
very close to $n (1-\la) (\la d)^{k-1}$ -- if $Q_k$ is above this value then it drifts down, whereas if $Q_k$ is below then it drifts up.
The point is that these drifts can be bounded below in magnitude, regardless of the precise values of the $u_i$ that go towards making up $Q_k$.
What we need is for $\big( \beta_1, \dots, \beta_{k-1}, \beta_k \big)$ to be a left eigenvector of the $k \times k$ matrix
$$
M_k = \begin{pmatrix}
-\la d - 1 & 1 & 0 & \cdots & 0 & 0 & 0 \\
\la d & -\la d-1 & 1 & \cdots & 0 & 0 & 0 \\
0 & \la d & -\la d-1 & \cdots & 0 & 0 & 0 \\
\vdots & \vdots & \vdots &\ddots & \vdots &\vdots & \vdots \\
0 & 0 & 0 & \cdots & -\la d -1 & 1 & 0 \\
0 & 0 & 0 & \cdots & \la d & -\la d -1 & 1 \\
0 & 0 & 0 & \cdots & 0 & \la d & -1
\end{pmatrix},
$$
or, equivalently, of the matrix
$$
M'_k = M_k +(\la d +1)I_k =
\begin{pmatrix}
0 & 1 & 0 & \cdots & 0 & 0 & 0 \\
\la d & 0 & 1 & \cdots & 0 & 0 & 0 \\
0 & \la d & 0 & \cdots & 0 & 0 & 0 \\
\vdots & \vdots & \vdots &\ddots & \vdots &\vdots & \vdots \\
0 & 0 & 0 & \cdots & 0 & 1 & 0 \\
0 & 0 & 0 & \cdots & \la d & 0 & 1 \\
0 & 0 & 0 & \cdots & 0 & \la d & \la d
\end{pmatrix}.
$$
The non-negative matrix $M'_k$ has a unique largest eigenvalue, with a positive left eigenvector.  By inspection, we see that this eigenvector is
close to the all-1 vector, with an eigenvalue close to $\la d + 1$, so that $M_k$ has largest eigenvalue very close to 0.
Recursion shows that a better approximation to the Perron-Frobenius eigenvector of $M'_k$ is
$\big( \beta_1, \dots, \beta_{k-1}, \beta_k \big)$, where
$$
\beta_i = 1 - \frac{1}{(\la d)^i} - \frac{(i-1)}{(\la d)^k},
$$
for $i=1, \dots, k$, and the largest eigenvalue of $M_k$ is very close to $-1/(\la d)^{k-1}$.  We shall see in Lemma~\ref{lem.qk-drift}
that this approximation is close enough for our purposes.

For $1\le j<k$, a similar analysis reveals that, if $Q_j(x) = n \sum_{i=1}^j \gamma_{j,i} (1-u_i)$, then
$$
(1+\la) \Delta Q_j(x) \simeq \sum_{i=1}^j (1-u_i(x)) \left[ \gamma_{j,i-1} + \la d \gamma_{j,i+1} - (\la d +1) \gamma_{j,i} \right] + (1-u_{j+1}(x)).
$$
(See the proof of Lemma~\ref{lem.qj-drift}.)
We think of $1-u_{j+1}(x)$ as an ``external'' term (which in practice will be very close to $Q_{j+1}(x)/n$), which will determine the
value towards which $Q_j$ drifts.  We would like the rest of the expression to be a negative multiple of $Q_j(x)$.  For this we need
$\big(\gamma_{j,1}, \dots, \gamma_{j,j}\big)$ to be a left eigenvector of the $j \times j$ matrix
$$
M_j = \begin{pmatrix}
-\la d-1 & 1 & 0 & \cdots & 0 & 0 & 0 \\
\la d & -\la d-1 & 1 & \cdots & 0 & 0 & 0 \\
0 & \la d & -\la d-1 & \cdots & 0 & 0 & 0 \\
\vdots & \vdots & \vdots &\ddots & \vdots &\vdots & \vdots \\
0 & 0 & 0 & \cdots & -\la d -1 & 1 & 0 \\
0 & 0 & 0 & \cdots & \la d & -\la d -1 & 1 \\
0 & 0 & 0 & \cdots & 0 & \la d & -\la d -1
\end{pmatrix},
$$
or, equivalently, of the matrix
$$
M'_j = M_j +(\la d+1)I_j =
\begin{pmatrix}
0 & \la d & 0 & \cdots & 0 & 0 & 0 \\
1 & 0 & \la d & \cdots & 0 & 0 & 0 \\
0 & 1 & 0 & \cdots & 0 & 0 & 0 \\
\vdots & \vdots & \vdots &\ddots & \vdots &\vdots & \vdots \\
0 & 0 & 0 & \cdots & 0 & \la d & 0 \\
0 & 0 & 0 & \cdots & 1 & 0 & \la d \\
0 & 0 & 0 & \cdots & 0 & 1 & 0
\end{pmatrix}.
$$
These matrices are tridiagonal Toeplitz matrices, and there is an exact formula for the eigenvalues and eigenvectors.  (See, for instance,
Example~7.2.5 in~\cite{meyer}.)  The Perron-Frobenius eigenvalue of $M'_j$ is $2 \sqrt{\la d} \cos \Big ( \frac{\pi}{j+1} \Big )$, with the left
eigenvector $\big( \gamma_{j,1}, \dots, \gamma_{j,j} \big)$ given by
$$
\gamma_{j,i}= (\la d)^{(j-i)/2} \frac{\sin \Big (\frac{i\pi}{j+1} \Big )}{\sin \Big ( \frac{j \pi}{j+1} \Big ) }.
$$
This means that the largest eigenvalue of $M_j$ is $-\la d +O(\sqrt{\la d})$, so that we obtain
$$
(1+\la) \Delta Q_j(x) \simeq -\la d \frac{Q_j(x)}{n} + \frac{Q_{j+1}(x)}{n} \quad (1\le j <k),
$$
meaning that $Q_j(x)$ will drift to a value close to $Q_{j+1}(x)/\la d$.  The choices of coefficients ensure that, if the $u_j(x)$ are
all near to $\tilde{u}_j$, then
$$
Q_j(x) \simeq n (1-\la) \sum_{i=1}^j \frac{\sin \Big (\frac{i\pi}{j+1} \Big )}{\sin \Big ( \frac{j \pi}{j+1} \Big ) } (\la d)^{i-1 + (j-i)/2},
$$
and the top term $i=j$ dominates the rest of the sum, provided $\la d$ is large, so $Q_j(x) \simeq (1-u_j(x))$: this is also true for $j=k$.
Thus the relationship $Q_j \simeq Q_{j+1}/\la d$ is as we would expect.

This means that, if $Q_{j+1}(X_t)$ remains in an interval around $\tilde {Q}_{j+1} := n (1-\la) (\la d)^j$ for a long time, then $Q_j(X_t)$ will
enter some interval around $\tilde{Q}_j$ within a short time, and stay there for a long time.  We can then conduct the analysis for each $Q_j$
in turn, starting with $j=k$, to show that indeed all the $Q_j(X_t)$ quickly become close to $\tilde {Q}_j$, and stay close for a long time.
This will then imply that the $u_j(X_t)$ all become and remain close to $\tilde{u}_j$.

A subsidiary application of this same technique forms another important step in the proofs (see the proof of Lemma~\ref{lem.cE}(1)).
If we do not assume that $u_{k+1}(x)$ is zero,
but instead build this term into our calculations, we obtain the approximation
$$
(1+\la) \Delta Q_k(x) \simeq (1-\la - u_{k+1}(x)) - \frac{Q_k(x)}{(\la d)^{k-1} n}.
$$
If $u_{k+1}(X_t)$ remains above $\eps (1-\la)$, for some $\eps > 0$, for a long time, this drift equation tells us that $Q_k$ drifts down into an interval
whose upper end is below the value $\tilde{Q}_k$, and then each of the $Q_j$ in turn drift down into intervals whose upper ends are below the
corresponding $\tilde{Q}_j$, and remain there.  For $j=1$, this means that the number of empty queues is at most $(1-\delta) (1-\la)n$,
for some positive $\delta$, for a long period of time; this results in a persistent drift down in the total number of customers (since the
departure rate is bounded below by $n - (1-\delta)(1-\la) n = \la n + \delta (1-\la) n$ while the arrival rate is $\la n$), and this is not
possible.

\section{Random Walks with Drifts} \label{sec.prelim}

In this section, we prove some general results about the long-term behaviour of real-valued functions of a Markov chain with bounds on the drift.

We start with two lemmas concerning random walks with a drift, both adapted from lemmas introduced in Luczak and McDiarmid~\cite{lmcd06}.
In each case, we assume that we have a sequence $(R_t)$ of
real-valued random variables on some probability space.  On some ``good'' event, the jumps $Z_t = R_t - R_{t-1}$ have magnitude at most~1,
and expectation at most $-v < 0$.  The first lemma shows that, on the good event, with high probability, such a random walk, started
at some value $r_0$, hits a lower value $r_1$ after not too many more than $(r_0-r_1)/v$ steps.

\begin{lemma}
\label{lem.hitting-time}
Let $\phi_0 \subseteq \phi_1 \subseteq \ldots \subseteq \phi_m$ be a filtration, and let $Z_1, \ldots, Z_m$ be random variables taking values in
$[-1,1]$ such that each $Z_i$ is $\phi_i$-measurable.  Let $E_0, E_1, \ldots, E_{m-1}$ be events where $E_i \in \phi_i$ for each $i$, and let
$E = \bigcap_{i=0}^{m-1} E_i$.  Fix $v \in (0,1)$, and let $r_0, r_1 \in \R$ be such that $r_0 > r_1$ and $vm \ge 2(r_0-r_1)$.  Set $R_0 = r_0$ and,
for each integer $t > 0$, let $R_t = R_0 + \sum_{i=1}^t Z_i$.

Suppose that, for each $i=1, \ldots, m$,
$$
\E (Z_i \mid \phi_{i-1}) \le -v \mbox{ on } E_{i-1} \cap \{R_{i-1} > r_1\}.
$$
Then
$$
\Pr (E \cap \{R_t > r_1 \quad \forall t \in \{1, \ldots, m\}\}) \le \exp \Big ( -\frac{v^2m}{8} \Big ).
$$
\end{lemma}

\begin{proof}
We first prove the lemma assuming the inequalities $\E (Z_i \mid \phi_{i-1}) \le - v$ hold almost surely, that is ignoring
the events $E_{i-1} \cap \{R_{i-1} > r_1\}$.  We shall then see how to incorporate these events.

We can couple the $Z_i$ with random variables $V_i$ taking values in $[-1,1]$ such that $\E [V_i \mid \phi_{i-1}] = -v$ for each $i$ and
$\Pr (Z_i \le V_i) = 1$ for each $i$: to do this, we define $V_i$ as follows.  Take a $U[0,1]$ random variable $W_i$, independent of the
$Z_i$ and the other $W_j$, and set $p_i= (v+1)/(1-\E [Z_i \mid \phi_{i-1}]) \in [0,1]$.  Now set
$V_i = Z_i \1_{\{W_i \le p_i \}} + \1_{\{W_i > p_i\}}$.
Note that indeed $|V_i| \le 1$ for each $i$ and $\Pr(V_i \ge Z_i)=1$.  Furthermore,
$$
\E [V_i \mid \phi_{i-1}] = p_i \E [Z_i \mid \phi_{i-1}] + (1-p_i) = -p_i (1- \E [Z_i \mid \phi_{i-1}]) +1 = -v.
$$

For each $t \ge 0$, let $S_t = \sum_{i=1}^t V_i$, so $S_t \ge \sum_{i=1}^t Z_i = R_t - R_0$, set $\mu_t = -vt = \E S_t$, and note that
$(S_t - \mu_t)$ is a martingale.  By the Hoeffding-Azuma inequality,
$\Pr (S_t \ge \mu_t + y) \le \exp (-y^2/2t)$.
Thus, if $a=r_0-r_1 \le \frac12 vm$,
\begin{eqnarray*}
\Pr (R_t > r_1 \quad \forall t \in \{1, \ldots, m\} \mid R_0=r_0) & \le & \Pr ( R_m - R_0 > -a ) \\
& \le & \Pr (S_m > -a)\\
& \le & \exp \Big (  - \frac{(vm-a)^2}{2m}  \Big )\\
& \le & \exp \Big ( -\frac{v^2 m}{8}\Big ).
\end{eqnarray*}

Now let us return to the full lemma as stated, with the events $E_i$. For each $i=0,1, \ldots, m-1$, let $F_i = E_i \cap \{R_i > r_1 \}$, and
for each $i=1, \ldots, m$, let $\tilde{Z}_i = Z_i \1_{F_{i-1}} - \1_{\overline{F}_i}$.  Let $\tilde{R}_0 = R_0$ and, for $t=1, \ldots, m$, let
$\tilde{R}_t = \tilde{R}_0 + \sum_{i=1}^t \tilde{Z}_i$. Then
$\E (\tilde{Z}_i \mid \phi_{i-1} ) \le - v$.
Hence, by what we have just proved applied to the $\tilde{Z}_i$,
\begin{eqnarray*}
\lefteqn{\Pr (E \cap \{R_t > r_1 \quad \forall t \in \{1, \ldots, m\}\}\mid R_0 = r_0)} \\
&=& \Pr (E \cap \{\tilde{R}_t > r_1 \quad \forall t \in \{1, \ldots, m\}\} \mid \tilde{R}_0 = r_0)\\
&\le& \Pr (\tilde{R}_t > r_1 \quad \forall t \in \{1, \ldots, m\})\\
&\le& \exp \Big ( -\frac{v^2 m}{8}\Big ),
\end{eqnarray*}
as required.
\end{proof}

The next lemma states that, if a discrete-time 1-dimensional random walk $(H_t)$,
starting at $h_0$ and making jumps of size at most~1, has negative drift whenever it lies in the interval $[h_0-b,h_0+a)$, then it is unlikely to
``cross against the drift'' and make its first exit from the interval at the upper end.

\begin{lemma}
\label{lem.crossing-drift}
Let $h_0$, $a$ and $b$ be positive real numbers. Let $v \in (0,1]$. Let $\phi_0 \subseteq \phi_1 \subseteq \ldots$ be a filtration, and let
$Z_1, Z_2, \ldots$ be random variables taking values in $[-1,1]$ such that each $Z_i$ is $\phi_i$-measurable.
Let $E_0, E_1, \ldots$ be events where $E_i \in \phi_i$ for each $i$.  Let $H_0$ be $\phi_0$-measurable, and, for each integer $t > 0$, let
$H_t = H_0 + \sum_{i=1}^t Z_i$. Assume for each $i=1, \ldots$,
$$
\E (Z_i \mid \phi_{i-1}) \le - v \mbox{ on } E_{i-1} \cap \{h_0 - b \le H_{i-1} < h_0 + a\}.
$$
Let
$$
T = \inf \{t \ge 1: H_t \in (-\infty,h_0-b) \cup [h_0+a,\infty)\}, \mbox{ and } E = \bigcap_{i=0}^{T-1} E_i.
$$
Then, on the event that $H_0 = h_0$,
$$
\Pr (E \cap \{H_T \ge h_0 + a\} \mid \phi_0 ) \le e^{-2va}.
$$
\end{lemma}

\begin{proof}
Let us first ignore the events $E_i$.

Note that, for any function $f$ convex on $[-1,1]$, we have
$$f(z) \le z \frac{f(1)-f(-1)}{2} + \frac{f(1) + f(-1)}{2}, \quad z \in [-1,1],$$
so, for each $i$,
\begin{eqnarray*}
\E [f(Z_i) \mid \phi_{i-1}] &\le& \E \Big [Z_i \frac{f(1)-f(-1)}{2} + \frac{f(1) + f(-1)}{2} \mid \phi_{i-1}\Big ]\\
&=& \frac{f(1)-f(-1)}{2} \E \big[ Z_i \mid \phi_{i-1} \big] + \frac{f(1) + f(-1)}{2}.
\end{eqnarray*}

Let $M_t = \big ( \frac{1+v}{1-v} \big )^{H_t}$, for each $t \in \Z_+$,
and note that $f(z) = \big (\frac{1+v}{1-v} \big )^z$ is convex on $[-1,1]$.
Then, for each $t \in \N$,
\begin{eqnarray*}
\lefteqn{\E [M_t \mid\ \phi_{t-1} ] } \\
& = & M_{t-1} \E \Big [ \Big (\frac{1+v}{1-v}\Big)^{Z_t} \mid  \phi_{t-1}\Big ]\\
& \le & M_{t-1} \Big( \Big(\frac{1+v}{2(1-v)} - \frac{1-v}{2(1+v)} \Big) \E [Z_i \mid \phi_{i-1} ]+
\frac{1+v}{2(1-v)} + \frac{1-v}{2(1+v)} \Big )\\
& \le & M_{t-1} \Big( -v \Big(\frac{1+v}{2(1-v)} - \frac{1-v}{2(1+v)} \Big) + \frac{1+v}{2(1-v)} + \frac{1-v}{2(1+v) }\Big )\\
& = & M_{t-1} \Big ( \frac{1-v^2}{2(1-v)} + \frac{1-v^2}{2(1+v)} \Big )\\
& = & M_{t-1},
\end{eqnarray*}
so $(M_t)$ is a supermartingale.  We deduce that, on the event that $H_0=h_0$,
$$
\E [M_t \mid \phi_0] \le \E [M_0 \mid \phi_0] = \Big (\frac{1+v}{1-v}\Big )^{h_0}.
$$
Thus, by the optional stopping theorem, on the event that $H_0 = h_0$,
%$\displaystyle \E [M_T \mid \phi_0] \le \Big (\frac{1+v}{1-v}\Big )^{h_0}$, and so
\begin{eqnarray*}
\lefteqn{ \Big ( \frac{1+v}{1-v} \Big)^{h_0} \,\ge\, \E [M_T \mid \phi_0] }\\
&\ge& \Pr (H_T \ge h_0+a \mid \phi_0) \Big ( \frac{1+v}{1-v} \Big )^{h_0+a} + \E \Big[\Big( \frac{1+v}{1-v}\Big )^{H_T} \1_{\{H_T < h_0-b\}} \mid \phi_0 \Big].
\end{eqnarray*}
%$$
%\le \Big ( \frac{1+v}{1-v} \Big)^{h_0},
%$$
Using the elementary inequality $(1+v)/(1-v) \ge e^{2v}$ for $0\le v < 1$, we deduce that, on the event that $H_0 = h_0$,
$$
\Pr (H_T \ge h_0+a \mid \phi_0) \le \Big ( \frac{1+v}{1-v} \Big )^{-a} \le e^{-2va},
$$
which yields the result in the case without the events $E_i$.

Now let us incorporate the events $E_i$, and consider the full lemma as stated.  For each $i=0,1, \ldots$, let
$F_i = E_i \cap \{h_0-b \le H_i < h_0+a \} \in \phi_i$ and
$\tilde{Z}_i =Z_i \1_{F_{i-1}} - \1_{\overline{F_{i-1}}}$.
Let $\tilde{H}_t$ and $\tilde{T}$ be defined in the obvious way.  Then $\tilde{Z}_i$ is $\phi_i$-measurable, $\tilde{Z}_i \in [-1,1]$ and
$\E[\tilde{Z}_i \mid \phi_{i-1} ] \le -v$.

On the event $\bigcap_{i=0}^{T-1} F_i = \bigcap_{i=0}^{T-1} E_i = E$,
we have $H_T = \tilde{H}_{\tilde T}$, and so, applying the first part of the proof to the $\tilde{Z}_i$,
$$
\Pr (E \cap \{H_T \ge h_0+a\} \mid \phi_0) \le \Pr (\tilde{H}_{\tilde{T}} \ge h_0+a \mid \phi_0) \le e^{-2va},
$$
on the event that $H_0 = h_0$, as required.
\end{proof}

We now use the two lemmas above to prove a result about real-valued functions of a Markov chain, that we shall use repeatedly in our proofs.

For a real-valued function $F$ defined on the set $\Z_+^n$ of queue-lengths vectors, a copy $(X_t)$ of the $(n,d,\la)$-supermarket process,
and $x \in \Z_+^n$, we define
$$
\Delta F(x) := \E[F(X_{t+1}) - F(X_t) \mid X_t = x],
$$
and call this the {\em drift} of $F$ (at $x$).  Similarly, we shall also use the notation $\Delta F(X_t)$ to denote the random variable
$\E[F(X_{t+1}) - F(X_t) \mid X_t]$.

\begin{lemma} \label{lem.drifts-down2}
Let $h$, $v$, $c$, $\rho \ge2 $, $m$ and $s$ be positive real numbers with $v m \ge 2 (c - h)$.
Let $(X_t)_{t\ge 0}$ be a discrete-time Markov process with state-space $\cX$, adapted to the filtration $(\phi_t)_{t\ge 0}$.
Let $\cS$ be a subset of $\cX$, and let $F$ be a real-valued function on $\cX$ such that, for all $x \in \cS$ with
$F(x) \ge h$,
$$
\Delta F(x) \le - v,
$$
and for all $t \ge 0$, $|F(X_{t+1}) - F(X_t)| \le 1$ a.s.
Let $T^*$ be any stopping time, and suppose that $F(X_{T^*}) \le c$ a.s.

Let
\begin{eqnarray*}
T_0 &=& \inf \{ t \ge T^*: X_t \notin \cS \}, \\
T_1 &=& \inf \{ t \ge T^*: F(X_t) \le h \}, \\
T_2 &=& \inf \{ t > T_1: F(X_t) \ge h + \rho \}.
\end{eqnarray*}

Then
\begin{itemize}
\item [(i)] $\displaystyle \Pr (T_1 \land T_0 > T^* + m) \le \exp( - v^2 m /8)$;
\item [(ii)] $\displaystyle \Pr (T_2 \le s \land T_0) \le s \exp( - \rho v)$.
\end{itemize}
\end{lemma}

When we use the lemma, $m$ will be much smaller than $s$, and moreover with high probability $T^*$ will be much smaller than $s$,
and also $\Pr(T_0 \le s)$ will be small.  In these circumstances, the lemma
allows us to conclude that, with high probability, $F(X_t)$ decreases from its value at $T^*$ (at most $c$) to below $h$ in at most
a further $m$ steps, and does not increase back above $h+\rho$ before time~$s$.  We shall sometimes use the conclusion of (ii) in the weaker
form $\Pr (T_2 \le s < T_0) \le s \exp( - \rho v)$.

\begin{proof}
We start by proving the lemma in the special case where the stopping time $T^*$ is equal to~0.

For~(i), we apply Lemma~\ref{lem.hitting-time}.  The filtration $\phi_0 \subseteq \phi_1 \subseteq \cdots \subseteq \phi_m$
will be the initial segment of the filtration $(\phi_t)_{t \ge 0}$.
For $t\ge 1$, we set $Z_t = F(X_t) - F(X_{t-1})$, so that $R_t := R_0 + \sum_{i=1}^t Z_i = F(X_t)$.  For $t \ge 0$, we set $E_t$ to be the
event that $T_0 > t$ (i.e., $X_i \in \cS$ for all $i$ with $0\le i \le t$), so
$E= \bigcap_{i=0}^{m-1} E_i$ is the event that $T_0 \ge m$.  We set $r_0 = F(X_0) \le c$, and $r_1 = h$.
We may assume that $r_0 > r_1$; otherwise $T_1 =0$ and there is nothing to prove.

On the event $E_{i-1} \cap \{ R_{i-1} > r_1\}$, we have $X_{i-1} \in \cS$ and $F(X_{i-1}) > r_1 = h$,
so $\E (Z_i \mid \phi_{i-1}) = \Delta F(X_{i-1}) \le - v$.
Thus, noting that $vm \ge 2(r_0-r_1)$ by our assumption on $m$, we see that the conditions of Lemma~\ref{lem.hitting-time} are satisfied.
The event that $R_t > r_1$ for all $t=1, \dots, m$ is the event that $T_1 > m$, so
$$
\Pr (T_1 \land T_0 > m) \le  \Pr (\{T_1 > m\} \cap \{ T_0 \ge m\}) \le e^{-v^2m/8},
$$
as required for~(i).

%\medbreak

We move on to (ii).
For each time $r \in \{ 0,\dots, s-1\}$, set
$$
T(r) = \min \{ t \ge 0 : F(X_{r+t}) \notin [h, h+\rho) \}.
$$
We say that $r$ is a {\em departure point} if: $T_1 \le r$, $F(X_r) \in [h,h+1)$, $F(X_{r+T(r)}) \ge h+\rho$,
and $r + T(r) \le s \land T_0$.  To say that $T_2 \le s \land T_0$ means that $F(X_t)$ crosses from its value, at most $h$, at time $T_1$, up to
a value at least $h+\rho$, taking steps of size at most~1, by time $s \land T_0$.  This is equivalent to saying that there is at least
one departure point $r \in [0,s)$.  Therefore
\begin{eqnarray*}
\lefteqn{\Pr (T_2 \le s \land T_0)}\\
&\le& \sum_{r=0}^{s-1} \Pr \Big( \{ T_1 \le r \} \cap \{F(X_r) \in [h, h+1) \} \\
&& \qquad \qquad \cap \{F(X_{r+T(r)}) \ge h+\rho\} \cap \{r+T(r) \le s \land T_0 \} \Big) \\
&=& \sum_{r=0}^{s-1} \E \Big[ \1_{\{T_1\le r\}} \1_{\{F(X_r) \in [h, h+1)\}}
\E \big[ \1_{\{F(X_{r+T(r)}) \ge h+\rho\}} \1_{\{r+T(r) \le s \land T_0\}} \mid \phi_r \big] \Big].
\end{eqnarray*}

Fix any $r \in [0,s)$.  We claim that, for any $h_0 \in [h,h+1)$, on the $\phi_r$-measurable event that $F(X_r) = h_0$,
the conditional expectation
$$
\E \big[ \1_{\{F(X_{r+T(r)}) \ge h+\rho\}} \1_{\{r+T(r) \le s \land T_0\}} \mid \phi_r \big]
$$
is at most
$e^{-\rho v}$.  This will imply that each term of the sum above is at most $e^{-\rho v}$, and so that
$\Pr (T_2 \le s \land T_0) \le s \exp( - \rho v)$, as required.

%\smallbreak

To prove the claim, we use Lemma~\ref{lem.crossing-drift}.
We consider the re-indexed process $(X'_t) = (X_{r+t})$; by the Markov property,
this is a Markov chain with the same transition probabilities as $(X_t)$, and initial state $X'_0 = X_r$ with $F(X'_0) = h_0$.
We set $\varphi'_i = \phi_{r+i}$ for each $i$, so that $(X'_i)$ is adapted to the filtration $(\varphi'_i)$.  Let $Z_i = F(X'_i)-F(X'_{i-1})$,
so that $|Z_i| \le 1$, and $Z_i$ is $\varphi'_i$-measurable, for each $i$.  Set $H_0 = h_0 = F(X'_0)$, so that
$H_t = H_0 + \sum_{i=1}^t Z_i = F(X'_t)$.  We set $a = h +\rho - h_0 \ge \rho -1 \ge \rho/2$, and
$b = h_0 - h$.
Thus the event $\{h_0 - b \le H_{i-1} < h_0 + a\}$ translates to $\{ h \le F(X'_{i-1}) < h+\rho \}$, and the event $\{ H_{T(r)} \ge h_0 + a\}$
translates to $F(X_{r+T(r)}) \ge h + \rho$.

For $i=0, \dots, s$, set $E_i = \{ r + i < s \land T_0 \}$,
noting that this event is in $\varphi'_i$, and that
$E := \bigcap_{i=0}^{T(r)-1} E_i = \{ r + T(r) \le s \land T_0 \}$.
On the event $E_{i-1} \cap \{h_0 - b \le H_{i-1} < h_0 + a\}$, we have $F(X'_{i-1}) \ge h$, and
$X'_{i-1} = X_{r+i-1} \in \cS$, and therefore
$\E(Z_i \mid \varphi'_{i-1}) = \Delta F(X'_{i-1}) \le - v$.
From Lemma~\ref{lem.crossing-drift}, we now conclude that, on the event $F(X_r)=h_0$,
$$
\Pr \Big( \{F(X_{r+T(r)}) \ge h+\rho\} \cap \{r+T(r) \le s \le T_0 \} \,\Big|\, \phi_r \Big)
$$
\begin{eqnarray*}
&\le& \Pr \Big( E \cap \{ H_{T(r)} \ge h_0 + a \} \Big) \\
&\le& e^{-2va} \\
&\le& e^{-\rho v},
\end{eqnarray*}
as required.
This completes the proof in the special case where $T^* = 0$.

%\smallbreak

We now proceed to the general case.  Suppose then that the hypotheses of the lemma are satisfied, with stopping time $T^*$.  We apply the result we
have just proved to the process $(X'_t) = (X_{T^*+t})$.  By the strong Markov property, $(X'_t)$ is
also a Markov process, adapted to the filtration $(\phi'_t)_{t\ge 0} = (\phi_{T^*+t})_{t\ge 0}$.  The condition that $F(X_{T^*}) \le c$ is
equivalent to $F(X'_0) \le c$.  Set:
\begin{eqnarray*}
T'_0 &=& \inf \{ t \ge 0 : X'_t \notin \cS \} = \inf \{ t \ge 0 : X_{T^*+t} \notin \cS \} = T_0 - T^*  \\
T'_1 &=& \inf \{ t \ge 0 : F(X'_t) \le h \} = \inf \{ t \ge 0 : F(X_{T^*+t}) \le h \} = T_1 - T^*  \\
T'_2 &=& \inf \{ t > T'_1 : F(X'_t) \ge h + \rho \} = \inf \{ t > T'_1 : F(X_{T^*+t}) \ge h + \rho \} = T_2 - T^*,
\end{eqnarray*}
and note that these are all stopping times with respect to the filtration $(\phi'_t)$.
The special case of the result (with $T^*=0$) now tells us that:
\begin{eqnarray*}
{\rm (i)} \quad \Pr (T_1 \land T_0 > T^* + m) &=& \Pr ((T^* + T'_1) \land (T^* + T'_0) > T^* + m) \\
&=& \Pr (T'_1 \land T'_0 > m) \\
&\le& \exp(-v^2 m /8); \\
{\rm (ii)} \quad \Pr (T_2 \le s \land T_0) &=& \Pr (T^* + T'_2 \le s \land (T^* + T'_0)) \\
&\le& \Pr (T^* + T'_2 \le (T^* + s) \land (T^* + T'_0)) \\
&=& \Pr (T'_2 \le s \land T'_0) \\
&\le& \exp(-\rho v).
\end{eqnarray*}
In both cases, these are the desired results.
\end{proof}

We shall also make use of a ``reversed'' version of Lemma~\ref{lem.drifts-down2} where $\Delta F(x) \ge v$ for all $x$ in some ``good'' set
$\cS$ with $F(x) \le h$.  The result and proof are practically identical to Lemma~\ref{lem.drifts-down2}, changing the directions of
inequalities where necessary, and using ``reversed'' versions of Lemmas~\ref{lem.hitting-time} and~\ref{lem.crossing-drift}.

%\bigbreak

The next lemma is a more precise version of Lemma~2.2 in~\cite{lmcd06}.  We omit the proof, which is exactly as in~\cite{lmcd06},
except that we track more carefully the values of the various constants appearing in that proof, and separate out the effects of the two occurrences
of $\delta$ in that theorem.

\begin{lemma}
\label{lem.return-time}
Let $(\phi_t)_{t\ge 0}$ be a filtration. Let $Z_1, Z_2, \ldots$ be $\{0,\pm 1\}$-valued random variables, where each $Z_i$ is $\phi_i$-measurable.
Let $S_0 \ge 0$ a.s., and for each positive integer $j$ let $S_j = S_0+ \sum_{i=1}^j Z_i$. Let $A_0, A_1, \ldots$ be events, where each $A_i$ is $\phi_i$-measurable.

Suppose that there is a positive integer $k_0$ and a constant $\delta$ with $0 < \delta < 1/2$ such that
$$
\Pr (Z_i = -1 \mid \phi_{i-1}) \ge \delta \quad \mbox{ on } A_{i-1} \cap \{S_{i-1} \in \{1, \ldots, k_0-1\}\}
$$
and
$$
\Pr (Z_i = -1 \mid \phi_{i-1}) \ge 3/4 \quad \mbox{ on } A_{i-1} \cap \{S_{i-1} \ge k_0\}.
$$
Then, for each positive integer $m$
$$
\Pr \Big( \bigcap_{i=1}^m \{S_i \not = 0\} \cap \bigcap_{i=0}^{m-1} A_i \Big )
\le \Pr (S_0 > \lfloor m/16 \rfloor) + 3\exp \left( - \frac{\delta^{k_0-1}}{200k_0} m \right).
$$
\end{lemma}

Several times we shall use the fact that, if $Z$ is a binomial or Poisson random variable with mean $\mu$, then for each $0 \le \epsilon \le 1$ we have
\begin{equation}
\label{eq.bin-lower}
\Pr (Z - \mu \le -\epsilon \mu) \le e^{-(1/2) \epsilon^2 \mu}.
\end{equation}
%and
%\begin{equation}
%\label{eq.bin-upper}
%\Pr (Z -\mu \ge \epsilon \mu) \le e^{-(1/3) \epsilon^2 \mu}.
%\end{equation}

%\bigbreak

\section{Coupling} \label{sec.coupling}

We now introduce a natural coupling of copies of the $(n,d,\la)$-supermarket process $(X_t^x)$ with different initial states
$x$.  The coupling is a natural adaptation to discrete time of that in~\cite{lmcd06}.

We describe the coupling in terms of three sequences of random variables.
There is an iid sequence ${\bf V} = (V_1, V_2, \ldots)$ of 0--1 random variables where each $V_i$ takes value 1 with probability
$\lambda/(1+\lambda)$; $V_i=1$ if and only if time $i$ is an arrival.  Corresponding to every time $i$ there is also an ordered list $D_i$ of
$d$ queue indices, each chosen uniformly at random with replacement.  Let ${\bf D} = (D_1, D_2, \ldots )$. Furthermore, corresponding to every
time $i$ there is a uniformly chosen queue index $\tilde{D}_i$.  Let ${\bf \tilde{D}} = (\tilde{D}_1, \tilde{D}_2, \ldots )$.  At time $i$,
$D_i$ will be used if $Z_i = 1$, and there will be an arrival to the first shortest queue in $D_i$; otherwise, there will be a departure
from the queue with index $\tilde{D}_i$, if that queue is currently non-empty.

Suppose that we are given a realisation $({\bf v},{\bf d},{\bf \tilde{d}})$ of $({\bf V},{\bf D},{\bf \tilde{D}})$.  For each possible initial
queue-lengths vector $x \in \Z_+^n$, this realisation yields a deterministic process $(x_t)$ with $x_0=x$: let us write
$x_t = s_t(x;{\bf v},{\bf d},{\bf \tilde{d}})$.  Then, for each $x \in \Z_+^n$, the process $s_t (x; {\bf V}, {\bf D}, {\bf \tilde{D}})$
has the distribution of the $(n,d,\la)$-supermarket process $X_t^x$ with initial state $x$.
In this way, we construct copies $(X_t^x)$ of the $(n,d,\la)$-supermarket process for each possible starting state $x$ on a single
probability space.  When we treat more than one such copy at the same time, we always work in this probability space, and we
let $\Pr(\cdot)$ denote the corresponding coupling measure.

%\medbreak

We shall use the following lemma, which is a discrete-time analogue of Lemma 2.3 in~\cite{lmcd06} and is proved in exactly the same way.

\begin{lemma}
\label{lem.coupling-distance}
Fix any triple ${\bf z}, {\bf d}, {\bf \tilde{d}}$ as above, and for each queue-lengths vector $x$ write $s_t (x)$ for
$s_t(x; {\bf z}, {\bf d}, {\bf \tilde{d}})$.  Then, for each $x,y \in \Z_+^n$, both $\|s_t (x) -s_t(y)\|_1$ and
$\|s_t (x) - s_t(y)\|_{\infty}$ are nonincreasing; and further, if $0 \le t < t'$ and $s_t(x) \le s_t (y)$, then $s_{t'}(x) \le s_{t'}(y)$.
\end{lemma}

%We shall use Lemma~\ref{lem.coupling-distance} immediately, and again later.

For a queue-lengths vector $x$, let $\| x \|_\infty = \max x(i)$ denote the maximum length of a queue in $x$, and $\| x \|_1 = \sum_{i=1}^n x(i)$
denote the total number of customers.  Given positive real numbers $\ell$ and $g$, we set
\begin{eqnarray*}
\cA_0(\ell,g) &=& \{ x : \| x \|_\infty \le \ell \mbox{ and } \| x \|_1 \le gn \}; \\
\cA_1(\ell,g) &=& \{ x : \| x \|_\infty \le 3\ell \mbox{ and } \| x \|_1 \le 3gn \}.
\end{eqnarray*}
We also set
$$
\ell^* = \log^2 n (1-\la)^{-1}, \, g^* =\ 2 (1-\la)^{-1}, \,
\cA_0^* = \cA_0(\ell^*, g^*), \,
%= \{ x : \| x \|_\infty \le \log^2 n (1-\la)^{-1} \mbox{ and } \| x \|_1 \le 2n (1-\la)^{-1} \}; \\
\cA_1^* = \cA_1(\ell^*, g^*).
% = \{ x : \| x \|_\infty \le 3\log^2 n (1-\la)^{-1} \mbox{ and } \| x \|_1 \le 6n (1-\la)^{-1} \}.
$$

The next result tells us that the $(n,d,\la)$-supermarket process $(Y_t)$ in equilibrium is very unlikely to
be outside the set $\cA_0^*$, for any $d$.  This is accomplished by proving the result for $d=1$, when
the process is easy to analyse explicitly, and then using coupling in $d$ to deduce the result for all $d$.
Of course, the result is actually extremely weak for all $d>1$, and later we shall show a much stronger
result whenever the various parameters of the model satisfy the conditions of Theorem~\ref{thm.technical};
the importance of the lemma below is that it gets us started and enables us to say {\em something}
about where the equilibrium of the process lives.

\begin{lemma}
\label{lem.compd1}
Let $(Y_t)$ be a copy of the $(n,d,\la)$-supermarket process in equilibrium.  Then
$\Pr(Y_t \notin \cA_0^*) \le 2 n e^{-\log^2 n}$.
\end{lemma}

\begin{proof}
Let $\tilde{Y}$ denote a stationary copy of the $(n,1,\la)$-supermarket process, in which each arriving customer joins a uniform random queue.
Then the queue lengths ${\tilde Y}_t(j)$ are independent geometric random variables with mean $\la/(1-\la)$, where
$\Pr ({\tilde Y}_t(j) = k) = (1- \la) \la^k$ for $k=0,1,2, \ldots$. Therefore,
$\Pr (\| \tilde{Y}_t \|_{\infty} \ge k )  \le  n \la^k$,
and also it can easily be checked that
$\Pr \left( \|\tilde{Y}_t \|_1 \ge 2 n (1-\la)^{-1} \right) \le  e^{-n/4}$.

As mentioned in Section~\ref{Sintro}, there is a coupling between supermarket processes with different values of $d$, which can be used to show that
the equilibrium copy $(Y_t)$ of the $(n,d,\la)$-supermarket process, for any $d$, also satisfies
$\Pr \left( \| Y_t \|_1 \ge 2 n (1-\la)^{-1} \right) \le  e^{-n/4}$ and
$\Pr (\|Y_t \|_{\infty} \ge \log^2 n (1-\la)^{-1} ) \le n \la^{\log^2 n (1-\la)^{-1}} \le n e^{-\log^2 n}$,
as required.
\end{proof}

Next we prove a very crude concentration of measure result: if the process $(Y_t)$ in equilibrium is concentrated inside
some set $\cA_0(\ell,g)$, and we start a copy $(X_t^x)$ of the process at a state $x \in \cA_0(\ell,g)$, then the process $(X_t^x$) is
unlikely to leave the larger set $\cA_1(\ell,g)$ over a long period of time.
%Again, we shall later prove much stronger results
%for the regime of interest to us, and again the importance of the result below is to get us started.

\begin{lemma}
\label{lem.l1-inf}
Let $\ell$ and $g$ be natural numbers and $x$ a queue-lengths vector in $\cA_0(\ell,g)$.
Then for any natural number $s$,
$$
\Pr ( \exists t \in [0,s],\, X_t^x \notin \cA_1(\ell,g) ) \le \Pr ( \exists t \in [0,s],\, Y_t \notin \cA_0(\ell,g) ).
$$
\end{lemma}

\begin{proof}
By Lemma~\ref{lem.coupling-distance}, we can couple $(X^x_t)$ and $(Y_t)$ in such a way that
$\|X^x_t - Y_t\|_1$ and $\|X^x_t - Y_t\|_{\infty}$ are both non-increasing, and hence that, for each $t \ge 0$,
\begin{eqnarray*}
\|X^x_t\|_1 & \le & \| X^x_t - Y_t\|_1 + \| Y_t \|_1 \\
&\le& \|x-Y_0\|_1 + \|Y_t\|_1 \\
&\le& \|x\|_1 + \| Y_0\|_1 + \|Y_t\|_1 \\
&\le& gn + \| Y_0\|_1 + \|Y_t\|_1,
\end{eqnarray*}
and similarly
$$
\|X^x_t\|_{\infty} \le \ell + \|Y_0\|_{\infty} + \|Y_t\|_{\infty}.
$$

We deduce that, for each $t \ge 0$,
\begin{eqnarray*}
\{ X_t^x \notin \cA_1(\ell,g) \} &=& \{ \| X_t^x \|_1 > 3gn \} \cup \{ \| X_t^x \|_\infty > \ell \} \\
&\subseteq& \{ \| Y_0 \|_1 > gn \} \cup \{ \| Y_t\|_1 > gn \} \\
&&\mbox{} \cup \{ \|Y_0\|_\infty > \ell \} \cup \{ \|Y_t\|_\infty > \ell \} \\
&=& \{ Y_0 \notin \cA_0(\ell,g) \} \cup \{ Y_t \notin \cA_0(\ell,g) \}.
\end{eqnarray*}
The result now follows immediately.
\end{proof}

We shall use Lemma~\ref{lem.l1-inf} later for general values of $\ell$ and $g$, but for now we note
the following immediate consequence of the previous two lemmas.

\begin{lemma} \label{lem.cA}
Let $x$ be any queue-lengths vector in $\cA_0^*$, and let $T_\cA^\dagger = \inf \{ t : X_t^x \notin \cA_1^* \}$.
Then, for $n \ge 2000$,
$$
\Pr (T_\cA^\dagger \le e^{\frac13 \log^2 n}) \le e^{-\frac12 \log^2 n}.
$$
\end{lemma}

\begin{proof}
The probability in question is $\Pr (\exists t \in [0,e^{\frac13 \log^2 n}],\, X^x_t \notin \cA_1^* )$ which,
by Lemma~\ref{lem.l1-inf} and Lemma~\ref{lem.compd1}, is at most
$$
\Pr ( \exists t \in [0,e^{\frac13 \log^2 n}],\, Y_t \notin \cA_0^* ) \le
(e^{\frac 13\log^2 n} +1) \Pr(Y_t \notin \cA_0^*) \le 3n e^{-\frac23 \log^2 n},
$$
which, for $n \ge 2000$, is at most $e^{-\frac 12 \log^2 n}$, as required.
\end{proof}

\section{Functions and Drifts} \label{sec.drifts}

We now start the detailed proofs.  The results in this section will be used in the course of the proof of Theorem~\ref{thm.technical}, and we could
assume that all the conditions of Theorem~\ref{thm.technical} hold; however, for this section
all that is necessary is that $k \ge 2$ and $\la d \ge 4$.

As explained in Section~\ref{sec.heuristics}, we will consider a sequence of functions $Q_k$, $Q_{k-1}$, \dots, $Q_1$ defined on the set
$\Z_+^n$ of queue-lengths vectors.  We now give precise definitions of these functions, along with another function $P_{k-1}$, and
derive some of their properties.

As in Section~\ref{sec.heuristics}, let $Q_k$ be the function defined on the set $\Z_+^n$ of all queue-lengths vectors by
$$
Q_k (x) = n\sum_{i=1}^k \beta_i (1-u_i(x)),
% = \sum_{i=1}^k \beta_i (n - \ell_i (x)),
$$
where, for $i=1, \ldots, k$,
$$
\beta_i = 1 - \frac{1}{(\la d)^i} - \frac{i-1}{(\la d)^k}.
$$
It is also convenient to set $\beta_0 = 0$.
Evidently $\beta_i < 1$ for each $i$, an inequality we shall use freely in future.
We also note that $\beta_i$ is increasing in~$i$, and that $\beta_k = 1 - k(\la d)^{-k}$.

%\medbreak

Let
$$
P_{k-1} (x) = n \sum_{i=1}^{k-1} (1-u_i(x)).
$$

Also, for $j=1, \ldots, k-1$, we let
$$
Q_j(x) = n \sum_{i=1}^j \gamma_{j,i} (1-u_i(x)),
$$
where the coefficients $\gamma_{j,i}$ are given by
$$
\gamma_{j,i}= (\la d)^{(j-i)/2} \frac{\sin \Big (\frac{i\pi}{j+1} \Big )}{\sin \Big ( \frac{j \pi}{j+1} \Big ) }.
$$
Consistent with the expression above, we also define $\gamma_{j,0} = \gamma_{j,j+1} = 0$.
It can easily be checked that, for each $i=1, \ldots, j-1$, and for each $j=1, \ldots, k-1$,
$$
\la d \gamma_{j,i+1} +  \gamma_{j,i-1} = 2 \sqrt{\la d} \cos \Big ( \frac{\pi}{j+1} \Big )\gamma_{j,i}.
$$
This is equivalent to saying that the $\gamma_{j,i}$ form eigenvectors of the tridiagonal Toeplitz matrices $M_j$ given in
Section~\ref{sec.heuristics}.

We will need some bounds on the sizes of the $Q_j(x)$.  Observe that $\gamma_{j,j} =1$ for each $j$, while generally we have
$$
1 \le \frac{\sin(i\pi/(j+1))}{\sin(j\pi/(j+1))} = \frac{\sin(i\pi/(j+1))}{\sin(\pi/(j+1))} \le i,
$$
since the sine function is concave on $[0,\pi]$.  Thus
\begin{equation} \label{eq:gamma-bounds}
(\la d)^{(j-i)/2} \le \gamma_{j,i} \le i (\la d)^{(j-i)/2},
\end{equation}
and therefore
\begin{equation}
Q_j(x) \le n \sum_{i=1}^j i (\la d)^{(j-i)/2} \le 2n(\la d)^{(j-1)/2}, \label{eq:Qjbound}
\end{equation}
provided $\la d \ge 4$.  We also note at this point that changing one component $x(\ell)$  of $x$ by $\pm1$ changes $Q_j(x)$ by at most
$\gamma_{j,1} = (\la d)^{(j-1)/2}$.

It can readily be checked that, for $j \ge 1$, the function
$$
f(i) = \sin\left( \frac{i\pi}{j+2} \right) \big/ \sin\left( \frac{i\pi}{j+1} \right)
$$
is increasing over the range $[1,j]$, and so we have, for $1\le i \le j \le k-2$:
\begin{eqnarray*}
\frac{\gamma_{j+1,i}}{\gamma_{j,i}} &=& \sqrt{\la d} \frac{\sin(i\pi/(j+2)) \sin(\pi/(j+1))}{\sin(i\pi/(j+1)) \sin(\pi/(j+2))} \\
&\le& \sqrt{\la d} \frac{\sin(j\pi/(j+2)) \sin(\pi/(j+1))}{\sin(j\pi/(j+1)) \sin(\pi/(j+2))} \\
&=& \sqrt{\la d} \frac{\sin(2\pi/(j+2))}{\sin(\pi/(j+2))} \\
&\le& 2\sqrt{\la d}.
\end{eqnarray*}
A consequence is that, for $j=1, \dots, k-2$, and any $x \in \Z_+^n$,
\begin{eqnarray}
\frac{Q_{j+1}(x)}{n} &=& (1-u_{j+1}(x)) + \sum_{i=1}^j \gamma_{j+1,i} (1-u_i(x)) \notag \\
&\le& (1-u_{j+1}(x)) + \sum_{i=1}^j 2 \sqrt{\la d} \gamma_{j,i} (1-u_i(x)) \notag \\
&\le& (1-u_{j+1}(x)) + 2 \sqrt{\la d} \frac{Q_j(x)}{n}. \label{eq:Qj+1}
\end{eqnarray}
For $j=k-1$, we have the stronger inequality that, for any $x \in \Z_+^n$,
\begin{equation}
\label{eq:Qk}
\frac{Q_k(x)}{n} \le \sum_{i=1}^k (1-u_i(x)) \le (1-u_k(x)) + \frac{Q_{k-1}(x)}{n}.
\end{equation}

%\bigbreak

We now prove the following result about the drift of the function $Q_k(x)$: roughly speaking, we wish to
show that it is approximately equal to
$$
\frac{1}{1+\la} \left(1-\la -u_{k+1}(x) - \frac{1}{(\la d)^{k-1}} \frac{Q_k(x)}{n}\right).
$$
%We shall use this result several times in what follows.

\begin{lemma} \label{lem.qk-drift}
For any state $x \in \Z_+^n$,
\begin{eqnarray*}
(1+\la) \Delta Q_k(x) &\le& \beta_k \big( (1-\la) - u_{k+1} (x) + \la \exp( - d Q_k(x)/ k n)  \big) \\
&&\mbox{} - \frac{1}{(\la d)^{k-1}} \frac{Q_k(x)}{n} \left(1-\frac{2}{\la d}\right), \\
(1+\la) \Delta Q_k(x) &\ge & \beta_k \left( (1-\la) - u_{k+1} (x) \right) - \frac{1}{(\la d)^{k-1}} \frac{Q_k(x)}{n} \\
&&\mbox{} - \left( \frac {Q_{k-1}(x)}{n} \right)^2 \frac{1}{(\la d)^{k-3}}.
\end{eqnarray*}
\end{lemma}

\begin{proof}
As in (\ref{eq.drift-u}), we have that, for $i =1, \dots, k$,
$$
\Delta u_i(x) = \frac{1}{n(1+\la)} \left(\la u_{i-1} (x)^d - \la u_i (x)^d - u_i (x) + u_{i+1}(x) \right).
$$
and that $u_0$ is identically equal to~1.  We deduce that
\begin{eqnarray*}
\Delta Q_k(x) & = & - n \sum_{i=1}^k \beta_i \Delta u_i(x) \\
& = & \frac{1}{1 + \lambda } \sum_{i=1}^k \beta_i
\Big (-\lambda u_{i-1} (x)^d + \lambda u_i (x)^d + u_i (x) - u_{i+1}(x) \Big ).
\end{eqnarray*}

We rearrange the formula above as follows:
\begin{eqnarray*}
\lefteqn{(1+ \lambda ) \Delta Q_k(x)} \\
& = & \beta_k \left( (1-\lambda ) + \lambda u_k(x)^d - u_{k+1} (x) + \lambda (1-u_{k-1} (x)^d) - (1-u_k(x))\right) \\
&&\mbox{} + \sum_{i=1}^{k-1} \beta_i \Big ( \lambda (1-u_{i-1}(x)^d) - \lambda (1-u_i(x)^d) - (1 - u_i (x)) + (1-u_{i+1}(x)) \Big )\\
& = &  \beta_k \left( (1-\lambda ) + \lambda u_k(x)^d - u_{k+1} (x) \right) \\
&&\mbox{} + \lambda \sum_{i=1}^{k-1} (\beta_{i+1} - \beta_i) (1-u_i(x)^d )
- \sum_{i=1}^k (\beta_i - \beta_{i-1}) (1-u_i(x)).
\end{eqnarray*}
Here we have used the facts that $\beta_0=0$ and $1-u_0(x) = 0$.

Now, for $1\le i \le k$, we have $1-u_k(x) \ge 1-u_i(x)$ and $\beta_k \ge \beta_i$, and so
$\beta_k (1-u_k(x)) \ge Q_k(x) / kn$, and hence
$$
0 \le u_k(x)^d \le \left( 1 - \frac { Q_k(x) }{kn} \right)^d \le \exp( - d Q_k(x)/ k n).
$$

In order to estimate the terms constituting the two sums, we note the inequalities
$$
d(1-u) - \binom{d}{2} (1-u)^2 \le 1-u^d \le d(1-u).
$$
To obtain our upper bound on $\Delta Q_k(x)$, we apply the inequality $1-u_i(x)^d \le d (1-u_i(x))$ for each $i=1, \dots, k-1$.
Since also
$$
\beta_{i+1}-\beta_i = \frac{1}{(\la d)^i} - \frac{1}{(\la d)^{i+1}} - \frac{1}{(\la d)^k}> 0,
$$
for $i=0, \dots, k-1$, we have
\begin{eqnarray*}
\lefteqn{\lambda \sum_{i=1}^{k-1} (\beta_{i+1} - \beta_i) (1-u_i(x)^d )
- \sum_{i=1}^k (\beta_i - \beta_{i-1}) (1-u_i(x))}\\
&\le & \la d \sum_{i=1}^{k-1} (\beta_{i+1} - \beta_i) (1-u_i(x)) - \sum_{i=1}^k (\beta_i - \beta_{i-1}) (1-u_i(x))\\
&=& - \left[ \frac{1}{(\la d)^{k-1}} - \frac{2}{(\la d)^k} \right] (1-u_k(x)) \\
&&\mbox{} + \sum_{i=1}^{k-1} \left[ \frac{\la d}{(\la d)^i} - \frac{\la d}{(\la d)^{i+1}} - \frac{\la d}{(\la d)^k}
- \frac{1}{(\la d)^{i-1}} + \frac{1}{(\la d)^i} + \frac{1}{(\la d)^k} \right] (1-u_i(x)) \\
&=& - \frac{1}{(\la d)^{k-1}} \left[ \left(1-\frac{2}{\la d}\right) (1-u_k(x)) +
\sum_{i=1}^{k-1} \left( 1 - \frac{1}{\la d} \right) (1-u_i(x)) \right] \\
&\le& -\frac{1}{(\la d)^{k-1}} \frac{Q_k(x)}{n} \left(1-\frac{2}{\la d}\right).
\end{eqnarray*}
This establishes the required upper bound on $(1+\la) \Delta Q_k(x)$.  The calculation works because
the $\beta_i$ are the entries of a good approximation to the dominant eigenvector of the matrix $M_k$ defined in Section~\ref{sec.heuristics}.

%\medbreak

For the lower bound, the previous calculation, and the bound $1-u_i(x)^d \ge d(1-u) -\binom{d}{2} (1-u)^2$, lead us to
\begin{eqnarray*}
\lefteqn{\lambda \sum_{i=1}^{k-1} (\beta_{i+1} - \beta_i) (1-u_i(x)^d )
- \sum_{i=1}^k (\beta_i - \beta_{i-1}) (1-u_i(x))}\\
&\ge & - \la \binom{d}{2} \sum_{i=1}^{k-1} (\beta_{i+1} - \beta_i) (1-u_i(x))^2 \\
&&\mbox{} - \frac{1}{(\la d)^{k-1}} \left[ \left(1-\frac{2}{\la d}\right) (1-u_k(x)) +
\sum_{i=1}^{k-1} \left( 1 - \frac{1}{\la d} \right) (1-u_i(x)) \right] \\
&\ge& - \la \binom{d}{2} \sum_{i=1}^{k-1} (\beta_{i+1} - \beta_i) (1-u_i(x))^2  -\frac{1}{(\la d)^{k-1}} \frac{Q_k(x)}{n}.
\end{eqnarray*}
Here we used the fact that $1-1/(\la d) \le \beta_i$ for each $i$.

It remains to show that
$$
\la \binom{d}{2} \sum_{i=1}^{k-1} (\beta_{i+1} - \beta_i) (1-u_i(x))^2 \le \left( \frac {Q_{k-1}(x)}{n} \right)^2 \frac{1}{(\la d)^{k-3}}.
$$
We observe that
\begin{eqnarray*}
\left( \frac {Q_{k-1}(x)}{n} \right)^2 &=& \left( \sum_{i=1}^{k-1} (\la d)^{(k-1-i)/2}
\frac{\sin \Big (\frac{i\pi}{k} \Big )}{\sin \Big ( \frac{(k-1) \pi}{k} \Big )} (1-u_i(x)) \right)^2 \\
&\ge& \sum_{i=1}^{k-1} (\la d)^{k-1-i} (1-u_i(x))^2 \\
&\ge& (\la d)^{k-1} \sum_{i=1}^{k-1} (\beta_{i+1} - \beta_i) (1-u_i(x))^2,
\end{eqnarray*}
which implies the required inequality.
\end{proof}

%\bigbreak

We prove a similar result for the functions $Q_j(x)$, $1\le j\le k-1$.  Ideally, the drift
bounds would be expressed in terms of $Q_j(x)$ itself and $Q_{j+1}(x)$: however, there is a complication.  In the upper bound,
there appears a term which can be bounded above by $\la \binom{d}{2} \sum_{i=1}^j \gamma_{j,i} (1-u_i(x))^2$, and we would like to
show that this is small compared with $\la d \sum_{i=1}^j \gamma_{j,i} (1-u_i(x))$.  This is true if $1-u_j(x) \ll 1/d$,
but in general we cannot assume this.  We bound this term above, very crudely, by
$$
\la \binom{d}{2} \left( \sum_{i=1}^{k-1} (1- u_i(x)) \right) \left( \sum_{i=1}^j \gamma_{j,i} (1-u_i(x)) \right)
= \la \binom{d}{2} \frac{P_{k-1}(x)Q_j(x)}{n^2};
$$
we use the function $P_{k-1}$ here because its drifts are relatively easy to handle.

\begin{lemma} \label{lem.qj-drift}
Fix $j$ with $1\le j \le k-1$. For any state $x \in \Z_+^n$, we have
\begin{eqnarray*}
(1+\la) \Delta Q_j(x) &\le&  - \la d \frac{Q_j(x)}{n} \left( 1 - \frac{2}{\sqrt{\la d}} - \frac{d P_{k-1}(x)}{n} \right) + \frac{Q_{j+1}(x)}{n}, \\
(1+\la) \Delta Q_j(x) &\ge & - \la d \frac{Q_j(x)}{n} \left( 1 + \frac{2}{\sqrt{\la d}} \right) + \frac{Q_{j+1}(x)}{n}.
\end{eqnarray*}
\end{lemma}

\begin{proof}
We begin by calculating
\begin{eqnarray*}
(1+\la) \Delta Q_j(x) &=& \sum_{i=1}^j \gamma_{j,i} \Big( -\la u_{i-1}(x)^d + \la u_i (x)^d + u_i(x) - u_{i+1}(x) \Big)\\
& = & \sum_{i=1}^j \gamma_{j,i} \Big ( \la (1-u_{i-1}(x)^d ) - \la (1-u_i (x)^d ) \Big ) \\
&&\mbox{} + \sum_{i=1}^j \gamma_{j,i}\Big (- (1-u_i(x)) + (1-u_{i+1}(x)) \Big).
\end{eqnarray*}
Rearranging now gives
\begin{eqnarray*}
(1+\la) \Delta Q_j(x)
& = & \sum_{i=1}^j (\gamma_{j,i-1} - \gamma_{j,i}) (1-u_i(x)) \\
&&\mbox{} - \la \sum_{i=1}^j (\gamma_{j,i} - \gamma_{j,i+1} ) (1-u_i(x)^d) + \gamma_{j,j} (1-u_{j+1}(x)).
\end{eqnarray*}
Recall that $\gamma_{j,0} = \gamma_{j,j+1} = 0$, and note that $\gamma_{j,1} > \gamma_{j,2} > \cdots > \gamma_{j,j} = 1$.

As before, we proceed by approximating $1-u_i(x)^d$ by $d(1-u_i(x))$, for $i \le j$.
Using first that $1-u_i(x)^d \le d(1-u_i(x))$ for each $i$, we have
\begin{eqnarray*}
\lefteqn{(1+\la) \Delta Q_j(x)} \\
&\ge & \sum_{i=1}^j (\gamma_{j,i-1} - \gamma_{j,i}) (1-u_i(x))
- \la d \sum_{i=1}^j (\gamma_{j,i} - \gamma_{j,i+1} ) (1-u_i(x))\\
&&\mbox{} + (1-u_{j+1} (x)) \\
&=& \sum_{i=1}^j (1-u_i(x)) \left[ \gamma_{j,i-1} + \la d \gamma_{j,i+1} - (\la d +1) \gamma_{j,i} \right] + (1-u_{j+1}(x)) \\
&=& - \sum_{i=1}^j (1-u_i(x)) \gamma_{j,i} \left[ \la d + 1 - 2 \sqrt{\la d} \cos \left( \frac{\pi}{j+1} \right) \right] + (1-u_{j+1}(x)) \\
&=& - \left[ \la d + 1 - 2 \sqrt{\la d} \cos \left( \frac{\pi}{j+1} \right) \right] \frac{Q_j(x)}{n} + (1-u_{j+1}(x)) \\
&\ge& - \la d \frac{Q_j(x)}{n} + \frac{Q_{j+1}(x)}{n} - 2\sqrt{\la d} \frac{Q_j(x)}{n},
\end{eqnarray*}
as claimed.  In the last line above, we used (\ref{eq:Qj+1}), as well as the inequality $2\sqrt {\la d} \cos (\pi/(j+1)) \ge \sqrt{2\la d} \ge 1$,
valid since $\la d \ge 4$.

For the upper bound, we use the facts that $1-u_{j+1}(x) \le \frac{Q_{j+1}(x)}{n}$ and
$1-u_i(x)^d \ge d(1-u_i(x)) - \binom{d}{2}(1-u_i(x))^2$, to obtain
\begin{eqnarray*}
\lefteqn{(1+\la) \Delta Q_j(x)} \\
&\le&  - \left[ \la d + 1 - 2 \sqrt{\la d} \cos \left( \frac{\pi}{j+1} \right) \right] \frac{Q_j(x)}{n} + (1 - u_{j+1}(x))\\
&&\mbox{} + \la \binom{d}{2} \sum_{i=1}^j (\gamma_{j,i} - \gamma_{j,i+1} ) (1-u_i(x))^2 \\
&\le& - \la d \frac{Q_j(x)}{n} \left( 1 - \frac{2}{\sqrt{\la d}}\right) +
\frac{Q_{j+1}(x)}{n} + \frac{P_{k-1}(x)}{n} \la \binom{d}{2} \sum_{i=1}^j \gamma_{j,i} (1-u_i(x)) \\
&\le& - \la d \frac{Q_j(x)}{n} \left( 1 - \frac{2}{\sqrt{\la d}}\right) +
\frac{Q_{j+1}(x)}{n} + \frac{P_{k-1}(x) Q_j(x)}{n^2} \la \binom{d}{2},
\end{eqnarray*}
as claimed.
\end{proof}

%\medbreak

Next we prove a similar result for the function $P_{k-1}$.  For this function, we need only a fairly
crude upper bound on the drift.

\begin{lemma} \label{lem.pk-1-drift}
For any state $x \in \Z_+^n$, we have
$$
(1+\la) \Delta P_{k-1}(x) \le - \la \left( 1 - \exp \left( - \frac{ d P_{k-1}(x)}{(k-1)n} \right) \right)  + \frac{Q_k(x)}{n}.
$$
\end{lemma}

\begin{proof}
The calculation this time is simpler: we have
\begin{eqnarray*}
\Delta P_{k-1}(x) &=& (1+ \la ) \E [P_{k-1}(X_{t+1}) - P_{k-1} (X_t) \mid X_t =x ]\\
& = & -\sum_{i=1}^{k-1}  (\la u_{i-1}(x)^d - \la u_i (x)^d - u_i(x) + u_{i+1}(x))\\
& = & \la u_{k-1}(x)^d - \la - u_k(x) + u_1(x)\\
& = & -\la (1-u_{k-1}(x)^d) + (1-u_k(x)) - (1-u_1(x)) \\
&\le & -\la (1-u_{k-1}(x)^d) + (1-u_k(x)).
\end{eqnarray*}
We have $1-u_k(x) \le \frac{Q_k(x)}{n}$ and $1-u_{k-1}(x) \ge \frac{1}{k-1} \frac{P_{k-1}(x)}{n}$, so
\begin{eqnarray*}
u_{k-1}(x)^d & \le & \left(1 -  \frac{1}{k-1} \frac{P_{k-1}(x)}{n} \right)^d \\
& \le & \exp \left( - \frac{ d P_{k-1}(x)}{(k-1)n} \right),
\end{eqnarray*}
which gives the required bound.
\end{proof}

\section{Hitting Times and Exit Times} \label{sec.hit-and-exit}

At this point, we begin the proof of Theorem~\ref{thm.technical}.  Accordingly, from now on we fix values of $n, d, k \in \N$,
$\la, \eps \in (0,1)$ such that:
\begin{eqnarray}
d^k (1-\la) &\ge& 2\log^2 n, \label{ineq:2}\\
k &\ge& 2, \label{ineq:4} \\
\eps &\le& \frac{1}{10}, \label{ineq:5} \\
\eps \sqrt d &\ge& 150 k, \label{ineq:6} \\
\eps &\ge& 100 k (1-\la) d^{k-1}, \label{ineq:7} \\
\eps^2 d n (1-\la)^2 &\ge& 600 k^2 \log^2 n. \label{ineq:8}
\end{eqnarray}

We explore some consequences of these assumptions.

\begin{lemma} \label{lem.inequalities}
For positive integers $n$, $d$, $k$, and real numbers $\la \in (0,1)$, $\eps \in (0,1)$ satisfying (\ref{ineq:2}) to (\ref{ineq:8}) above,
we also have:
\begin{eqnarray}
n &\ge& 10^{15}, \label{ineq:85} \\
\eps d &\ge& 200 k \log^2 n, \label{ineq:9} \\
\log^2 n \,(1-\la) &\le& \frac{\eps^2}{80000}, \label{ineq:95} \\
k &=& \left\lceil \frac{\log(1-\la)^{-1}}{\log d} \right\rceil, \label{ineq:10} \\
\eps^3 n (1-\la) &\ge& 60000k^3 \log^2 n \, d^{k-2}, \label{ineq:115}\\
\eps n (1-\la) &\ge & 60000, \label{ineq:1155} \\
k &\le& \log n, \label{ineq:116} \\
d^k (1-\la) &\ge& 2 k \log n, \label{ineq:103} \\
\la^k &\ge & 9/10, \label{ineq:12} \\
\beta_k \,=\, 1 - \frac{k}{(\la d)^k} &\ge& 1- \eps/2, \label{ineq:13} \\
\sum_{i=1}^\infty \frac{1}{(\la d)^i} \,\le\, \sum_{i=1}^\infty \frac{1}{(\la d)^{i/2}} &\le& \frac{\eps}{2k}. \label{ineq:14}
\end{eqnarray}
\end{lemma}

\begin{proof}
Squaring (\ref{ineq:7}) and multiplying by (\ref{ineq:8}) gives
$$
\eps^4 n \ge 6 \times 10^6 k^4 \log^2 n \, d^{2k-3}.
$$
Using (\ref{ineq:5}) and (\ref{ineq:4}) now yields
$n/ \log^2 n \ge 9.6 \times 10^{11}$,
which implies (\ref{ineq:85}).

(\ref{ineq:9}) follows from multiplying (\ref{ineq:2}) and (\ref{ineq:7}).  Multiplying by (\ref{ineq:7}) again gives
$\eps^2 \ge 20000 k^2 (1-\la) \log^2 n \, d^{k-2}$, which implies (\ref{ineq:95}) via (\ref{ineq:4}).

(\ref{ineq:10}) is equivalent to $d^k(1-\la) \ge 1 > d^{k-1}(1-\la)$, and these inequalities are clear from (\ref{ineq:2}) and (\ref{ineq:7}).

Multiplying (\ref{ineq:7}) and (\ref{ineq:8}) gives (\ref{ineq:115}), and the weaker version (\ref{ineq:1155}).

It follows from (\ref{ineq:9}) and (\ref{ineq:10}) that $k \le \log (1-\la)^{-1}$, and from (\ref{ineq:1155}) that $(1-\la)^{-1} \le n$.
(\ref{ineq:116}) follows, and also (\ref{ineq:103}) now follows from (\ref{ineq:2}).

(\ref{ineq:12}) follows from
$$
\la^k = (1-(1-\la))^k \ge 1 - k(1-\la) \ge 1- \frac{\eps}{100 d},
$$
where at the end we used (\ref{ineq:7}) and (\ref{ineq:4}).

(\ref{ineq:13}) is equivalent to $2 k \le \eps (\la d)^k$, which follows comfortably from (\ref{ineq:9}).

From (\ref{ineq:6}) and (\ref{ineq:12}), we have that
$$
\frac{1}{\sqrt{\la d}} \le \sqrt{\frac{10}{9}} \frac{\eps}{150k} \le \frac{\eps}{100k},
$$
which comfortably implies (\ref{ineq:14}).
\end{proof}

We shall mention explicitly each time we use one of the inequalities (\ref{ineq:2})--(\ref{ineq:14}).  Exceptionally, we will not mention
(\ref{ineq:85}); on several occasions we note that an inequality holds for large enough $n$, and $n \ge 10^{15}$ will always
suffice.

%\bigbreak

We define a sequence of pairs of subsets of $\Z_+^n$.  Each pair consists of a set $\cS_0$ in which
some inequality holds, and a set $\cS_1$ in which a looser version of the inequality holds: we also demand that $\cS_0$ and $\cS_1$ be
subsets of the previous set $\cR_1$ in the sequence.  Associated with each pair $(\cS_0, \cS_1)$ in the sequence is a {\em hitting time}
$$
T_\cS = \inf \{ t \ge T_\cR : X_t \in \cS_0 \},
$$
where $(\cR_0,\cR_1)$ is the previous pair in the sequence, and an {\em exit time}
$$
T_\cS^\dagger = \inf \{ t \ge T_\cS : X_t \notin \cS_1 \}.
$$
Our aim in each case is to prove that, with high probability, unless the previous exit time $T_\cR^\dagger$ occurs
early, $T_\cS$ is unlikely to be larger than some quantity $m_\cS$ whose order is to be thought of as polynomial in $n$.  More precisely, if we
start in a state in $\cA(\ell,g)$, then the sum of all the $m_\cS$ is of order at most the maximum of $k n (1-\la)^{-1}$, $g n (1-\la)^{-1}$
and $\ell n$; note that $k \le \log n$ and $(1-\la)^{-1} \le n$ (see (\ref{ineq:116}) and (\ref{ineq:1155})), so if $\ell$ and $g$ are bounded by
a polynomial in $n$, then so are all the $m_\cS$.

Throughout the proof, we set
$$
s_0 = e^{\frac 13 \log^2 n}.
$$
We shall also prove that, again with high probability, each exit time $T_\cS^\dagger$
is at least $s_0$, which is larger than the sum of all the terms $m_\cS$.
For convenience, we shall not be too precise about our error probabilities, and simply declare them all to be at most
$$
1/s_0 = e^{-\frac13 \log^2 n},
$$
or some small multiple of $1/s_0$.

We fix, for the moment, a pair of positive real numbers $\ell \ge k$ and $g \ge k$.
We set $q(\ell,g) = (23k + 72g) \eps^{-1} n (1-\la)^{-1} + 8 \ell n$, and we make the (mild) assumption that
$\ell$ and $g$ are chosen so that $q(\ell,g) \le s_0/2$.

The first pair of sets in our sequence will be as defined earlier:
\begin{eqnarray*}
\cA_0 = \cA_0(\ell,g) &=& \{ x : \| x \|_\infty \le \ell \mbox{ and } \| x \|_1 \le gn \}, \\
\cA_1 = \cA_1(\ell,g) &=& \{ x : \| x \|_\infty \le 3\ell \mbox{ and } \| x \|_1 \le 3gn \},
\end{eqnarray*}
and we adopt the hypothesis that $X_0 = x_0$ almost surely, where $x_0$ is a fixed state in $\cA_0=\cA_0(\ell,g)$, so
that $T_\cA := \min \{ t\ge 0 : X_t \in \cA_0 \} = 0$.

For $\ell = \ell^* = \log^2 n (1-\la)^{-1}$ and $g = g^* = 2 (1-\la)^{-1}$,
Lemma~\ref{lem.cA} tells us that indeed the exit time $T_\cA^\dagger = \inf\{ t>0 : X_t \notin \cA_1^* \}$ is unlikely to be
less than $s_0$.  For smaller values of $\ell$ and $g$, we do not know this a priori.

The sets we define are dependent on the chosen values of $n$, $d$, $k$, $\la$ and $\eps$, as well as on $\ell$ and $g$.
For the most part, we drop reference to this dependence from the notation.  However, later in the paper we shall need to
vary $\eps$ while keeping all other parameters fixed; in this case, we shall use the notation (e.g.) $\cB_0^\eps$ to
emphasise the dependence.

We define:
\begin{eqnarray*}
\cB_0 &=& \{ x: Q_k(x) \le (1+\eps) n (1-\la) (\la d)^{k-1} \} \cap \cA_1, \\
\cB_1 &=& \{ x: Q_k(x) \le (1+2\eps) n (1-\la) (\la d)^{k-1} \} \cap \cA_1, \\
\cC_0 &=& \{ x: P_{k-1}(x) \le 2k n (1-\la) (\la d)^{k-2} \} \cap \cB_1, \\
\cC_1 &=& \{ x: P_{k-1}(x) \le 3k n (1-\la) (\la d)^{k-2} \} \cap \cB_1, \\
\cD_0 &=& \{ x: Q_{k-1}(x) \le (1+4\eps) n (1-\la) (\la d)^{k-2} \} \cap \cC_1, \\
\cD_1 &=& \{ x: Q_{k-1}(x) \le (1+5\eps) n (1-\la) (\la d)^{k-2} \} \cap \cC_1, \\
\cE_0 &=& \{ x: u_{k+1}(x) \le \eps (1-\la) \mbox{ and } Q_k(x) \ge (1-3\eps) n (1-\la) (\la d)^{k-1} \} \cap \cD_1, \\
\cE_1 &=& \{ x: u_{k+1}(x) \le \eps (1-\la) \mbox{ and } Q_k(x) \ge (1-4\eps) n (1-\la) (\la d)^{k-1} \} \cap \cD_1.
\end{eqnarray*}
Next we have a sequence of pairs of sets, indexed by $j=k-1, \dots, 1$:
\begin{eqnarray*}
\cG_0^j &=& \Big\{ x: \Big[ 1- (4 + \frac{k-j-1/2}{k}) \eps \Big] n (1-\la) (\la d)^{j-1} \le Q_j(x) \\
&&\mbox{} \qquad \le \Big[1 + (4 + \frac{k-j-1/2}{k}) \eps \Big] n (1-\la) (\la d)^{j-1} \Big\} \cap \cG_1^{j+1}, \\
\cG_1^j &=& \Big\{ x: \Big[1 - (4 + \frac{k-j}{k}) \eps \Big] n (1-\la) (\la d)^{j-1} \le Q_j(x) \\
&&\mbox{} \qquad \le \Big[1 + (4 + \frac{k-j}{k}) \eps \Big] n (1-\la) (\la d)^{j-1} \Big\} \cap \cG_1^{j+1}.
\end{eqnarray*}
where we declare $\cG_1^k$ to be equal to $\cE_1$.
Finally, departing slightly from our pattern, we define
$$
\cH = \cH_0 = \cH_1 = \{ x: u_{k+1}(x) = 0 \} \cap \cG_1^1.
$$

The hitting times and exit times are all defined in accordance with the pattern given.  For instance
$T_\cB = \inf \{ t : X_t \in \cB_0 \}$, $T_\cB^\dagger = \inf \{ t > T_\cB : X_t \notin \cB_1 \}$, and
$T_\cC = \inf \{ t \ge T_\cB : X_t \in \cC_0 \}$.
We also set $T_{\cG^k} = T_\cE$ and $T_{\cG^k}^\dagger = T_\cE^\dagger$, in accordance with the notion that
the set pair $(\cG^{k-1}_0,\cG^{k-1}_1)$ follows $(\cE_0,\cE_1)$ in the sequence.

Initially, the sets above all depend on the values of $\ell$ and $g$ defining the initial pair of sets $(\cA_0, \cA_1)$,
since all the sets are intersected with $\cA_1$.  However, since states in $\cH$ have no queue of length $k+1$ or greater, we have
$\cH \subseteq \cA_0(k,k) \subseteq \cA_1(\ell,g)$ for all $\ell, g \ge k$, and so the set $\cH$ does not depend on $\ell$ and $g$, provided these
parameters are each at least $k$.

%\bigbreak

We now state a sequence of lemmas.
Throughout, we assume that $X_0 = x_0$ a.s., where $x_0$ is an arbitrary state in $\cA_0 = \cA_0(\ell,g)$.

\begin{lemma} \label{lem.cB}
Let $m_\cB = 8 k \eps^{-1} n (1-\la)^{-1}$.
\begin{enumerate}
\item [(1)] $\Pr (T_\cB \land T_\cA^\dagger \ge m_\cB) \le 1/s_0$.
\item [(2)] $\Pr (T_\cB^\dagger \le s_0 < T_\cA^\dagger) \le 1/s_0$.
\end{enumerate}
\end{lemma}

\begin{lemma} \label{lem.cC}
Let $m_\cC = 12 k n (1-\la)^{-1}(\la d)^{1-k}$.
\begin{enumerate}
\item [(1)] $\Pr (T_\cC \land T_\cB^\dagger \ge T_\cB + m_\cC) \le 1/s_0$.
\item [(2)] $\Pr (T_\cC^\dagger \le s_0 < T_\cB^\dagger) \le 1/s_0$.
\end{enumerate}
\end{lemma}

\begin{lemma} \label{lem.cD}
Let $m_\cD = 8 \eps^{-1} n (1-\la)^{-1} (\la d)^{-k/2}$.
\begin{enumerate}
\item [(1)] $\Pr (T_\cD \land T_\cC^\dagger \ge T_\cC + m_\cD) \le 1/s_0$.
\item [(2)] $\Pr (T_\cD^\dagger \le s_0 < T_\cC^\dagger) \le 1/s_0$.
\end{enumerate}
\end{lemma}

\begin{lemma} \label{lem.cE}
Let $m_\cE = m_\cE(g)= (13k+72g) \eps^{-1} n (1-\la)^{-1}$.
\begin{enumerate}
\item [(1)] $\Pr (T_\cE \land T_\cD^\dagger \ge T_\cD + m_\cE) \le 1/s_0$.
\item [(2)] $\Pr (T_\cE^\dagger \le s_0 < T_\cD^\dagger) \le 1/s_0$.
\end{enumerate}
\end{lemma}

\begin{lemma} \label{lem.cG}
Let $m_\cG = 32 k \eps^{-1} n (1-\la)^{-1} (\la d)^{-1}$.  For $j=k-1, \dots, 1$, we have:
\begin{enumerate}
\item [(1)] For $j= k-1, \dots, 1$, $\Pr (T_{\cG^j} \land T_{\cG^{j+1}}^\dagger \ge T_{\cG^{j+1}} + m_\cG) \le 1/s_0$;
\item [(2)] For $j= k-1, \dots, 1$, $\Pr (T_{\cG^j}^\dagger \le s_0 < T_{\cG^{j+1}}^\dagger) \le 1/s_0$.
\end{enumerate}
\end{lemma}

\begin{lemma} \label{lem.cH}
Let $m_\cH = m_\cH(\ell) = n (8 \ell + 32\log^2 n)$.
\begin{enumerate}
\item [(1)] $\Pr (T_\cH \land T_{\cG^1}^\dagger \ge T_{\cG^1} + m_\cH) \le 1/s_0$.
\item [(2)] $\Pr (T_\cH^\dagger \le s_0 < T_{\cG^1}^\dagger) \le 1/s_0$.
\end{enumerate}
\end{lemma}

We note here that $q(\ell,g) = (23k + 72g) \eps^{-1} n (1-\la)^{-1} + 8 \ell n$ is larger than all the constants
$m_\cB$, $m_\cC$, \dots appearing in the lemmas, so these constants are all at most $s_0/2$.
Combining the lemmas gives the following result.

\begin{proposition} \label{prop.stay-in-H}
For any $x_0 \in \cA_0 = \cA_0(\ell,g)$, and a copy $(X_t)$ of the process with $X_0 = x_0$ a.s., we have
$$
\Pr (X_t \in \cH \mbox{ for all } t \in [q(\ell,g),s_0] ) \ge 1 - \frac{2k+8}{s_0} - \Pr (T_\cA^\dagger \le s_0).
$$
\end{proposition}

\begin{proof}
The idea is that, with high probability, either the chain $(X_t)$ exits $\cA_1(\ell,g)$ before time $s_0$, or
the chain enters each of the sets $\cB_0$, \dots, $\cH_0$ in turn, within time $q(\ell,g)$, and does not exit any of the sets $\cA_1$, \dots, $\cH_1$
before time $s_0$, which is what we need.

More formally, consider the following list of events concerning the various stopping times we have defined:
$$
E_1 = \{T_\cA^\dagger > s_0\}, \quad E_2 = \{T_\cB \le m_\cB\}, \quad E_3 = \{T_\cB^\dagger > s_0\},
$$
$$
E_4 = \{T_\cC \le m_\cB + m_\cC\}, \quad E_5 = \{T_\cC^\dagger > s_0\}, \quad E_6 = \{T_\cD \le m_\cB + m_\cC + m_\cD\},
$$
$$
E_7 = \{T_\cD^\dagger > s_0\}, \quad E_8 = \{T_\cE \le m_\cB + \dots + m_\cE\}, \quad E_9 = \{T_\cE^\dagger > s_0\},
$$
$$
E_{10} = \{T_{\cG^{k-1}} \le m_\cB + \cdots + m_\cE + m_\cG\}, \quad E_{11} = \{T_{\cG^{k-1}}^\dagger > s_0\}, \quad \dots,
$$
$$
E_{2k+6} = \{T_{\cG^{1}} \le m_\cB + \cdots + (k-1) m_\cG\}, \quad E_{2k+7} = \{T_{\cG^1}^\dagger > s_0\},
$$
$$
E_{2k+8}= \{T_\cH \le m_\cB + \cdots + (k-1)m_\cG + m_\cH\}, \quad E_{2k+9} = \{T_\cH^\dagger > s_0\}.
$$
If $E_{2k+8}$ holds, then
\begin{eqnarray*}
T_\cH &\le& m_\cB + m_\cC + m_\cD + m_\cE + (k-1)m_\cG + m_\cH \\
&\le& 8k \eps^{-1} n (1-\la)^{-1} + 12 k n (1-\la)^{-1}(\la d)^{1-k} \\
&&\mbox{} + 8 \eps^{-1} n (1-\la)^{-1} (\la d)^{-k/2} + (13k+72g) \eps^{-1} n (1-\la)^{-1} \\
&&\mbox{} + 32 (k-1) k \eps^{-1} n (1-\la)^{-1} (\la d)^{-1} + n (8 \ell + 32\log^2 n) \\
&\le& k \eps^{-1} n (1-\la)^{-1} \big( 8 + \frac{12\eps}{\la d} + \frac{8}{\la d} + 13 + \frac{32(k-1)}{\la d} \\
&&\mbox{} + 32\eps \log^2 n (1-\la) \big) + 72 g \eps^{-1} n (1-\la)^{-1} + 8\ell n \\
&\le& \eps^{-1} n (1-\la)^{-1} ( 23k + 72g ) + 8 \ell n \\
&=& q(\ell,g),
\end{eqnarray*}
where we used (\ref{ineq:6}) and (\ref{ineq:12}) to tell us that $\frac{32(k-1) + 8 + 12\eps}{\la d} \le 1$, and (\ref{ineq:95}) to
show that $32\eps \log^2 n (1-\la) \le 1$.  Therefore, if $E = \bigcap_{j=1}^{2k+9} E_j$ holds, then in particular $E_{2k+8}$ and $E_{2k+9}$ hold,
which implies that $X_t \in \cH$ for $q(\ell,g) \le t \le s_0$.  Thus $E$ is contained in the event $\{ X_t \in \cH \mbox{ for all } t \in [q(\ell,g),s_0]\}$,
and it suffices to show that $\Pr (\overline{E}) \le \frac{2k+8}{s_0} + \Pr (\overline{E_1})$.  We write
$$
\Pr (\overline{E}) = \Pr (\overline {E_1}) + \sum_{j=2}^{2k+9} \Pr \left( \overline{E_j} \cap \bigcap_{i=1}^{j-1} E_i \right),
$$
and now we see that it suffices to prove that each of the terms $\Pr \left( \overline{E_j} \cap \bigcap_{i=1}^{j-1} E_i \right)$ is at most $1/s_0$.

We show how to derive the first few of these inequalities from Lemmas~\ref{lem.cB}-\ref{lem.cH}; first we have
$$
\Pr (\overline{E_2} \cap E_1) = \Pr (T_\cA^\dagger > s_0, \, T_\cB > m_\cB) \le \Pr (T_\cB \land T_\cA^\dagger \ge m_\cB) \le 1/s_0
$$
by Lemma~\ref{lem.cB}(1).  Then we have
$$
\Pr (\overline{E_3} \cap E_1 \cap E_2) \le \Pr(\overline{E_3} \cap E_1) = \Pr(T_\cB^\dagger \le s_0 < T_\cA^\dagger) \le 1/s_0
$$
by Lemma~\ref{lem.cB}(2).
Next we have, using the fact that $m_\cB + m_\cC \le s_0$,
\begin{eqnarray*}
\Pr(\overline{E_4} \cap E_1 \cap E_2 \cap E_3) &\le& \Pr (\overline{E_4} \cap E_2 \cap E_3) \\
&=&\Pr (T_\cB^\dagger > s_0, \, T_\cB \le m_\cB, \, T_\cC > m_\cB + m_\cC)\\
&\le& \Pr (T_\cC \land T_\cB^\dagger > m_\cB + m_\cC, \, T_\cB \le m_\cB) \\
&\le& \Pr (T_\cC \land T_\cB^\dagger > T_\cB + m_\cC) \\
&\le& 1/s_0,
\end{eqnarray*}
by Lemma~\ref{lem.cC}(1).  For $j = 5, \dots, 2k+9$, the upper bound on $\Pr \left( \overline{E_j} \cap \bigcap_{i=1}^{j-1} E_i \right)$ follows
either as for $j=3$ or as for $j=4$: it is important here that $m_\cB + m_\cC + m_\cD + m_\cE + (k-1)m_\cG + m_\cH \le q(\ell,g) \le s_0$.
\end{proof}

We now have the following consequence for an equilibrium copy $(Y_t)$ of the $(n,d,\la)$-supermarket process.

\begin{corollary} \label{cor.eqm}
$\Pr(Y_t \in \cH \mbox{ for all } t \in [0,s_0]) \ge 1-(4k+20)/s_0 \ge 1-e^{-\frac14 \log^2 n}$,
for $n \ge 1000$.
\end{corollary}

\begin{proof}
Recall the definitions of $\ell^*$, $g^*$, $\cA_0^*$ and $\cA_1^*$ from Section~\ref{sec.coupling}.  Set also $q^* = q(\ell^*,g^*)$, and note that
$q^* \le s_0/2$, with plenty to spare.
From Lemma~\ref{lem.compd1}, we have that
$\Pr (Y_0 \notin \cA_0^*) \le n e^{-\log^2 n} \le e^{-\frac13 \log^2 n} = 1/s_0$,
since $n\ge 5$.  Also, from Lemma~\ref{lem.cA}, for a copy $(X_t^x)$ of the process starting in a state $x \in \cA_0^*$, we have that
$\Pr(T_\cA^\dagger < s_0) \le 1/s_0$.  We now have
\begin{eqnarray*}
\lefteqn{ \Pr (Y_t \notin \cH \mbox{ for some } t \in [0,s_0/2])} \\
&=& \Pr (Y_t \notin \cH \mbox{ for some } t \in [q^*,q^*+s_0/2]) \\
&\le& \Pr (Y_t \notin \cH \mbox{ for some } t \in [q^*,q^*+s_0/2] \mid Y_0 \in \cA_0^*) + \Pr (Y_0 \notin \cA_0^*) \\
&\le& \Pr (Y_t \notin \cH \mbox{ for some } t \in [q^*,s_0] \mid Y_0 \in \cA_0^*) + \Pr (Y_0 \notin \cA_0^*) \\
&\le& \sup_{x \in \cA_0^*} \Pr (X^x_t \notin \cH \mbox{ for some } t \in [q^*,s_0]) + \frac{1}{s_0} \\
&\le& \frac {2k+8}{s_0} + \frac{1}{s_0} + \frac{1}{s_0} = \frac{2k+10}{s_0},
\end{eqnarray*}
by Proposition~\ref{prop.stay-in-H}.  Hence $\Pr (Y_t \notin \cH \mbox{ for some } t \in [0,s_0]) \le (4k+20)/s_0$.

For the final inequality, note that $(4k+20)/s_0 \le (4\log n +20)e^{-\frac13 \log^2 n} < e^{-\frac14 \log^2 n}$ for $n \ge 1000$.
\end{proof}

We can now use the result above to prove the following more explicit version of Proposition~\ref{prop.stay-in-H}.

\begin{theorem} \label{thm.stay-in-H}
Suppose that $\ell$ and $g$ are at least $k$, and that $q(\ell,g) \le s_0/2$.
Let $x_0$ be any queue-lengths vector in $\cA_0(\ell,g)$, and suppose that $X_0 = x_0$ a.s.
Then we have
$$
\Pr (X_t \in \cH \mbox{ for all } t \in [q(\ell,g),s_0] ) \ge 1 - \frac{6k+28}{s_0}.
$$
\end{theorem}

\begin{proof}
We apply, successively, Proposition~\ref{prop.stay-in-H}, Lemma~\ref{lem.l1-inf} and Corollary~\ref{cor.eqm} to obtain that
\begin{eqnarray*}
\lefteqn{\Pr (X_t \in \cH \mbox{ for all } t \in [q(\ell,g),s_0] )} \\
&\ge& 1 - \frac{2k+8}{s_0} - \Pr (T_\cA^\dagger \le s_0) \\
&=& 1- \frac{2k+8}{s_0} - \Pr (\exists t \in [0,s_0],\, X_t \notin \cA_1(\ell,g)) \\
&\ge& 1- \frac{2k+8}{s_0} - \Pr (\exists t \in [0,s_0],\, Y_t \notin \cA_0(\ell,g)) \\
&\ge& 1- \frac{2k+8}{s_0} - \Pr (\exists t \in [0,s_0],\, Y_t \notin \cH) \\
&\ge& 1- \frac{2k+8}{s_0} - \frac{4k+20}{s_0},
\end{eqnarray*}
as required.
\end{proof}

In the next sections, we shall prove Lemmas~\ref{lem.cB} to~\ref{lem.cH}.  Then we show that $\cH \subseteq \cN^\eps(n,d,\la,k)$.
Theorem~\ref{thm.technical} will then follow from Corollary~\ref{cor.eqm}, since
$s_0/2 = \frac12 e^{\frac13 \log^2 n} > e^{\frac14 \log^2 n}$ for $n \ge 18$.

%\bigbreak

We draw one further conclusion from the results in this section.  Suppose that $(X_t)$ starts in a state
$x_0$ in the set
$$
\cI = \cA_0^* \cap \cB_0 \cap \cC_0 \cap \cD_0 \cap \cE_0 \cap \bigcap_{j=1}^{k-1} \cG^j_0 \cap \cH_0.
$$
Then all the hitting times $T_\cB$, $T_\cC$, $T_\cD$, $T_\cE$, $T_\cG^j$ and $T_\cH$ are equal to~0.  In the notation of the proof of
Proposition~\ref{prop.stay-in-H}, this implies that the events $E_j$ for $j$ even occur with probability~1.  Also, by Lemma~\ref{lem.cA},
$\Pr (\overline{E_1}) \le 1/s_0$.  So following the proof of Proposition~\ref{prop.stay-in-H} yields the result below.

\begin{theorem} \label{thm.start-in-I-stay-in-H}
Suppose $x_0 \in \cI$, and $X_0 = x_0$ a.s.  Then
$$
\Pr (X_t \in \cH \mbox{ for all } t \in [0,s_0] ) \ge 1 - (k+5)/s_0.
$$
\end{theorem}

We shall explore the consequences of this result further in Section~\ref{sec.H-I-N}.

\section{Proofs of Lemmas~\ref{lem.cB}, \ref{lem.cC} and~\ref{lem.cD}} \label{sec.BCD}

In this section, we prove the first three of the sequence of lemmas stated in the previous section, and also derive tighter
inequalities on the drifts of the functions $Q_j(x)$ for $x\in \cD_1$.
The proofs of the three lemmas are all straightforward applications of Lemma~\ref{lem.drifts-down2}, and all similar to one another.

\bigbreak

\noindent
{\bf Proof of Lemma~\ref{lem.cB}}

\begin{proof}
We apply Lemma~\ref{lem.drifts-down2}.  We set $(\phi_t) = (\cF_t)$, the natural filtration of the process, and also: $F=Q_k$, $\cS = \cA_1$,
$$
h = (1+\eps) (1-\la)n (\la d)^{k-1}, \quad \rho= \eps (1-\la) n (\la d)^{k-1},
$$
$m = m_\cB = 8 k \eps^{-1} n (1-\la)^{-1}$, $s = s_0 = e^{\frac13 \log^2 n}$ and $T^*=0$.
We have $\rho \ge 60000 (\la d)^{k-1} \ge 2$, by (\ref{ineq:1155}) and (\ref{ineq:12}).  It is also clear that
$Q_k(x) \le c := k n$ for any $x \in \Z_+^n$.  We note also that $Q_k$ takes jumps of size at most~1.

Suppose now that $Q_k(x) \ge h$.  Then
$$
\exp\left( -\frac{d Q_k(X_t)}{kn} \right) \le \exp\left( - \frac{(1-\la)(\la d)^k}{k} \right).
$$
Now we have, using first (\ref{ineq:12}) and (\ref{ineq:103}) and then (\ref{ineq:1155}), that
$$
\frac{(1-\la)(\la d)^k}{k} \ge \frac{9}{5} \log n >  - \log \left( \frac{\eps}{60000}(1-\la) \right).
$$
Thus we have
$$
\exp\left( -\frac{d Q_k(X_t)}{kn} \right) \le \frac{\eps}{60000}(1-\la).
$$

Hence, by Lemma~\ref{lem.qk-drift}, for $x$ with $Q_k(x) \ge h$, we have
\begin{eqnarray*}
(1+\la)\Delta Q_k(x) &\le & \beta_k \big( (1-\la) - u_{k+1} (x) + \la \exp( - d Q_k(x)/ k n) \big) \\
&&\mbox{} - \frac{1}{(\la d)^{k-1}} \frac{Q_k(x)}{n} \left(1-\frac{2}{\la d}\right), \\
&\le& \beta_k \left( (1-\la) + \la \frac{\eps(1-\la)}{60000} \right) - (1+\eps)(1-\la) \left( 1 -\frac{2}{\la d}\right) \\
%&=& (1-\la) \left[ \beta_k (1+ \frac{\la}{3}\eps) - (1+\eps)(1-\frac{2}{\la d}) \right] \\
&\le& (1-\la) \left[ 1 + \frac{\eps}{60000} -1 - \eps + (1+\eps)\frac{2}{\la d} \right] \\
&\le& - (1-\la) \frac{\eps}{2},
\end{eqnarray*}
where at the end we used the fact that $\frac{2}{\la d} \le \frac{\eps}{6}$, which follows comfortably from (\ref{ineq:6}).
So $\Delta Q_k(x) \le - (1-\la) \eps / 4 := - v$.  Note that $m_\cB v = 2c$.

We have now verified that the conditions of Lemma~\ref{lem.drifts-down2} are satisfied, for the given values of the parameters.
As in the lemma, we have $T_0 = T_\cA^\dagger$, $T_1 = \inf \{ t : Q_k(X_t) \le h \}$ and $T_2 = \inf \{ t > T_1 : Q_k(X_t) \ge h + \rho \}$.

It need not be the case that $T_1 = T_\cB$, since $X_{T_1}$ need not be in $\cA_1$.  However, we do have
$T_1 \land T_\cA^\dagger = T_\cB \land T_\cA^\dagger$ and thus
\begin{eqnarray*}
\Pr (T_\cB \land T_\cA^\dagger > m_\cB) &=& \Pr (T_1 \land T_\cA^\dagger > m_\cB) \\
&\le& \exp( - v^2 m_\cB/8 ) \\
&=& \exp( - \eps k n (1-\la) / 16) \\
&\le& \exp( - 3750 \eps^{-2} k^4 \log^2 n d^{k-2}) \\
&\le& 1/s_0.
\end{eqnarray*}
In the penultimate line, we used (\ref{ineq:115}); in the final line, all we needed was that $3750 \eps^{-2} k^4 \log^2 n d^{k-2} \ge \frac13 \log^2 n$,
which is true with plenty to spare.

Also the events $T_2 \le s_0 < T_\cA^\dagger$ and $T_\cB^\dagger \le s_0 < T_\cA^\dagger$ coincide, so we have
\begin{eqnarray*}
\Pr (T_\cB^\dagger \le s_0 < T_\cA^\dagger) &\le& \Pr (T_2 \le s_0 < T_\cA^\dagger) \\
&\le& s \exp( - \rho v) \\
&=& s_0 \exp (- \eps^2 (1-\la)^2 n (\la d)^{k-1} / 6) \\
&\le& s_0 \exp(- 90 k^2 \log^2 n \, d^{k-2}) \\
&\le& 1/s_0,
\end{eqnarray*}
as required, where in the penultimate line we used (\ref{ineq:8}) and (\ref{ineq:12}).
\end{proof}

\bigbreak

\noindent
{\bf Proof of Lemma~\ref{lem.cC}}

\begin{proof}
Again we apply Lemma~\ref{lem.drifts-down2} to the Markov process $(X_t)$ with its natural filtration.  Set $F = P_{k-1}$,
$\cS = \cB_1$,
$$
h = 2 kn (1-\la) (\la d)^{k-2}, \, \rho = kn (1-\la) (\la d)^{k-2},
$$
$m = m_\cC = 12k n(1-\la)^{-1} (\la d)^{1-k}$, and $s = s_0$.  Set $T^* = T_\cB$.  It is again clear from (\ref{ineq:1155})
that $\rho \ge 2$, and also that $P_{k-1}$ takes jumps of size at most~1, and that $P_{k-1}(x) \le c: = kn$ for all $x \in \Z^n_+$.
Here $T_0 = T_\cB^\dagger$, $T_1 = \inf \{ t \ge T_\cB : P_{k-1}(X_t) \le h \}$,
and $T_2 = \inf \{ t > T_1 : P_{k-1}(X_t) \ge h + \rho \}$.

For $x \in \cB_1$ with $P_{k-1}(x) \ge h$, we have $Q_k(x) \le (1+2\eps)n(1-\la)(\la d)^{k-1}$ and so, by Lemma~\ref{lem.pk-1-drift},
\begin{eqnarray*}
\lefteqn{(1+\la) \Delta P_{k-1}(x)} \\
& \le & - \la \left( 1 - \exp\left( - \frac{ d P_{k-1}(x)}{(k-1)n} \right) \right)  + \frac{Q_k(x)}{n} \\
& \le & - \la \left( 1 - \exp\left( - 2d(1-\la)(\la d)^{k-2}\right) \right) + (1+2\eps)(1-\la)(\la d)^{k-1}.
\end{eqnarray*}
Now we have $y = 2d(1-\la)(\la d)^{k-2} \le \frac{\eps}{50k} \le 1/6$, by (\ref{ineq:7}); it is easy to check that
$e^{-y} \le 1 - \frac56 y$ for $0\le y \le 1/6$.  Also $1+2\eps < 4/3$ from (\ref{ineq:5}), so
\begin{eqnarray*}
(1+\la) \Delta P_{k-1}(x) & \le & - \frac{5}{3} (1-\la) (\la d)^{k-1} + \frac{4}{3} (1-\la) (\la d)^{k-1} \\
& = & - \frac{1}{3} (1-\la) (\la d)^{k-1},
\end{eqnarray*}
We conclude that, for such $x$, $\Delta P_{k-1}(x) \le - \frac{1}{6} (1-\la) (\la d)^{k-1} := - v$.
Note that $m_\cC v = 2 c$.

As in the previous lemma, it need not be the case that $T_1 = T_\cC$, since $X_{T_1}$ need not be in $\cB_1$, so we may have $T_\cC > T_1$.  However,
we do have $T_1 \land T_\cB^\dagger = T_\cC \land T_\cB^\dagger$.  From Lemma~\ref{lem.drifts-down2}, we obtain, using also (\ref{ineq:115}),
\begin{eqnarray*}
\Pr(T_\cC \land T_\cB^\dagger > T_\cB + m_\cC) &=&
\Pr( T_1 \land T_0 > T_\cB + m_\cC) \\
&\le& \exp( - v^2 m_\cC /8) \\
&=& \exp ( - k n (1-\la) (\la d)^{k-1} / 24 ) \\
&\le& \exp \left( - 2500 \eps^{-3} n k^4 \log^2 n\right) \\
&\le& 1/s_0.
\end{eqnarray*}

Similarly, the events $T_2 \le s_0 < T_\cB^\dagger$ and $T_\cC^\dagger \le s_0 < T_\cB^\dagger$ coincide, and so
\begin{eqnarray*}
\Pr( T_\cC^\dagger \le s_0 < T_\cB^\dagger ) &=& \Pr (T_2 \le s_0 < T_0) \\
&\le& s_0 \exp( - \rho v) \\
&=& s_0 \exp( - kn (1-\la)^2 (\la d)^{2k-3} /6) \\
&\le& s_0 \exp \left( - 90 \eps^{-2} (\la d)^{2k-4} k^3 \log^2 n \right) \\
&\le& 1/s_0.
\end{eqnarray*}
as required.  Here, we used (\ref{ineq:8}) and (\ref{ineq:12}).
\end{proof}

\bigbreak

\noindent
{\bf Sketch of proof of Lemma~\ref{lem.cD}}

\begin{proof}
The basic plan for this proof is the same as for the previous two lemmas, but here we have to take account of the fact that $Q_{k-1}$ can take
jumps of size up to $(\la d)^{(k-2)/2}$, and accordingly we apply Lemma~\ref{lem.drifts-down2} to the ``scaled'' function
$F(x) = Q'_{k-1}(x) = Q_{k-1}(x)/(\la d)^{(k-2)/2}$.

Apart from this, the proof is identical in structure to that of Lemma~\ref{lem.cC}, and we give only the key calculation.
For $x\in \cC_1$ with $Q'_{k-1}(x) \ge h = (1+4\eps) n(1-\la) (\la d)^{(k-2)/2}$, we have $Q_k(x) \le (1+2\eps) n (1-\la) (\la d)^{k-1}$,
$P_{k-1}(x) \le 3kn (1-\la) (\la d)^{k-2}$ and $Q_{k-1}(x) \ge (1+4\eps) n (1-\la) (\la d)^{k-2}$.  Thus,
by Lemma~\ref{lem.qj-drift} with $j=k-1$, we have
\begin{eqnarray*}
\lefteqn{(1+\la) \Delta Q_{k-1}(x)} \\
&\le& - \la d \frac{Q_{k-1}(x)}{n} \left( 1 - \frac{2}{\sqrt{\la d}} - \frac{d P_{k-1}(x)}{n} \right) + \frac{Q_k(x)}{n}, \\
&\le& - \la d (1+4\eps) (1-\la) (\la d)^{k-2} \left( 1 - \frac{2}{\sqrt{\la d}} - 3 k d (1-\la) (\la d)^{k-2} \right) \\
&&\mbox{} + (1+2\eps) (1-\la) (\la d)^{k-1} \\
&=& - (1-\la)(\la d)^{k-1} \Big[  (1+4\eps) \left( 1 - \frac{2}{\sqrt{\la d}} - 3 k d (1-\la) (\la d)^{k-2} \right) - (1+2\eps) \Big] \\
&\le& - (1-\la)(\la d)^{k-1} \Big[  (1+4\eps) \big(1 - \frac{\eps}{50} - \frac{3\eps}{100}\big) - (1+2\eps) \Big] \\
&\le& - \eps (1-\la) (\la d)^{k-1}.
\end{eqnarray*}
In the penultimate line, we used (\ref{ineq:6}) and (\ref{ineq:12}), giving that $\eps \sqrt{\la d} \ge 100$, and also
(\ref{ineq:7}), giving that $\eps \ge 100k (1-\la) \la^{k-2} d^{k-1}$.  Thus, for such $x$, the drift in the scaled chain satisfies
$\Delta Q'_{k-1}(x) \le - \frac12 \eps (1-\la) (\la d)^{k/2} := - v$.  Now $m_\cC v = 4n$, and $Q'_{k-1}(x) \le 2n$ for all $x$
by (\ref{eq:Qjbound}).

It is now straightforward to derive the result.
\end{proof}

A queue-lengths vector $x \in \cD_1$ satisfies the three inequalities:
\begin{eqnarray}
Q_k(x) &\le& (1+2\eps) n (1-\la) (\la d)^{k-1}, \label{eq:Qkx} \\
P_{k-1}(x) &\le& 3k n (1-\la) (\la d)^{k-2}, \notag \\
Q_{k-1}(x) &\le& (1+5\eps) n (1-\la) (\la d)^{k-2}; \label{eq:Qk-1x}
\end{eqnarray}
in fact the second of these is redundant, as $P_{k-1}(x) \le Q_{k-1}(x) \le 2n (1-\la) (\la d)^{k-2}$ for all $x \in \Z_+^n$.
Substituting these bounds into the bounds of Lemmas~\ref{lem.qk-drift} and~\ref{lem.qj-drift}, we obtain the following.

\begin{lemma} \label{lem.drifts}
For $x \in \cD_1$, we have
\begin{eqnarray*}
(1+\la) \Delta Q_k(x) &\le& \beta_k(1-\la-u_{k+1}(x)) - \frac{Q_k(x)}{n (\la d)^{k-1}} \\
&& \mbox{} + \exp(-dQ_k(x)/k n) + \frac{\eps}{6}(1-\la), \\
(1+\la) \Delta Q_k(x) &\ge& \beta_k(1-\la-u_{k+1}(x)) - \frac{Q_k(x)}{n (\la d)^{k-1}} - \frac{\eps}{6}(1-\la),
\end{eqnarray*}
and, for $1\le j \le k-1$,
\begin{eqnarray*}
(1+\la) \Delta Q_j(x) &\le& -\la d \frac{Q_j(x)}{n} \left( 1 - \frac{\eps}{25k} \right) + \frac{Q_{j+1}(x)}{n},\\
(1+\la) \Delta Q_j(x) &\ge& -\la d \frac{Q_j(x)}{n} \left( 1 + \frac{\eps}{50k} \right) + \frac{Q_{j+1}(x)}{n}.
\end{eqnarray*}
\end{lemma}

\begin{proof}
For $x \in \cD_1$, we combine the upper bound for $Q_k(x)$ in Lemma~\ref{lem.qk-drift} with (\ref{eq:Qkx}), and obtain
\begin{eqnarray*}
(1+\la) \Delta Q_k(x) &\le& \beta_k(1-\la-u_{k+1}(x)) - \frac{Q_k(x)}{n (\la d)^{k-1}} \\
&& \mbox{} + \beta_k \la \exp(-dQ_k(x)/k n) + \frac{2}{\la d} \frac{Q_k(x)}{n(\la d)^{k-1}} \\
&\le& \beta_k(1-\la-u_{k+1}(x)) - \frac{Q_k(x)}{n (\la d)^{k-1}} \\
&& \mbox{} + \exp(-dQ_k(x)/k n) + \frac{2}{\la d} (1+2\eps) (1-\la).
\end{eqnarray*}
Here we used also the facts that $\beta_k<1$ and $\la<1$.
Using (\ref{ineq:5}) and (\ref{ineq:12}), we have $\frac{2}{\la d}(1+2\eps) \le 3/d$.  Also, by (\ref{ineq:6}),
$3/d \le \eps^2/7500 k^2 \le \eps/6$, which gives the required result.

To obtain the second inequality from its counterpart in Lemma~\ref{lem.qk-drift}, we need to show that
$$
\left( \frac{Q_{k-1}(x)}{n} \right)^2 \frac{1}{(\la d)^{k-3}} \le \frac{\eps}{6} (1-\la).
$$
From (\ref{eq:Qk-1x}), the left-hand side is at most $(1+5\eps)^2 (1-\la)^2 (\la d)^{k-1} \le 3 (1-\la)^2 (\la d)^{k-1}$,
using also (\ref{ineq:5}).  Also, from (\ref{ineq:7}), we have
$\eps \ge 100 k (1-\la) (\la d)^{k-1} \ge 18 (1-\la) (\la d)^{k-1}$.  Combining these inequalities gives the result.

To deduce the upper and lower bounds on $(1+\la) \Delta Q_j(x)$ from Lemma~\ref{lem.qj-drift}, we first note that
$2/\sqrt{\la d} \le \eps/50k$,
by (\ref{ineq:6}) and (\ref{ineq:12}).  This already gives the lower bound; for the upper bound, we also observe that
$$
\frac{d P_{k-1}(x)}{n} \le \frac{d Q_{k-1}(x)}{n} \le 2 (1-\la) d^{k-1} \le \frac{\eps}{50 k},
$$
using (\ref{eq:Qk-1x}) and (\ref{ineq:7}).
\end{proof}

\bigbreak

\section {Proof of Lemma~\ref{lem.cE}} \label{sec.E}

This section is devoted to the rather more complex proof of Lemma~\ref{lem.cE}.  First, we prove a statement
stronger than part~(1) of the lemma.  We set
$$
\cK = \left\{ x : u_{k+1} (x) \le \eps (1-\la) \mbox{ and } Q_k (x) \ge n(1-\frac{\eps}{3}) (1-\la) (\la d)^{k-1} \right\} \cap \cD_1;
$$
$$
W_\cK = \inf \{t \ge T_\cD : X_t \in \cK \}.
$$
Note that $\cK \subseteq \cE_0$, so to prove Lemma~\ref{lem.cE}(1) it suffices to prove that
$$
\Pr (W_\cK \land T_\cD^\dagger \ge T_\cD + m_\cE) \le 1/s_0.
$$
We prove this result on the assumption that $T_\cD=0$ (i.e., that $x_0 \in \cA_0 \cap \cB_0 \cap \cC_0 \cap\cD_0$).  The general case follows
immediately by applying the result for $T_\cD =0$ to the shifted process $(X'_t) = (X_{T_\cD+t})$, using the strong Markov property.  So our task is to show
that $\Pr (W_\cK \land T_\cD^\dagger \ge m_\cE) \le 1/s_0$, where $W_\cK = \inf \{ t \ge 0 : X_t \in \cK \}$, whenever
$X_0 = x_0$ a.s., for any $x_0 \in \cA_0 \cap \cB_0 \cap \cC_0 \cap\cD_0$.

We define the following further sets, hitting times and exit times.  We set
\begin{eqnarray*}
\cL^{k+1}_1 &=& \cD_1 \setminus \cK \\
&=& \left\{ x : u_{k+1} (x) > \eps (1-\la) \mbox{ or } Q_k (x) < n(1-\frac{\eps}{3}) (1-\la) (\la d)^{k-1} \right\} \cap \cD_1,
\end{eqnarray*}
$W_{\cL^{k+1}} = 0$ and
$W_{\cL^{k+1}}^\dagger = \inf \{ t \ge 0 : X_t \notin \cL^{k+1}_1\} = W_\cK \land T_\cD^\dagger$.
Also, for $j=k, \ldots, 1$, let
\begin{eqnarray*}
\cL^j_0 &=& \left\{ x : Q_j (x) \le n(1-\la )(\la d)^{j-1} (1- \frac{\eps}{6} - \frac{j \eps}{6k}) \right\} \cap \cL^{j+1}_1; \\
\cL^j_1 &=& \left\{ x : Q_j (x) \le n(1-\la )(\la d)^{j-1} (1- \frac{\eps}{6} - \frac{j \eps}{6k} + \frac{\eps}{24k}) \right\} \cap \cL^{j+1}_1; \\
W_{\cL^j} &=& \inf \{t \ge W_{\cL^{j+1}}: X_t \in \cL_0^j \};\\
W_{\cL^j}^\dagger &=& \inf \{ t \ge W_{\cL^j}: X_t \notin \cL_1^j )\}.
\end{eqnarray*}

Our goal is to show that $\Pr (W_{\cL^{k+1}}^\dagger < m_\cE) \ge 1-1/s_0$.
If $x_0 \in \cK$, then $W_{\cL^{k+1}}^\dagger = 0$ and we are done, so we may assume that $x_0 \notin \cK$, and hence that $x_0 \in \cL^{k+1}_1$.
Thus Lemma~\ref{lem.cE}(1) follows from the proposition below.

\begin{proposition} \label{prop.prop}
Let $x_0$ be any queue-lengths vector in $\cL^{k+1}_1$.  For a copy $(X_t)$ of the $(n,d,\la)$-supermarket process with
$X_0 = x_0$ a.s., we have
$$
\Pr (W_{\cL^{k+1}}^\dagger \ge m_\cE) \le 1/s_0.
$$
\end{proposition}

%\medbreak

For the proof of Proposition~\ref{prop.prop}, we fix a state $x_0 \in \cL^{k+1}_1$, and
work with a copy $(X_t)$ of the $(n,d,\la)$-supermarket process where $X_0 = x_0$ a.s.

Our general plan for proving Proposition~\ref{prop.prop} is as follows.  We suppose that the process $(X_t)$ stays inside
$\cL_1^{k+1} = \cD_1 \setminus \cK$ over the interval $[0,m_\cE)$, with the aim of showing that this event has low probability.  Observe that, if
$x \in \cL_1^{k+1} \setminus \cL^k_0$, then $u_{k+1}(x) > \eps (1-\la)$ and $Q_k(x) > n(1-\frac{\eps}{3}) (1-\la) (\la d)^{k-1}$.
This ``excess'' in $u_{k+1}$ would result in a downward drift in $Q_k(X_t)$, so if the
process does not exit $\cL_1^{k+1}$ quickly, then it enters $\cL^k_0$ quickly, and stays in $\cL^k_1$ throughout the interval $[0,m_\cE)$:
i.e., $W_{\cL^k}$ is small and $W_{\cL^k}^\dagger$ is large, with high probability.  This
means that $Q_k(X_t)$ maintains a ``deficit'' compared to $\tilde Q_k := n(1-\la)(\la d)^{k-1}$ until time $m_\cE$.
A deficit in $Q_k(X_t)$ would lead to a deficit in each $Q_j(X_t)$ in turn, compared to
$\tilde Q_j := n(1-\la)(\la d)^{j-1}$, for $j=k-1, k-2, \dots, 1$: each $W_{\cL^j}$ is small, and $W_{\cL^j}^\dagger$
is large, with high probability.  Finally, a deficit in $Q_1(X_t)$ compared to $\tilde Q_1 = n(1-\la)$ is unsustainable,
as this would lead to a drift down in the total number of customers over a long enough time interval to empty the entire system of customers.
This would entail exiting the set $\cB_1 \supseteq \cL_1^{k+1}$, a contradiction.

\begin{lemma}
\label{lem.qk}
\begin{enumerate}
\item
$\displaystyle \Pr (W_{\cL^k} \land W_{\cL^{k+1}}^\dagger \ge 12k\eps^{-1} n (1-\la)^{-1} ) \le 1/12s_0$.
\item
$\displaystyle \Pr (W_{\cL^k}^\dagger < m_\cE \le W_{\cL^{k+1}}^\dagger) \le 1/6s_0$.
\end{enumerate}
\end{lemma}

\begin{proof}
We apply Lemma~\ref{lem.drifts-down2} to the process $(X_t)$, with its natural filtration, and the function $F = Q_k$.
We set $h = (1-\frac{\eps}{3}) n (1-\la) (\la d)^{k-1}$ and $\rho = \frac{\eps}{24k} n (1-\la) (\la d)^{k-1}$; it follows
from (\ref{ineq:115}) that $\rho \ge 2$.  We also set
$\cS = \cL_1^{k+1}$ and $T^*=0$.  We note that $Q_k(x) \le c := kn$ for every $x$, and we take $m = 12k \eps^{-1} n (1-\la)^{-1}$,
and $s = m_\cE - 1$.  Then $T_0 = W_{\cL^{k+1}}^\dagger$, $T_1 = \inf \{ t : Q_k(X_t) \le h\}$ and
$T_2 = \inf \{ t > T_1 : Q_k(X_t) \ge h + \rho \}$, as in the lemma.

For $x \in \cL_1^{k+1}$ with $Q_k(x) > h$, we have $u_{k+1}(x) > \eps (1-\la)$ and $x \in \cD_1$.  So Lemma~\ref{lem.drifts} applies, and we have
\begin{eqnarray*}
\lefteqn{(1+\la) \Delta Q_k(x)} \\
&\le& \beta_k(1-\la-u_{k+1}(X_t)) - \frac{Q_k(x)}{n (\la d)^{k-1}} + \exp(-dQ_k(x)/kn) + \frac{\eps}{6} (1-\la) \\
&\le& (1-\la)(1-\eps) - (1-\la)(1-\frac\eps 3) \\
&&\mbox{} + \exp\left(-\frac{d(1-\frac \eps 6)(1-\la)(\la d)^{k-1}}{k}\right) + \frac{\eps}{6} (1-\la)\\
&\le& -\frac12 \eps (1-\la) + \exp\left(-\frac{(1-\frac \eps 6)d^k(1-\la)\la^{k-1}}{k}\right).
\end{eqnarray*}
From (\ref{ineq:5}), (\ref{ineq:103}) and (\ref{ineq:12}), we have that
$$
\frac{(1-\frac \eps 6)d^k(1-\la)\la^{k-1}}{k} \ge \frac{177}{100}\log n.
$$
Also $e^{-\frac{177}{100} \log n } \le 1/2n \le (1-\la) \eps/6$
for $n\ge 2$, using (\ref{ineq:1155}).  So
$$
(1+\la) \Delta Q_k(x) \le - \frac13 \eps (1-\la) \mbox{ and }
\Delta Q_k(x) \le - \frac16 \eps (1-\la) := -v,
$$
for such $x$.  Note that $m v = 2 c$.  Hence we may apply Lemma~\ref{lem.drifts-down2}.

As in earlier lemmas, we have $T_1 \land W_{\cL^{k+1}}^\dagger = W_{\cL^k} \land W_{\cL^{k+1}}^\dagger$, so we obtain
\begin{eqnarray*}
\Pr (W_{\cL^k} \land W_{\cL^{k+1}}^\dagger > m) &=& \Pr (T_1 \land T_0 > m) \\
&\le& \exp( - v^2 m /8) \\
&=& \exp(-\eps k n (1-\la) / 24) \\
&\le& \exp(- 2500 \eps^{-2} k^4 \log^2 n \,d^{k-2}) \\
&<& 1/12s_0,
\end{eqnarray*}
where we used (\ref{ineq:115}).

Also the events $W_{\cL^k}^\dagger < m_\cE \le W_{\cL^{k+1}}^\dagger$ and $T_2 < m_\cE \le W_{\cL^{k+1}}^\dagger$ coincide, and the
second is equivalent to $T_2 \le s < W_{\cL^{k+1}}^\dagger$ (since $s = m_\cE-1$).  So
\begin{eqnarray*}
\Pr ( W_{\cL^k}^\dagger < m_\cE \le W_{\cL^{k+1}}^\dagger) &=& \Pr (T_2 \le s < T_0) \\
&\le& s \exp( - \rho v) \\
&=& s \exp(- \eps^2 n (1-\la)^2 (\la d)^{k-1}/144k) \\
&\le& m_\cE \exp (- \frac{15}{4} k \log^2 n \, d^{k-2}) \\
&<& 1/6s_0,
\end{eqnarray*}
as required.  Here we used (\ref{ineq:8}) and (\ref{ineq:12}).
\end{proof}

The next lemma states that, if the process stays in some set $\cL^{j+1}_1$ for a long time, then it quickly enters the ``next'' set $\cL^j_0$.

\begin{lemma}
\label{lem.qj}
For each $j=k-1, \ldots,1$,
\begin{enumerate}
\item
$\displaystyle \Pr (W_{\cL^{j+1}}^\dagger \land W_{\cL^j} > W_{\cL^{j+1}} + \eps^{-1} n (1-\la)^{-1} ) \le 1/3ks_0$.
\item
$\displaystyle \Pr (W_{\cL^j}^\dagger < m_\cE \le W_{\cL^{j+1}}^\dagger ) \le 1/3ks_0$.
\end{enumerate}
\end{lemma}

\begin{proof} (Sketch)
This proof is very similar to that of earlier lemmas, and we mention only a few points.  As in Lemma~\ref{lem.cD}, we apply Lemma~\ref{lem.drifts-down2}
to the scaled process $Q'_j(x) = Q_j(x)/(\la d)^{(j-1)/2}$.  The key step is to show that, for $x \in \cL_1^{j+1}$ with $Q'_j \ge h =
n (1-\la) (\la d)^{(j-1)/2} (1-\frac{\eps}{6}-\frac{j\eps}{6k})$, we have
$\Delta Q'_j(x) \le - \frac{\eps}{24k} (1-\la) (\la d)^{(j+1)/2} := -v$.  The proof now proceeds as earlier ones.
The calculation for part (2) of the lemma goes as follows, with $\rho = \frac{\eps}{24k} n (1-\la) (\la d)^{(j-1)/2}$:
\begin{eqnarray*}
\Pr ( W_{\cL^j}^\dagger < m_\cE \le W_{\cL^{j+1}}^\dagger)
&\le& m_\cE \exp( - \rho v) \\
&=& m_\cE \exp \left( - \frac{\eps^2}{576k^2} n (1-\la)^2 (\la d)^j \right) \\
&\le& \frac12 s_0 \exp \left( - \frac{15}{16} \log^2 n \, d^{j-1}\right) \\
&\le& 1/3ks_0,
\end{eqnarray*}
where we used (\ref{ineq:8}) and (\ref{ineq:12}).
In the case $j=1$, this is the place where practically the full strength of (\ref{ineq:8}) is used.
\end{proof}

We now prove a hitting time lemma for $\|X_t\|_1$, the total number of customers in the system at time~$t$.
Let $W_\cM =  \min \{t \ge W_{\cL^1}: \|X_t\|_1 = 0\}$.

\begin{lemma} \label{lem.h}
$$
\Pr (W_{\cL^1}^\dagger \land W_{\cM} > W_{\cL^1} + 72 g \eps^{-1} n (1-\la)^{-1}) \le 1/12s_0.
$$
\end{lemma}

\begin{proof}
We apply Lemma~\ref{lem.drifts-down2}(i) to the chain $(X_t)$, with the filtration $(\cF_t)$, and the function $F(x) = \|x\|_1$, which takes jumps of size
at most~1.  Since $\cA_1(\ell,g) \supseteq \cL_1^1$, we have $\|X_0\|_1 \le c := 3gn$.
We also set $\cS = \cL_1^1$, $T^* = W_{\cL^1}$, $h = 0$ and $m = 72 g \eps^{-1} n (1-\la)^{-1}$.

Note that $\|X_{t+1}\|_1 - \|X_t\|_1$ is equal to $+1$ if the event at time $t$ is an arrival, with probability $\la/(1+\la)$, and equal
to $-1$ if the event is a potential departure from a non-empty queue, with probability $u_1(X_t)/(1+\la)$, so the
drift $\Delta \|x\|_1$ is equal to $\frac{1}{1+\la} (\la - u_1(x))$.  For $x \in \cL_1^1$, we have
$$
1-u_1(x) = \frac{Q_1(x)}{n} \le (1-\la) \left( 1 - \frac{\eps}{6}-\frac{\eps}{6k}+ \frac{\eps}{24k} \right) \le (1-\la) \left( 1-\frac{\eps}{6} \right).
$$
Hence, for $x \in \cL_1^1$,
$$
(1 + \la) \Delta \|x\|_1 = (1-u_1(x)) - (1-\la) \le -\frac{\eps}{6} (1-\la),
$$
and so
$\Delta \|x\|_1 \le - \frac{\eps}{12} (1-\la) := -v$.  Note that $v m = 2 c$.

Hence we may apply Lemma~\ref{lem.drifts-down2}(i).  With $T_0$ and $T_1$ as in that lemma, we have
$T_0 = W_{\cL^1}^\dagger$ and $T_1 = W_{\cM}$, so we conclude that
\begin{eqnarray*}
\Pr ( W_{\cL^1}^\dagger \land W_{\cM} \ge W_{\cL^1} + m ) &\le& \exp( - v^2m/8 ) \\
&=& \exp( - \eps g n (1-\la) / 16) \\
&\le& \exp \left( - 3750 \eps^{-2} g k^3 \log^2 n \, d^{k-2} \right) \\
&\le& 1/12s_0,
\end{eqnarray*}
as required.  Here we used (\ref{ineq:115}).
\end{proof}

%\bigbreak

We now combine Lemmas~\ref{lem.qk}, \ref{lem.qj} and \ref{lem.h} to prove Proposition~\ref{prop.prop}.

Observe that, for a copy $(X_t)$ of the $(n,d,\la)$-supermarket process starting in a state $x_0 \in \cL_1^{k+1}$, exactly one of the following occurs:
\begin{itemize}
\item[(a)] $W_{\cL^{k+1}}^\dagger < m_\cE$,
\item[(b)] not (a), and one of $W_{\cL^k}^\dagger$, $W_{\cL^{k-1}}^\dagger$, \dots, $W_{\cL^1}^\dagger$ is less than $m_\cE$,
\item[(c)] neither of the above, and $W_{\cL^k} > 12 k\eps^{-1} n (1-\la)^{-1}$,
\item[(d)] none of the above, and $W_{\cL^j} > W_{\cL^{j+1}} + \eps^{-1} n (1-\la)^{-1}$ for some $j =k-1, \dots, 1$,
\item[(e)] none of the above, and $W_{\cM} > W_{\cL^1} + 72 g \eps^{-1} n (1-\la)^{-1}$,
\item[(f)] none of the above, and $W_{\cM} < m_\cE \le W_{\cL^{k+1}}^\dagger$.
\end{itemize}
Indeed, if none of (a)--(e) occurs, then $W_{\cL^{k+1}}^\dagger \ge m_\cE$ since (a) fails, and also
\begin{eqnarray*}
W_\cM &=& W_{\cL^k} + \sum_{j=1}^{k-1} (W_{\cL^j} - W_{\cL^{j+1}}) + (W_\cM - W_{\cL^1}) \\
&\le& 12k\eps^{-1} n (1-\la)^{-1} + (k-1) \eps^{-1} n (1-\la)^{-1} + 72 g \eps^{-1} n (1-\la)^{-1} \\
&<& (13k + 72g) \eps^{-1} n (1-\la)^{-1} \\
&=& m_\cE.
\end{eqnarray*}

We now show that the probability of each of (b)--(f) is small.
For (b), Lemmas~\ref{lem.qk}(2) and \ref{lem.qj}(2) give that
\begin{eqnarray*}
\lefteqn{ \Pr (W_{\cL^k}^\dagger \land W_{\cL^{k-1}}^\dagger \land \cdots \land W_{\cL^1}^\dagger <
m_\cE \le W_{\cL^{k+1}}^\dagger) } \\
&\le& \Pr (W_{\cL^k}^\dagger < m_\cE \le W_{\cL^{k+1}}^\dagger)
+ \sum_{j=1}^{k-1} \Pr (W_{\cL^j}^\dagger < m_\cE \le W_{\cL^{j+1}}^\dagger) \\
&\le&  \frac{1}{6s_0} + (k-1) \frac{1}{3ks_0}  \\
&\le&  \frac{1}{2s_0},
\end{eqnarray*}
i.e., the probability of (b) is at most $1/2s_0$.
The probability of (c) is at most $1/12s_0$ by Lemma~\ref{lem.qk}(1).
The probability of (d) is at most $(k-1) \frac{1}{3ks} \le 1/3s_0$ by Lemma~\ref{lem.qj}(1).  The probability of (e) is at most
$1/12s_0$ by Lemma~\ref{lem.h}.  Finally, (f) is not possible, since at time $W_{\cM}$ there are no customers in the system, so
$Q_k(X_{W_{\cM}}) > n$, and thus $W_{\cM} \ge T_\cB^\dagger$, but also $T_\cB^\dagger \ge W_{\cL^{k+1}}^\dagger$ since
$\cL_1^{k+1} \subseteq \cD_1 \subseteq \cB_1$ by definition.

Thus the probability of (a), for a copy of the process starting in a state in $\cL_1^{k+1}$, is at least
$$
1 - \frac{1}{2s_0} - \frac{1}{12s_0} - \frac{1}{3s_0} - \frac{1}{12s_0} = 1- \frac{1}{s_0},
$$
which is what we need to prove Proposition~\ref{prop.prop}, and thus also Lemma~\ref{lem.cE}(1).

%\bigbreak

Now we move to the proof of Lemma~\ref{lem.cE}(2), stating that the exit time $T_\cE^\dagger$ is large with
high probability.  There are two things to prove here.  The first is that, if $X_t \in \cE_1$, then it is very
unlikely that, at time $t+1$, a customer arrives and creates a queue of length $k+1$.  The second is that, once $Q_k(X_t)$ has reached
$(1-3\eps) n (1-\la) (\la d)^{k-1}$, while $u_{k+1}(X_t)$ is at most $\eps (1-\la)$, $Q_k$ is unlikely to ``cross down against the
drift'' to $(1-4\eps) n (1-\la) (\la d)^{k-1}$.

For $t \ge 0$, let $L_t$ denote the event that, at time~$t$, a customer arrives and joins a queue of length at least~$k$ (equivalently,
the probability that the event is an arrival and that all the selected queues have length at least~$k$).
So $L_t$ is the event that $u_j(X_t) > u_j(X_{t-1})$ for some $j \ge k+1$.

\begin{lemma} \label{lem.no-new-k+1}
On the event that $X_t \in \cE_1$, we have $\Pr (L_{t+1} \mid \cF_t) < e^{- \log^2 n}$.
\end{lemma}

\begin{proof}
From the definition of $L_t$, we have
$\Pr (L_{t+1} \mid \cF_t) = \frac{\la}{1+\la} u_k(X_t)^d \le u_k(X_t)^d$.
For $x \in \cE_1$, we have $Q_k(x) \ge (1-4\eps) n (1-\la) (\la d)^{k-1}$ and
$Q_{k-1}(x) \le (1+5\eps) n (1-\la) (\la d)^{k-2} \le \frac13 \eps n (1-\la) (\la d)^{k-1}$, using (\ref{ineq:6}) with a lot to spare.
Therefore, by (\ref{eq:Qk}), (\ref{ineq:5}) and (\ref{ineq:12}), we have
$$
1-u_k(x) \ge \frac{Q_k(x)}{n} - \frac{Q_{k-1}(x)}{n} \ge (1-\frac{13}{3} \eps) (1-\la) (\la d)^{k-1} \ge \frac12 (1-\la) d^{k-1}.
$$
Hence, from (\ref{ineq:2}), we have, on the event that $X_t \in \cE_1$,
$$
u_k(X_t)^d \le \left( 1 - \frac12 (1-\la) d^{k-1} \right)^d \le
\exp\left( -\frac{1}{2} (1-\la) d^k \right) \le \exp (- \log^2 n),
$$
as required.
\end{proof}

%\bigbreak

Let $U^\dagger = \inf \{ t > T_\cE : u_{k+1}(X_t) > \eps (1-\la)\}$ and
$V^\dagger = \inf \{ t > T_\cE : Q_k(X_t) < (1-4\eps) n (1-\la) (\la d)^{k-1} \}$, and note that
$T_\cE^\dagger = T_\cD^\dagger \land U^\dagger \land V^\dagger$.  We thus have
$$
\Pr (T_\cE^\dagger \le s_0 < T_\cD^\dagger) \le \Pr (U^\dagger \le s_0 \land T_\cD^\dagger \land  V^\dagger)
 + \Pr (V^\dagger \le s_0 \land T_\cD^\dagger \land U^\dagger).
$$

We claim that each of these last two probabilities is at most $1/2s_0$.  For the first, we may apply Lemma~\ref{lem.no-new-k+1}.
Observe that, if $U^\dagger = t+1$, then the event $L_{t+1}$ occurs.  We now have:
\begin{eqnarray*}
\Pr (U^\dagger \le s_0 \land T_\cD^\dagger \land V^\dagger)
&=& \sum_{t=0}^{s_0-1} \Pr (U^\dagger = t+1 \le T_\cD^\dagger \land V^\dagger) \\
&=& \sum_{t=0}^{s_0-1} \Pr (U^\dagger = t+1 \mbox{ and } X_t \in \cE_1) \\
&=& \sum_{t=0}^{s_0-1} \E [ \1_{\{X_t \in \cE_1\}} \E( \1_{\{U^\dagger = t+1\}} \mid \cF_t ) ] \\
&\le& \sum_{t=0}^{s_0-1} \E [ \1_{\{X_t \in \cE_1\}} \E( \1_{L_{t+1}} \mid \cF_t ) ].
\end{eqnarray*}
By Lemma~\ref{lem.no-new-k+1}, each term is at most $e^{-\log^2 n}$, and so we have
$$
\Pr (U^\dagger \le s_0 \land T_\cD^\dagger \land V^\dagger) \le s_0 e^{-\log^2 n} < 1/2s_0,
$$
as claimed.

%\medbreak

To obtain the other required inequality, we apply the reversed version of Lemma~\ref{lem.drifts-down2}(2).
We consider the process $(X_t)$, with its natural filtration, the function $F = Q_k$, and the set
$\cS = \{ x : u_{k+1}(x) \le \eps (1-\la) \} \cap \cD_1$.
We set $h = (1-3\eps) n (1-\la) (\la d)^{k-1}$ and $\rho = \eps n (1-\la) (\la d)^{k-1} \ge 2$.  We also set
$s = s_0$ and $T^* = T_\cE$.
We have $T_0 = \inf \{ t \ge T_\cE : X_t \notin \cD_1 \mbox{ or } u_{k+1}(X_t) > \eps (1-\la) \}$,
so that $T_0 \ge T_\cD^\dagger \land U^\dagger$ (strict inequality occurs if $T_\cD^\dagger < T_\cE$).  Also
$T_1 = \inf \{ t \ge T_\cE : Q_k(X_t) \ge h\} = T_\cE$, and
$T_2 = \inf \{ t > T_\cE : Q_k(X_t) \le h - \rho \} = V^\dagger$.

Take $x \in \cS$ with $Q_k(x) \le h$.  As $x \in \cD_1$, we apply Lemma~\ref{lem.drifts} to obtain
\begin{eqnarray*}
(1+\la) \Delta Q_k(x) &\ge& \beta_k  (1-\la - u_{k+1}(x)) - \frac{Q_k(x)}{n(\la d)^{k-1}} - \frac{\eps}{6} (1-\la) \\
&\ge& \beta_k (1-\la)(1-\eps) - (1-\la)(1-3\eps) - \frac{\eps}{6}(1-\la) \\
&\ge& (1-\la) \left[ \left( 1- \frac{\eps}{2} \right)(1-\eps) - 1 + 3\eps - \frac{\eps}{6} \right] \\
&\ge& \eps (1-\la),
\end{eqnarray*}
where we also used (\ref{ineq:13}).  This yields $\Delta Q_k(x) \ge \frac12 \eps (1-\la) := v$, for such~$x$.

The reversed version of Lemma~\ref{lem.drifts-down2}(2) gives that
\begin{eqnarray*}
\Pr (V^\dagger \le s_0 \land T_\cD^\dagger \land U^\dagger )
&\le& \Pr (T_2 \le s_0 \land T_0) \\
&\le& s_0 \exp (-\rho v) \\
&=& s_0 \exp ( - \eps^2 n (1-\la)^2 (\la d)^{k-1} /2 ) \\
&\le& s_0 \exp( - 270 k^2 \log^2 n \, d^{k-2}) \\
&\le& 1/2s_0,
\end{eqnarray*}
as required.  Here we used (\ref{ineq:8}) and (\ref{ineq:12}).

This completes the proof of Lemma~\ref{lem.cE}.

\section{Proofs of Lemmas~\ref{lem.cG} and \ref{lem.cH}} \label{sec.GH}

In this section, we prove the final two of our sequence of lemmas.

\medbreak

\noindent
{\bf Proof of Lemma~\ref{lem.cG}}

\begin{proof}
Fix $j$ with $1 \le j \le k-1$, and consider the state of the process at the hitting time $T_{\cG^{j+1}}$.  The hitting time
$T_{\cG^j}$ is the first time $t \ge T_{\cG^{j+1}}$ that $Q_j(X_t)$ lies in the interval between
$\Big[ 1 - (4 + \frac{k-j-1/2}{k}) \eps \Big] n (1-\la) (\la d)^{j-1}$ and
$\Big[ 1 + (4 + \frac{k-j-1/2}{k}) \eps \Big] n (1-\la) (\la d)^{j-1}$.
Let $B_h$ be the event that
$Q_j(X_{T_{\cG^{j+1}}}) > \Big[ 1 + (4 + \frac{k-j-1/2}{k}) \eps \Big] n (1-\la) (\la d)^{j-1}$, and $B_\ell$ be the
event that $Q_j(X_{T_{\cG^{j+1}}}) < \Big[ 1 - (4 + \frac{k-j-1/2}{k}) \eps \Big] n (1-\la) (\la d)^{j-1}$.

For part~(1) of the lemma, we have to show that, on the event $B_h$, with high
probability $Q_j(X_t)$ enters the interval from above within time $m_\cG$, and also that, on the event $B_\ell$, with high
probability $Q_j(X_t)$ enters the interval from below within time $m_\cG$.  These two results are essentially the same, and we
give details only for the first.  Of course, we have nothing to prove on the event that $Q_j(X_{T_{\cG^{j+1}}})$ is already in
the interval.

We apply Lemma~\ref{lem.drifts-down2}(i) to $(X_t)$, with its natural filtration, and the scaled function
$F(x)=Q'_j(x) = Q_j(x)/(\la d)^{(j-1)/2}$.  We take $\cS = \cG_1^{j+1}$ and $T^* = T_{\cG^{j+1}}$.  We set
$$
h = \Big[ 1 + (4 + \frac{k-j-1/2}{k}) \eps \Big] n (1-\la) (\la d)^{(j-1)/2},
$$
and $m = m_\cG = 32 k \eps^{-1} n (1-\la)^{-1} (\la d)^{-1}$.  From (\ref{eq:Qjbound}), we have that $Q'_j(x) \le c := 2n$ for all $x$.
Also $T_0 = T_{\cG^{j+1}}^\dagger$ and $T_1 = \inf \{ t \ge T_{\cG^{j+1}} : Q'_j(X_t) \le h\}$.

For $x \in \cG_1^{j+1}$, we have
$$
Q_{j+1}(x) \le \left[ 1+ (4 + \frac{k-j-1}{k}) \eps \right] n (1-\la)(\la d)^j.
$$
(This follows from the specification
of $\cG_1^{j+1}$ for $j < k-1$, and since $\cG_1^k = \cE_1 \subseteq \cB_1$ for $j=k-1$.)  If also $Q'_j(x) \ge h$,
we have
$$
Q_j(x) \ge \left[ 1 + (4 + \frac{k-j-1/2}{k}) \eps \right] n (1-\la) (\la d)^{j-1}.
$$
Lemma~\ref{lem.drifts} applies since $x \in \cD_1$, so
\begin{eqnarray*}
(1+\la) \Delta Q_j(x)
&\le& -\la d \frac{Q_j(x)}{n}\left( 1 - \frac{\eps}{25k} \right) + \frac{Q_{j+1}(x)}{n} \\
&\le& - \Big[ 1 + (4 + \frac{k-j-1/2}{k}) \eps \Big] (1-\la) (\la d)^j \Big( 1 - \frac{\eps}{25k} \Big) \\
&&\mbox{} + \Big[1 + (4 + \frac{k-j-1}{k}) \eps \Big] (1-\la) (\la d)^j \\
&\le& -\frac1{4k} \eps (1-\la) (\la d)^j,
\end{eqnarray*}
and so $\Delta Q'_j(x) \le -\frac1{8k} \eps (1-\la) (\la d)^{(j+1)/2} := -v$.  Note that $v m_\cG \ge 2 c$.

Lemma~\ref{lem.drifts-down2}(i) now gives, using (\ref{ineq:115}),
\begin{eqnarray*}
\Pr (T_1 \land T_0 > T_{\cG^{j+1}} + m_\cG) &\le& \exp( -v^2 m_\cG /8) \\
&=& \exp(-\eps n (1-\la) (\la d)^j /16k) \\
&\le & \exp \left( - 3750 \eps^{-2} k^2 \log^2 n \right) \\
&\le& 1/2s_0.
\end{eqnarray*}
On the $T_{\cG^j}$-measurable event $B_h$, the stopping times
$T_1 \land T_0$ and $T_{\cG^j} \land T_0$ coincide, so we have
$$
\Pr ( B_h \cap \{ T_{\cG^j} \land T_{\cG^{j+1}}^\dagger > T_{\cG^{j+1}} + m_\cG\} ) \le 1/2s_0.
$$

Essentially exactly the same calculation gives
$$
\Pr ( B_\ell \cap \{ T_{\cG^j} \land T_{\cG^{j+1}}^\dagger > T_{\cG^{j+1}} + m_\cG\} ) \le 1/2s_0,
$$
and part~(1) of the lemma now follows, for this value of $j$.

%\medbreak

To prove part~(2) of the lemma, we need to show that, once $X_t$ has reached $\cG_0^j$, and while it remains in $\cG^{j+1}_1$, the process is unlikely
to leave the set $\cG_1^j$ quickly.  There are two separate things to prove: that $Q_j(X_t)$ is unlikely to cross against the drift from
$\Big[ 1 + (4 + \frac{k-j-1/2}{k}) \eps \Big] n (1-\la) (\la d)^{j-1}$ to $\Big[ 1 + (4 + \frac{k-j}{k}) \eps \Big] n (1-\la) (\la d)^{j-1}$ before time $s_0$,
and also that $Q_j(X_t)$ is unlikely to cross against the drift from
$\Big[ 1 - (4 + \frac{k-j-1/2}{k}) \eps \Big] n (1-\la) (\la d)^{j-1}$ to $\Big[ 1 - (4 + \frac{k-j}{k}) \eps \Big] n (1-\la) (\la d)^{j-1}$ before time $s_0$.
Again, the two calculations required here are essentially identical, and we shall concentrate on the first.

We apply Lemma~\ref{lem.drifts-down2}(ii), again for the process $(X_t)$ with its natural filtration, and the scaled function
$F(x) = Q_j(x)/(\la d)^{(j-1)/2}$.  We take the same values of parameters as above, and additionally set
$\rho = \frac{\eps}{2k} n (1-\la) (\la d)^{(j-1)/2}$ and $s=s_0$.  Here $T_2 = \inf\{ t > T_1 : Q'_j(X_t) \ge h + \rho\}$,
and we have, using (\ref{ineq:8}) and (\ref{ineq:12}),
\begin{eqnarray*}
\Pr (T_2 \le s_0 < T_{\cG^{j+1}}^\dagger) &\le& s_0 \exp(-\rho v) \\
&=& s_0 \exp ( - \eps^2 n (1-\la)^2 (\la d)^j /16k ) \\
&\le& s_0 \exp \left( - \frac{135}{4} k \log^2 n \right) \\
&\le& 1/2s_0.
\end{eqnarray*}
Setting $U_2 = \inf \{ t > T_1 : Q_j(X_t) \le \Big[ 1 - (4 + \frac{k-j}{k}) \eps \Big] n (1-\la) (\la d)^{j-1} \}$,
we have, similarly, $\Pr (U_2 \le s_0 < T_{\cG^{j+1}}^\dagger) \le 1/2s_0$.

The events $T_2 \land U_2 \le s_0 < T_{\cG^{j+1}}^\dagger$ and $T_{\cG^j}^\dagger \le s_0 < T_{\cG^{j+1}}^\dagger$ coincide,
so
\begin{eqnarray*}
\Pr (T_{\cG^j}^\dagger \le s_0 < T_{\cG^{j+1}}^\dagger)
&\le& \Pr (T_2 \le s < T_0) + \Pr (U_2 \le s < T_0) \\
&\le&\frac{1}{2s_0} + \frac{1}{2s_0} = \frac{1}{s_0},
\end{eqnarray*}
as required for part~(2) for this value of $j$.
\end{proof}

\bigbreak

\noindent
{\bf Proof of Lemma~\ref{lem.cH}}

\begin{proof}
We first prove part~(1).  For $i=1, \dots, n$, let $N_i$ be the number of potential departures from queue~$i$ over the time period
between $T_{\cG^1}$ and $T_{\cG^1}+ m_\cH$, so $N_i$ is a binomial random variable with parameters $(m_\cH, 1/n(1+\la))$.
Recall that $L_t$ is the event that, at time $t$, a customer arrives
and joins a queue of length $k$ or longer, and observe that
\begin{eqnarray*}
\lefteqn{ \Pr (T_\cH \land T_{\cG^1}^\dagger \ge T_{\cG^1} + m_\cH) } \\
&\le& \Pr \Big( \bigcup_{t=T_{\cG^1}+1}^{T_{\cG^1} + m_\cH} ( L_t \cap \{ X_{t-1} \in \cG_1^1 \} ) \Big)
+ \Pr \Big( \exists i , N_i < 3 \ell \Big).
\end{eqnarray*}
Indeed, at time $T_{\cG^1}$, the
process is in $\cA_1(\ell,g)$, and so there is no queue with more than $3 \ell$ customers in it at that time.
If there are at least $3 \ell$ potential departures from each queue over the time interval, and
$\bigcup_{t=T_{\cG^1}+1}^{T_{\cG^1} + m_\cH} L_t$ does not occur, then by time $T_{\cG^1} + m_\cH$, every queue is reduced to
length at most~$k$, and no new queue of length $k+1$ is created before $T_{\cG^1} + m_\cH$.

Now let $(X'_t) = (X_{T_{\cG^1}+t})$, $(\cF'_t) = (\cF_{T_{\cG^1}+t})$ and $L'_t = L_{T_{\cG^1}+t}$.  We have:
\begin{eqnarray*}
\Pr \Big( \bigcup_{t=T_{\cG^1}+1}^{T_{\cG^1} + m_\cH} ( L_t \cap \{ X_{t-1} \in \cG_1^1 \} ) \Big)
&=& \Pr \Big( \bigcup_{t=1}^{m_\cH} ( L'_t \cap \{ X'_{t-1} \in \cG_1^1 \} )  \Big)\\
&\le & \sum_{t=1}^{m_\cH} \Pr ( L'_t \cap \{ X'_{t-1} \in \cG_1^1 \} ) \\
&=& \sum_{t=1}^{m_\cH} \E \big[ \1_{\{X'_{t-1} \in \cG_1^1\}} \E [ \1_{L'_t} \mid \cF'_{t-1} ] \big] \\
&\le& m_\cH e^{-\log^2 n} \\
&\le& 1/2s_0,
\end{eqnarray*}
where we used the strong Markov property, and Lemma~\ref{lem.no-new-k+1}.

Recall that $m_\cH = n (8\ell + 32\log^2 n)$, so that
the mean $\mu$ of each $N_i$ is $m_\cH / n(1+\la) \ge 4\ell + 16\log^2 n$.  By (\ref{eq.bin-lower}),
with $\eps = 1/4$, we have
$$
\Pr (N_i \le 3 \ell) \le \Pr (N_i \le \frac34 \mu) \le e^{-\mu/32} \le e^{-\frac12 \log^2 n}
$$
for each $i$.  Thus the probability that there are fewer than $3 \ell$ departures from any queue over the interval
from $T_{\cG^1}$ to $T_{\cG^1} + m_\cH$ is at most $n e^{-\frac12 \log^2 n} < 1/2s_0$, and part (1) follows.

For part~(2), as above we have
$$
\Pr \Big( \bigcup_{t=T_{\cG^1}+1}^{T_{\cG^1} + s_0} ( L_t \cap \{ X_{t-1} \in \cG_1^1 \} ) \Big)
\le s_0 e^{-\log^2 n}
\le 1/s_0.
$$
Thus $\Pr (T_\cH^\dagger \le s_0 < T_{\cG_1}^\dagger)$ is at most the probability that $X_t$ exits the set $\cH_1$ before time $T_{\cG_1}^\dagger \land s_0$, necessarily by the creation of a new queue of length $k+1$, is at most $1/s_0$, as required.
\end{proof}

\section{The sets $\cH$, $\cI$ and $\cN$} \label{sec.H-I-N}

One goal of this section is to show that $\cH \subseteq \cN$, thus completing the proof of Theorem~\ref{thm.technical}: see Corollary~\ref{cor.eqm} and the
remarks after.  We also show that the set $\cN$ is ``path-connected'', a fact we shall need in the next section.

We continue to assume that $n$, $d$, $k$, $\la$ and $\eps$ satisfy the hypotheses of Theorem~\ref{thm.technical}.
For this section, it will be important to be explicit about the fact that the various sets we have defined depend on the value of
the parameter $\eps$: accordingly, we shall refer to our sets as (e.g.) $\cH^\eps$, $\cI^\eps$ and $\cN^\eps$.
(Note that these sets do not depend on the values of $g$ and $\ell$ used in defining earlier sets in the sequence.)

%\medbreak

By definition, the set $\cH^\eps$ consists of those queue-lengths vectors $x$ satisfying all of the following:
\begin{eqnarray*}
\| x \|_\infty &\le& 3\log^2 n (1-\la)^{-1}, \quad
\| x \|_1 \,\le\, 6n (1-\la)^{-1}, \\
Q_k(x) &\le& (1+2\eps) n (1-\la) (\la d)^{k-1}, \quad
P_{k-1}(x) \,\le\, 3k n (1-\la) (\la d)^{k-2}, \\
Q_{k-1}(x) &\le& (1+5\eps) n (1-\la) (\la d)^{k-2}, \quad
u_{k+1}(x) \,\le\, \eps (1-\la), \\
Q_k(x) &\ge& (1-4\eps) n (1-\la) (\la d)^{k-1}, \quad
u_{k+1}(x) \,=\, 0, \\
Q_j(x) &\ge& \Big[ 1-(4+\frac{k-j}{k})\eps \Big] n (1-\la) (\la d)^{j-1} \quad (1\le j \le k-1), \\
Q_j(x) &\le& \Big[ 1+(4+\frac{k-j}{k})\eps \Big] n (1-\la) (\la d)^{j-1} \quad (1\le j \le k-1).
\end{eqnarray*}
Evidently many of these conditions are redundant.  The condition that $u_{k+1}(x) = 0$ implies not only that $u_{k+1}(x) \le \eps (1-\la)$, and
that $\| x \|_\infty \le 3 \log^2 n (1-\la)^{-1}$, but also that $\| x \|_1 \le kn < 6n (1-\la)^{-1}$.  The upper bound on $P_{k-1}(x)$
is implied by $P_{k-1}(x) \le Q_{k-1}(x) \le (1+5\eps) n (1-\la) (\la d)^{k-2}$.  Also, the earlier upper bound on $Q_{k-1}(x)$ is weaker than the
final one.  Thus $\cH^\eps$ consists of those queue-lengths vectors $x$ satisfying all of:
\begin{eqnarray*}
u_{k+1}(x) &=& 0, \\
Q_k(x) &\le& (1+2\eps) n (1-\la) (\la d)^{k-1}, \\
Q_j(x) &\ge& \Big[ 1-(4+\frac{k-j}{k})\eps \Big] n (1-\la) (\la d)^{j-1} \quad (1\le j \le k), \\
Q_j(x) &\le& \Big[ 1+(4+\frac{k-j}{k})\eps \Big] n (1-\la) (\la d)^{j-1} \quad (1\le j \le k-1).
\end{eqnarray*}

%\medbreak

Similarly, $\cI^\eps$ consists of those queue-lengths vectors $x$ satisfying:
\begin{eqnarray*}
\| x \|_\infty &\le& \log^2 n (1-\la)^{-1}, \quad
\| x \|_1 \,\le\, 2n (1-\la)^{-1}, \\
Q_k(x) &\le& (1+\eps) n (1-\la) (\la d)^{k-1}, \quad
P_{k-1}(x) \,\le\, 2k n (1-\la) (\la d)^{k-2}, \\
Q_{k-1}(x) &\le& (1+4\eps) n (1-\la) (\la d)^{k-2}, \quad
u_{k+1}(x) \,\le\, \eps (1-\la), \\
Q_k(x) &\ge& (1-3\eps) n (1-\la) (\la d)^{k-1}, \quad
u_{k+1}(x) \,=\, 0, \\
Q_j(x) &\ge& \Big[ 1-(4+\frac{k-j-1/2}{k})\eps \Big] n (1-\la) (\la d)^{j-1} \quad (1\le j \le k-1), \\
Q_j(x) &\le& \Big[ 1+(4+\frac{k-j-1/2}{k})\eps \Big] n (1-\la) (\la d)^{j-1} \quad (1\le j \le k-1).
\end{eqnarray*}
Removing redundancies, $\cI^\eps$ consists of those vectors $x$ such that:
\begin{eqnarray*}
u_{k+1}(x) &=& 0, \\
Q_k(x) &\le& (1+\eps) n (1-\la) (\la d)^{k-1}, \\
Q_k(x) &\ge& (1-3\eps) n (1-\la) (\la d)^{k-1}, \\
Q_j(x) &\ge& \Big[ 1-(4+\frac{k-j-1/2}{k})\eps \Big] n (1-\la) (\la d)^{j-1} \quad (1\le j \le k-1), \\
Q_j(x) &\le& \Big[ 1+(4+\frac{k-j-1/2}{k})\eps \Big] n (1-\la) (\la d)^{j-1} \quad (1\le j \le k-1).
\end{eqnarray*}
Rather crudely, we have $\cI^\eps \subseteq \cH^\eps \subseteq \cI^{2\eps}$, for any $\eps > 0$.

Now we bring the sets $\cN^\eps$ into the picture.
Recall that
\begin{eqnarray*}
\lefteqn{\cN^\eps = \{ x: u_{k+1}(x) = 0,} \\
&& \quad (1 - 5\eps) (1-\la) (\la d)^{j-1} \le 1- u_j (x) \le (1 + 5\eps) (1-\la) (\la d)^{j-1} \\
&& \qquad (j=1, \dots, k) \}.
\end{eqnarray*}

\begin{lemma} \label{lem.H-N-H}
For any $\eps > 0$, $\cH^\eps \subseteq \cN^\eps \subseteq \cH^{3\eps}$.
\end{lemma}

\begin{proof}
Take $x \in \cH^\eps$: we check that $x \in \cN^\eps$.  We do have $u_{k+1}(x) = 0$.  Note that
$\frac{1}{\sqrt{\la d}} \le \frac{\eps}{135k}$, from (\ref{ineq:6}) and (\ref{ineq:12}).  Using also (\ref{eq:Qj+1}), and that
$Q_j(x) \le 2 n (1-\la) (\la d)^{j-2}$ by (\ref{ineq:5}), we have
for $j=1, \dots, k-1$,
\begin{eqnarray*}
1-u_j(x) &\ge& \frac{Q_j(x)}{n} - 2 \sqrt{\la d} \frac{Q_{j-1}(x)}{n} \\
&\ge & \Big[ 1 - (4+\frac{k-j}{k})\eps \Big] (1-\la) (\la d)^{j-1} - 4 \sqrt{\la d} (1-\la) (\la d)^{j-2} \\
&\ge & (1-\la) (\la d)^{j-1} \Big[ 1 - (5 -\frac{1}{k})\eps - 4 \frac{\eps}{135 k} \Big] \\
&\ge & (1 - 5\eps) (1-\la) (\la d)^{j-1}.
\end{eqnarray*}
We also have
\begin{eqnarray*}
1-u_j(x) &\le& \frac{Q_j(x)}{n} \\
&\le& \Big[ 1 + (4 + \frac{k-j}{k}) \eps \Big] (1-\la) (\la d)^{j-1} \quad (j=1,\dots, k-1) \\
&\le& (1 + 5\eps) (1-\la) (\la d)^{j-1}, \\
1-u_k(x) &\le& \frac{Q_k(x)}{n} \frac{1}{\beta_k} \le (1 + 5\eps) (1-\la) (\la d)^{k-1},
\end{eqnarray*}
using (\ref{ineq:13}).
So $x \in \cN^\eps$, as required.

%\medbreak

Now take $x \in \cN^\eps$; we check that $x \in \cH^{3\eps}$.  We do have $u_{k+1}(x) = 0$.  Also
\begin{eqnarray*}
Q_k(x) &\le& n \sum_{j=1}^k (1-u_j(x)) \\
&\le& (1 + 5\eps) n (1-\la) \sum_{j=1}^k (\la d)^{j-1} \\
&\le& (1 + 5\eps) n (1-\la) (\la d)^{k-1} \big(1 + \sum_{i=1}^{k-1} \frac{1}{(\la d)^i} \big) \\
&<& (1 + 6\eps) n (1-\la) (\la d)^{k-1},
\end{eqnarray*}
where we used (\ref{ineq:14}).
For $1\le j \le k-1$, we have, using (\ref{eq:gamma-bounds}) and (\ref{ineq:14}),
\begin{eqnarray*}
Q_j(x) &=& n \sum_{i=1}^j \gamma_{j,i} (1-u_i(x)) \\
&\le& n (1-u_j(x)) + n \sum_{i=1}^{j-1} i (\la d)^{(j-i)/2} (1-u_i(x)) \\
&\le& (1 + 5 \eps) n (1-\la) (\la d)^{j-1} + n k \sum _{i=1}^{j-1} (\la d)^{(j-i)/2} (1+5\eps) (1-\la) (\la d)^{i-1} \\
%&=& (1 + 5 \eps) n (1-\la) (\la d)^{j-1} +  2 n k (1-\la) \sum_{i=1}^{j-1} (\la d)^{(j+i-2)/2} \\
&\le& n (1-\la) (\la d)^{j-1} \left[ 1 + 5 \eps + 2 k \sum_{i=1}^{j-1} \frac{1}{(\la d)^{(j-i)/2}} \right] \\
%&\le& n (1-\la) (\la d)^{j-1} \left[ 1 + 5 \eps + \frac{4}{135} \eps \right] \\
&\le& (1 + 6 \eps) n (1-\la) (\la d)^{j-1}.
\end{eqnarray*}
For $1\le j \le k-1$, we have
$Q_j(x) \ge n (1-u_j(x)) \ge (1-5\eps) n (1-\la) (\la d)^{j-1}$,
while also, using (\ref{ineq:13}),
\begin{eqnarray*}
Q_k(x) &\ge& n \beta_k (1-u_k(x)) \\
&\ge& \beta_k (1-5\eps) n (1-\la) (\la d)^{k-1} \\
&\ge& (1-6\eps) n (1-\la) (\la d)^{k-1}.
\end{eqnarray*}
So indeed $x \in \cH^{3\eps}$.
\end{proof}

Hence Theorem~\ref{thm.stay-in-H} and Corollary~\ref{cor.eqm} hold with $\cH$ replaced by $\cN$.
Moreover, we have the following analogue of Theorem~\ref{thm.start-in-I-stay-in-H}.
Here and subsequently, we require that $\eps \le 1/60$, to ensure that the conditions of Theorem~\ref{thm.technical} are met with
$\eps$ replaced by $6\eps$.

\begin{theorem} \label{thm.start-in-N-stay-in-N}
Take $\eps \le 1/60$, and $x_0 \in \cN^\eps$.  Suppose $(X_t)$ is a copy of the $(n,d,\la)$-supermarket process in which $X_0 = x_0$ almost surely.
Then
$$
\Pr (X_t \in \cH^{6\eps} \mbox{ for all } t \in [0,s_0]) \ge 1 - (k+5)/s_0,
$$
and hence
$$
\Pr (X_t \in \cN^{6\eps} \mbox{ for all } t \in [0,s_0]) \ge 1 - (k+5)/s_0.
$$
\end{theorem}

\begin{proof}
If $x_0 \in \cN^\eps \subseteq \cH^{3\eps} \subseteq \cI^{6\eps}$, then, by Theorem~\ref{thm.start-in-I-stay-in-H},
with probability at least $1-(k+5)/s_0$, for all times $t \in [0,s_0]$, $X_t \in \cH^{6\eps} \subseteq \cN^{6\eps}$,
as required.
\end{proof}

%\bigbreak

We say two queue-lengths vectors are {\em adjacent} if they differ by one customer in one queue.  A {\em path} of length $m$ between two
vectors $x$ and $y$ is a sequence $x=x_0x_1\cdots x_m=y$ of queue-lengths vectors, with each pair $(x_i,x_{i+1})$
adjacent.  The path is said to lie in a set $\cS$ if each $x_i$ is in $\cS$.

\begin{lemma} \label{lem.path-connected}
Between any two queue-lengths vectors in $\cN^\eps$, there is a path of length at most $4 n (1-\la) (\la d)^{k-1}$ lying in $\cN^\eps$.
\end{lemma}

\begin{proof}
For $j =1, \dots, k$, set $u^*_j = \frac{1}{n} \lfloor n (1- (1-\la) (\la d)^{j-1}) \rfloor$.  Now set
$$
\cP = \left\{ z: u_{k+1}(z) = 0, u_j(z) = u^*_j \mbox{ for } j=1, \dots, k \right\}.
$$
In other words, $\cP$ consists of those queue-lengths vectors $z$ with no queues of length greater than $k$, and such that the number
$n u_j(z)$ of queues of length at least $j$ is equal to $n u^*_j = \lfloor n (1- (1-\la) (\la d)^{j-1}) \rfloor$, for each $j = 1, \dots, k$.
We think of $\cP$ as forming the ``centre'' of the set $\cN^\eps$.  Our plan is to show that every queue-lengths vector in $\cN^\eps$
is joined to some vector in $\cP$ by a short path lying entirely in $\cN^\eps$, and then to show that every two vectors in $\cP$ are
connected by a short path, again lying entirely in $\cN^\eps$.

Let $x$ be a queue-lengths vector in $\cN^\eps$.  We first show that there is a path within $\cN^\eps$ of length at most
$6k\eps n (1-\la) (\la d)^{k-1}$ from $x$ to a vector $x'$ in $\cP$, of the form
$x = x^{(k+1)} \cdots x^{(k)} \cdots x^{(k-1)} \cdots \cdots x^{(0)} = x'$, where:
\begin{itemize}
\item [(a)] $u_{k+1}(y) = 0$ for all vectors $y$ on the path,
\item [(b)] for each $j=1, \dots, k$:
\begin{itemize}
\item all the vectors $y$ on the section of the path from $x$ up to $x^{(j+1)}$ satisfy $u_j(y) = u_j(x)$,
\item all the vectors $y$ on the section of the path from $x^{(j)}$ to $x'$ satisfy $u_j(y) =u^*_j$,
\item on the section of path between $x^{(j+1)}$ and $x^{(j)}$, the value of $u_j(y)$ changes monotonically in steps of size $1/n$
from $u_j(x)$ to $u^*_j$.
\end{itemize}
\end{itemize}
Such a path will certainly lie in $\cN^\eps$, since, for each $j$ and each $y$ on the path, $u_j(y)$ lies between $u_j(x)$ and $u^*_j$.

To establish the existence of such a path, we explain how to construct each section individually.  For each $j$, assuming only that $x^{(j+1)}$ is in
$\cN^\eps$, we show that there is a sequence of adjacent vectors $y$, starting with $x^{(j+1)}$, so that $n u_j(y)$ changes monotonically from
$n u_j(x)$ to $n u^*_j$ in steps of size $1$ along the sequence, while each other $u_i$ is constant along the sequence.  If $u_j(x^{(j+1)}) > u^*_j$, the number
of queues of length exactly $j$ in $x^{(j+1)}$ is
\begin{eqnarray*}
\lefteqn{ n (u_j(x^{j+1}) - u_{j+1}(x^{(j+1)}))} \\
&\ge& n [u_j(x^{j+1}) - u^*_j] + \lfloor n (1- (1-\la) (\la d)^{j-1}) \rfloor -n [ 1- (1-5\eps)(1-\la)(\la d)^j ] \\
&\ge& n [u_j(x^{j+1}) - u^*_j],
\end{eqnarray*}
so we may form a suitable sequence by taking $n [u_j(x^{j+1}) - u^*_j]$ queues of length exactly $j$ and removing one customer from each of these
queues in turn.  Similarly, if $u_j(x^{j+1}) < u^*_j$, the number of queues of length exactly $j-1$ in $x^{(j+1)}$ is
\begin{eqnarray*}
\lefteqn{ n [u_{j-1}(x^{j+1}) - u_j(x^{(j+1)})]} \\
&\ge& n [u^*_j - u_j(x^{j+1})] - \lfloor n (1- (1-\la) (\la d)^{j-1}) \rfloor + n [ 1- (1+5\eps)(1-\la)(\la d)^{j-2} ] \\
&\ge& n [u^*_j - u_j(x^{j+1})],
\end{eqnarray*}
so we may form a suitable sequence by taking $n [u^*_j - u_j(x^{j+1})]$ queues of length exactly $j-1$ and adding one customer to each of these queues in
turn.  Neither of these operations affects $u_i$, the proportion of queues of length at least $i$, for any value of $i$ other
than $j$.  The section of path from $x^{(j+1)}$ and $x^{(j)}$ has length $n |u_j(x^{j+1}) - u^*_j| \le 5\eps n (1-\la) (\la d)^{j-1}$.

Thus we can construct a path from $x$ to $x'$ staying within $\cN^\eps$, and this path has length at most
$$
\sum_{j=1}^k \lceil 5\eps n (1-\la) (\la d)^{j-1} \rceil \le 6\eps n (1-\la) (\la d)^{k-1},
$$
where we used (\ref{ineq:14}).

%\medbreak

We now show that there is a path of length at most $3n (1-\la) (\la d)^{k-1}$ between any two queue-lengths vectors $x$ and $x'$ in
the ``centre'' $\cP$ of $\cN^\eps$, staying within $\cN^\eps$.  This path will be of the form $x = x_{(0)} \cdots x_{(1)} \cdots \cdots x_{(k)} = x'$.
Our path will have the following properties, for each $j =1, \dots, k-1$:
\begin{itemize}
\item [(a)] for every vector $y$ on the path from $x_{(j)}$ to $x_{(k)}$, the set of queues of length $j-1$ is the same in $y$ as in $x'$; thus the
set of all queues of length at most $j-1$ is fixed from $x_{(j)}$ to the end of the path $x_{(k)}$, and hence so is their number
$n(1-u_j(x')) = n(1-u^*_j)$;
\item [(b)] along the path from $x$ to $x_{(j-1)}$, $u_j(y)$ decreases monotonically from $u_j^*$ to a value $u_j(x_{(j-1)}) \ge u^*_j (1 - \eps)$;
\item [(c)] along the section of path from $x_{(j-1)}$ to $x_{(j)}$, $u_j(y)$
stays between $u_j(x_{(j-1)})$ and $u^*_j + 1/n$, until it becomes equal to $u^*_j$ at $x_{(j)}$.
\end{itemize}

With reference to (b) above, we shall actually show that: (b')~for each $j > i$,
as $y$ goes from $x_{(i-1)}$ and $x_{(i)}$ along the path, the value of $n u_j(y)$ decreases monotonically by at most $n(1-u^*_i)$.
Thus, for $j = 2, \dots, k$, the total decrease in $u_j(y)$ as $y$ goes from $x$ to $x_{(j-1)}$ is at most
$$
\sum_{i=1}^{j-1} n (1-u^*_i) \le n (1-\la) \sum_{i=1}^{j-1} (\la d)^{j-1} < \frac12 \eps n (1-\la) (\la d)^{j-1} < \eps n (1-u^*_j),
$$
where we used (\ref{ineq:14}).  So (b') implies (b).

To construct the section of path between $x = x_{(0)}$ and $x_{(1)}$, we consider queues of length~0: in both $x$ and $x'$, the number of these empty
queues is equal to $n (1-u^*_1) = \lceil n (1-\la) \rceil$.  We let $K_{i_1}, \dots, K_{i_r}$ be the queues that are empty in $x$ but not in
$x'$, and $K_{\ell_1}, \dots, K_{\ell_r}$ be the queues that are empty in $x'$ but not in $x$ -- so $r \le n(1-u^*_1)$.
The section of path from $x_{(0)}$ to $x_{(1)}$ is constructed by: adding a customer to queue $K_{i_1}$, emptying out queue $K_{\ell_1}$, adding a customer
to queue $K_{i_2}$, emptying out queue $K_{\ell_2}$, and so on.  In this way, for every vector on this section of path, the number of empty queues is
within~1 of $n (1-u^*_1)$.  Meanwhile, for
$j \ge 2$, as we go along the path from $x_{(0)}$ to $x_{(1)}$, the value of $u_j(y)$ may be decreased (since some of the queues $K_{\ell_q}$
may have length greater than $j$ in $x_{(0)}$: we decrease the lengths of these queues to~0, without creating any new queues of length at least $j$),
but by at most $r \le n (1-u^*_1)$, as required.

We now proceed in the same way for the set of queues of each length $j=1, 2, \dots, k-1$ in turn.  We describe the construction of the
section of path from $x_{(j)}$ to $x_{(j+1)}$.  In $x_{(j)}$, the set of queues of each length less than $j$ is the same as in $x'$, and also
we have $u_{j+1}(x_{(j)}) \le u^*_{j+1}$, by the properties of the construction up to this point.
Let $K_{i_1}, \dots, K_{i_r}$ be the queues that have length exactly $j$ in $x_{(j)}$ but are
longer in $x'$, and $K_{\ell_1}, \dots, K_{\ell_s}$ be the queues that have length $j$ in $x'$ but are longer in $x_{(j)}$.  Note that
$r - s = n [ (1 - u_{j+1}(x_{(j)})) - (1 - u^*_{j+1}) ] \ge 0$, and also that $s$ is at most the number of queues of length~$j$ in $x'$, which is
at most $n (1-u^*_{j+1})$.  The section of path from $x_{(j)}$ to $x_{(j+1)}$ is
constructed by: adding a customer to queue $K_{i_1}$, reducing the length of queue $K_{\ell_1}$ to $j$, adding a customer to queue $K_{i_2}$,
reducing the length of queue $K_{\ell_2}$ to $j$, and so on.  At the end, we add a customer to each of the remaining queues $K_{i_{s+1}}, \dots K_{i_r}$
in turn.  In this way, for each intermediate vector $y$ on this section of path, the number $n u_{j+1}(y)$ of queues of length at least $j+1$ is at least
its value $n u_{j+1} (x_{(j)}$ at the beginning of this section of path, and at most
$n u^*_{j+1} + 1$ (this can be achieved if $u_{j+1}(x_{(j)}) = u^*_{j+1}$, in which case $n u_{j+1}(y)$ alternates between $n u^*_{j+1}$ and $n u^*_{j+1} +1$
along the section of path).  For $h > j+1$, $u_h(y)$ is decreased as go along this section of the path, by at most $s \le n (1-u^*_{j+1})$, as
required.

Along this path from $x$ to $x'$ as a whole, for each queue $i$, its length changes monotonically from $x(i)$ to $x'(i)$.  So the total length
of the path is at most
\begin{eqnarray*}
\sum_{i=1}^n |x(i) - x'(i)| &\le& \sum_{i=1}^n (k-x(i)) + (k-x'(i)) \\
&=& n \left( \sum_{i=1}^k (1-u_k(x)) + \sum_{i=1}^k (1-u_k(x')) \right) \\
&=& 2n \sum_{i=1}^k (1-u^*_i) \\
&\le& \frac{5}{2} n \sum_{i=1}^k (1-\la) (\la d)^{i-1},
%&\le& 3n (1-\la) (\la d)^{k-1},
\end{eqnarray*}
which is at most $3n (1-\la) (\la d)^{k-1}$ by (\ref{ineq:14}).

Thus there is a path between any pair of states in $\cN^\eps$ of length at most
$$
3n (1-\la) (\la d)^{k-1} + 12\eps n (1-\la) (\la d)^{k-1} \le 4 n (1-\la) (\la d)^{k-1},
$$
as claimed.
\end{proof}

\section{Rapid Mixing} \label{sec.mixing}

Our aims in this section are to prove a variety of results about rapid mixing of the $(n,d,\la)$-supermarket process.  We continue to
assume that the parameters $n$, $d$, $k$, $\la$ and $\eps$ of the model satisfy the conditions of Theorem~\ref{thm.technical} (or equivalently
(\ref{ineq:2})--(\ref{ineq:8})).  For this section, we make the stronger assumption that
$\eps \le 1/60$, so that $(n,d,k,\la,6\eps)$ also satisfies the
conditions.

We first consider two copies of the process starting in adjacent
states in $\cA_0(\ell,g)$, coupled according to the coupling referred to in Lemma~\ref{lem.coupling-distance}.
The proof partly follows along the lines of the proof of Lemma 2.6 in~\cite{lmcd06}.

We set
$$
q(\ell,g) = (23k + 72g) \eps^{-1} n (1-\la)^{-1} + 8 \ell n,
$$
as in Proposition~\ref{prop.stay-in-H}.  We assume throughout this section that $q(\ell,g) \le s_0/2$.

\begin{lemma}
\label{lem.coalesce-adj}
Let $x,y$ be a pair of adjacent states
in $\cA_0(\ell,g)$, with $x(j_0) = y(j_0)-1$ for some queue $j_0$, and $x(j) = y(j)$ for $j \not = j_0$.  Consider coupled copies $(X^x_t)$ and
$(X^y_t)$ of the $(n,d,\la)$-supermarket process, where $X^x_0 = x$ and $X^y_0=y$.
For all times $t \ge 2q(\ell,g)$, we have
\begin{eqnarray*}
\E \|X^x_t - X^y_t\|_1 &=& \Pr (X^x_t \not = X^y_t) \\
&\le& e^{-\frac14 \log^2 n} + 4\exp \left( - \frac{t}{3200 k d^{k-1} n}\right).
\end{eqnarray*}
\end{lemma}

\begin{proof}
By Lemma~\ref{lem.coupling-distance}, $X^x_t$ and $X^y_t$ are always neighbours or equal, always $X_t^x \le X_t^y$, and if for some
time $s$ we have $X^x_s= X^y_s$, then $X^x_t = X^y_t$ for all $t \ge s$. Thus in particular $\E \|X_t^x - X_t^y\|_1 = \Pr (X^x_t \not = X^y_t)$.

Initially, the queue $j_0$ is unbalanced, i.e., $X^x_0 (j_0) \not = X^y_0(j_0)$, and all other queues are balanced.  Observe that the index of the
unbalanced queue in the coupled pair of processes may change over time. Let $W_t$ denote the longer of the unbalanced queue lengths at time $t$, if there
is such a queue, and let $W_t = 0$ otherwise. The time for the two coupled processes to coalesce is the time $T$ until $W_t$ hits $0$.

Let us first run $(X_t^x)$ and $(X_t^y)$ together using the coupling.  Let $T_\cH^x$ and $T_\cH^y$ denote the times $T_\cH$, as defined
in Section~\ref{sec.hit-and-exit}, for the two copies of the process, and set
$T^*_\cH = T_\cH^x \lor T_\cH^y$.  By Theorem~\ref{thm.stay-in-H}, $T^*_\cH \le q(\ell,g)$ with probability at least
\begin{eqnarray*}
1-\frac{2(6k+28)}{s_0} &=& 1 - (12k+56)e^{-\frac13 \log^2 n} \\
&\ge& 1 - (12 \log n + 56) e^{-\frac13 \log^2 n} \\
&\ge& 1 - \frac13 e^{-\frac14 \log^2 n},
\end{eqnarray*}
where we used (\ref{ineq:116}); the final inequality holds for $n \ge 10000$.

We now track the performance of the coupling after time $T^*_\cH$. If the processes have coalesced by time $T^*_\cH$ (i.e., if $T \le T^*_\cH$),
then we are done.  Otherwise, $X^x_{T^*_\cH}$ and $X^y_{T^*_\cH}$ are still adjacent, and there is some random index $J_0$ such that the queue $J_0$
is unbalanced, i.e., $X^x_{T^*_\cH} (J_0) \not = X^y_{T^*_\cH}(J_0)$, and all other queues are balanced.  Moreover, since $u_{k+1}(x) = 0$ for
all $x \in \cH$, we have $W_{T^*_\cH} \le k$.

We shall use Lemma~\ref{lem.return-time} to give a suitable upper bound on $\Pr (W_t > 0)$. The idea is that, since, with high probability,
both copies of the process remain in $\cH$ for a long time, the unbalanced queue length $W_t$ will often be driven below $k$, and then
there is a chance of going all the way down to $0$.

For each $t \ge 0$, let $B_t$ be the event that $X^y_s, X^x_s \in \cH$ for all $s$ with $T^*_\cH \le s \le t-1$.
It follows from Theorem~\ref{thm.stay-in-H} that $\Pr (\overline{B}_t) \le (12k+56)/s_0 \le \frac13 e^{-\frac14 \log^2 n}$ (as above),
provided $t \le s_0$.

Let $N_r$ be the number of jumps of the longer unbalanced queue length in the first $r$ steps after $T^*_\cH$.  Also set $N=N_T$, the total
number of these jumps, with $N_T=0$ if $T \le T^*_\cH$.  For $j=1,2, \ldots$, let $T_j$ be the time of the $j$th jump after $T^*_\cH$ if $N \ge j$,
and otherwise set $T_j = T^*_\cH \lor T$. %[check if this makes sense]
Thus, if $T_\cH^* < T$, we have $T^*_\cH < T_1 < \cdots < T = T_N = T_{N+1} = \cdots$.  If $T_\cH^* \ge T$, then all of the $T_j$ are equal to $T_\cH^*$.

Let $S_0 = y(J_0)= W_{T^*_\cH} \1_{\{T^*_\cH < T\}}$, the longer unbalanced queue length at time $t=T^*$ if coalescence has not occurred.  For each
positive integer $j$, if $N \ge j$, let $S_j = W_{T_j}$, which is either $0$ or the longer of the unbalanced queue lengths at time $T_j$,
immediately after the $j$th arrival or departure at the unbalanced queue.  Also, if $N \ge j$, let $Z_j$ be the $\pm 1$-valued random variable
$S_j - S_{j-1}$.  For each non-negative integer $j$, let $\phi_j$ be the $\sigma$-field ${\mathcal F}_{T_{j+1}-1}$, of all events before
time $T_{j+1}$.  Let also $A_j$ be the $\phi_j$-measurable event $B_{T_{j+1}}$, that is the event that $X^y_s,X^x_s \in \cH$
for each $s$ with $T^*_\cH \le s \le T_{j+1}-1$.

We shall use Lemma~\ref{lem.return-time}.  We shall take the sequences $(\phi_j)_{j\ge 0}$, $(Z_j)_{j\ge 0}$, $(S_j)_{j\ge 0}$ and $(A_j)_{j\ge 0}$
as defined above, and we set $k_0=k$ and $\delta = 1/(\la d +1)$.  Note first, at any time $t < T$, the probability, conditioned on $\cF_t$, of an arrival
to the longer of the unbalanced queues is at most $d\la/n(1+ \la)$, while the conditional probability of a departure from that queue is
$1/n (1+\la)$.  Therefore, on the event that $N \ge j$, the probability, conditioned on $\phi_{j-1}$, that the event at time $T_j$ is a departure from the
longer unbalanced queue is at least
$$
\frac{1/n(1+\la)}{1/n(1+\la) + d\la/n(1+\la)} = \frac{1}{1+d\la} = \delta.
$$
In other words, on the event $N \ge j$ we have
$\Pr (Z_j = -1\mid \phi_{j-1} ) \ge \delta$.

We now show that, on the event $\{N \ge j\} \cap A_{j-1} \cap \{S_{j-1} \ge k\}$, we have
$$
\Pr (Z_j = -1 \mid \phi_{j-1} ) \ge \frac34.
$$
To see this, consider a time $t \ge T^*_\cH$.  On the event $B_t$, we have $X_t \in \cH \subseteq \cE_1$, and so, by Lemma~\ref{lem.no-new-k+1},
the conditional probability $\Pr(L_{t+1} \mid \cF_t)$ that the event at time $t+1$ is an arrival to a queue of length $k$ or greater is at most
$e^{-\log^2 n}$.  In particular, on the event
$B_t \cap \{W_{t-1} \ge k\}$, the conditional probability that the event at time $t+1$ is an arrival joining the longer unbalanced queue is at most
$e^{-\log^2 n}$, while the conditional probability that the event at time $t+1$ is a departure from the longer unbalanced queue is
$1/n(1+\la)$.
Therefore, on the event $\{N \ge j\} \cap A_{j-1} \cap \{S_{j-1} \ge k\}$, we have
$$
\Pr (Z_j = -1 \mid \phi_{j-1} ) \ge \frac{1/n(\la +1)}{1/n(\la+1) + e^{-\log^2 n}} \ge \frac34,
$$
for $n \ge 10$.

We have now shown that $S_m-S_0$ can be written as a sum $\sum_{i=1}^m Z_i$ for $\{0,\pm1\}$-valued random variables
$Z_i$ that satisfy the conditions of Lemma~\ref{lem.return-time}, with $k_0=k$ and $\delta=1/(\la d+1)$.  (The argument above establishes
this for $m \le N$: for $m > N$, we have set $Z_m = S_m =0$, which also meets the requirements of the lemma.)
Note that
$$
\delta^{-(k-1)} = (\la d +1)^{k-1} \le d^{k-1} (1+1/d)^k \le d^{k-1} e^{k/d} \le d^{k-1} e^{\eps/150 \sqrt d} \le 2 d^{k-1},
$$
where we used~(\ref{ineq:6}).
Hence, for $m \ge 16 k$,
\begin{eqnarray*}
\Pr \Big ( \bigcap_{i=1}^m \{S_i \not = 0\} \cap \bigcap_{i=0}^{m-1} A_i \Big ) &\le& \Pr (S_0 > \lfloor m/16 \rfloor) +
3 \exp\left( - \frac{\delta^{k-1}}{200k} m\right) \\
&\le& 0 + 3\exp \left( - \frac{m}{400 k d^{k-1}}\right).
\end{eqnarray*}
Here $\Pr(\cdot)$ refers to the coupling measure in the probability space of Section~\ref{sec.coupling}, with coupled copies of the process for each
possible starting state.

Let $q = q(\ell,g)$, take $r$ with $64k n \le r \le s_0 - q$ and let $m = \lfloor r/4n \rfloor \ge 16k$.
Since, at each time after $T^*_\cH$ and before $T$, a jump in the longer unbalanced queue occurs with probability at least $1/2n$ while the queue is nonempty,
we have, by inequality~(\ref{eq.bin-lower}),
$\Pr (\{T > T^*_\cH + r\} \cap \{N_r < m\}) \le e^{-r/16n}$.
Also,
\begin{eqnarray*}
\Pr \Big (\{N_r \ge m\} \cap \bigcup_{i=0}^{m-1} \overline{A_i} \cap \{ T^*_\cH \le q\} \Big) &\le& \Pr (\overline{B_{q+r}}) \\
&\le& \Pr (\overline{B_{s_0}}) \\
&\le& \frac13 e^{-\frac14 \log^2 n}.
\end{eqnarray*}

Now we have that
\begin{eqnarray*}
\Pr (T > q + r) &\le& \Pr (T^*_\cH > q) + \Pr (\{ T > T^*_\cH + r \} \cap \{ T^*_\cH \le q \}) \\
&\le& \Pr (T^*_\cH > q) + \Pr (\{T > T^*_\cH + r\} \cap \{N_r < m\}) \\
&& \mbox{} + \Pr \Big (\{N_r \ge m\} \cap \bigcup_{i=0}^{m-1} \overline{A_i} \cap \{ T^*_\cH \le q \} \Big )\\
&& \mbox{} + \Pr \Big ( \{N_r \ge m\} \cap \bigcap_{i=0}^{m-1} A_i \cap \bigcap_{i=1}^m \{S_i \not = 0\} \Big ).
\end{eqnarray*}
To see this, note that $\{ N_r \ge m\} \cap \bigcup_{i=1}^m \{ S_i = 0\} \subseteq \{ T \le T^*_\cH +r \}$.
Now we have
\begin{eqnarray*}
\Pr (T > q+r) &\le& \frac13 e^{-\frac14 \log^2 n} + e^{-r/16n} + \frac13 e^{-\frac14 \log^2 n}  \\
&& \mbox{} + 3\exp \left( - \frac{r}{1600 k d^{k-1} n}\right) \\
&\le& \frac 23 e^{-\frac14 \log^2 n} + 4\exp \left( - \frac{r}{1600 k d^{k-1} n}\right),
\end{eqnarray*}
provided $64k n \le r \le s_0 - q$.

If $2q \le t \le s_0$, then setting $r = t - q \ge t/2$ gives
$$
\Pr (T > t) \le \frac 23 e^{-\frac14 \log^2 n} + 4\exp \left( - \frac{t}{3200 k d^{k-1} n}\right),
$$
which gives the required result.  For $t > s_0$, we have
$$
\Pr (T > t) \le \Pr (T > s_0) \le \frac 23 e^{-\frac14 \log^2 n} + 4\exp \left( - \frac{s_0 - q}{1600 k d^{k-1} n}\right) \le e^{-\frac14 \log^2 n},
$$
so the result holds in this case too.
\end{proof}

\begin{theorem}
\label{thm.coalesce-general}
Let $(X^x_t)$ and $(X^y_t)$ be two copies of the $(n,d,\la)$-supermarket process, starting in states $x$ and $y$ in $\cA_0(\ell,g)$.
Then, for $t \ge 2q(\ell,g)$, we have
$$
\E  \| X^x_t - X^y_t \|_1  \le 2gn \left( e^{-\frac14 \log^2 n} + 4\exp \left( - \frac{t}{3200 k d^{k-1} n}\right) \right).
$$
\end{theorem}

\begin{proof}
Given two distinct states $x$ and $y$ in $\cA_0(\ell,g)$, we can choose a path $x = z_0, z_1, \ldots, z_m=y$ of adjacent states in $\cA_0(\ell,g)$ from $x$
down to the empty queue-lengths vector and back up to $y$, where $m = \|x\|_1 + \|y\|_1 \le 2gn$.  By Lemma~\ref{lem.coalesce-adj},
for $t \ge 2q(\ell,g)$,
\begin{eqnarray*}
\E \|X^x_t - X^y_t\|_1 & \le & \sum_{i=0}^{m-1} \E \|X^{z_i}_t - X^{z_{i+1}}_t\|_1 \\
& \le & 2gn \left( e^{-\frac14 \log^2 n} + 4\exp \left( - \frac{t}{3200 k d^{k-1} n}\right) \right),
\end{eqnarray*}
as required.
\end{proof}

We saw in Corollary~\ref{cor.eqm} that $Y_t \in \cA_0(\ell,g)$ with probability at least $1-e^{-\frac14 \log^2 n}$,
whenever $\ell, g \ge k$, where $(Y_t)$ is a copy of the $(n,d,\la)$-supermarket process in equilibrium.
Thus we have the following corollary.

\begin{corollary} \label{cor.tv}
Take any $\ell, g \ge k$, and let $(X^x_t)$ be a copy of the $(n,d,\la)$-supermarket process with starting in a state $x \in \cA_0(\ell,g)$.
Also let $(Y_t)$ be a copy in equilibrium.  Then, for $t \ge 2q(\ell,g)$, we have
$$
d_{TV}(\cL(X^x_t),\cL(Y_t)) \le
2gn \left( 2 e^{-\frac14 \log^2 n} + 4\exp \left( - \frac{t}{3200 k d^{k-1} n}\right) \right).
$$
\end{corollary}

\begin{proof}
The total variation distance is at most the probability that $Y_0 \notin \cA_0(\ell,g)$, plus the maximum, over all states $y \in \cA_0(\ell,g)$, of
$\Pr (X^x_t \not= X^y_t)$.  By Corollary~\ref{cor.eqm} and Theorem~\ref{thm.coalesce-general}, this is at most
$$
e^{-\frac14 \log^2 n} +
2gn \left( e^{-\frac14 \log^2 n} + 4\exp \left( - \frac{t}{3200 k d^{k-1} n}\right) \right)
$$
$$
\le 2gn \left( 2 e^{-\frac14 \log^2 n} + 4\exp \left( - \frac{t}{3200 k d^{k-1} n}\right) \right),
$$
as claimed.
\end{proof}

This implies Theorem~\ref{thm.mixing2}, on setting $\eps = 1/60$ (the conclusion is independent of $\eps$, and $q(\ell,g)$ is decreasing in $\eps$,
so it is best to take the highest legitimate value), $\ell = \max (k, \| x\|_\infty)$ and $g = \max ( k, \|x\|_1 / n)$, so that
\begin{eqnarray*}
q(\ell,g) &\le& 60 (23k + 72 (k + \|x\|_1/n)) n (1-\la)^{-1} + 8 (k + \| x \|_\infty) n \\
&\le& 6000 k n (1-\la)^{-1} + 4320 \|x\|_1 (1-\la)^{-1} + 8 n \| x\|_\infty,
\end{eqnarray*}
as in the statement of the theorem.

We interpret Theorem~\ref{thm.mixing2} as saying that we have mixing in time of order $\max\{ q (\ell,g), k d^{k-1}n \log n \}$, where $q(\ell,g)$
is bounded as above, or alternatively of order
$$
\max \{ k n (1-\la)^{-1}, g n (1-\la)^{-1}, \ell n, k d^{k-1} n \log n \}.
$$
The first and last terms here are of similar magnitude: either could be larger.

Let us indicate briefly why the dependence on $\ell$ and $g$ is best possible.
Suppose first that, for the starting state $x$, $\| x \| = g n$, and that $g > 2k$.  For the equilibrium copy, $\Pr( \|Y_t\|_1  > kn) \le 1/s_0$,
so if $d_{TV} (\cL(X^x_t), \cL(Y_t)) \le 1/2$, then we have
$$
\Pr (\| X^x_0 \|_1 - \| X^x_t\|_1 \ge (g-k) n ) \ge \Pr ( \| X^x_t \|_1 \le k n ) \ge 1/3.
$$
Now, $\| X^x_0 \|_1 - \|X^x_t\|_1$ is at most the number of
potential departures minus the number of arrivals over the interval $[0,t]$, and this number is a sum of $t$ Bernoulli random variables of
mean $\frac{1}{1+\la} - \frac{\la}{1+\la}$, which is well-concentrated around
$\frac{1-\la}{1+\la} t$.  Therefore, if $d_{TV} (\cL(X^x_t), \cL(Y_t)) \le 1/2$, we must have
$\frac{1-\la}{1+\la} t \ge \frac23 (g-k)n$, which implies that $t \ge (g-k) n(1-\la)^{-1} \ge \frac12 g n (1-\la)^{-1}$.
Similarly, suppose that, in the initial state $x$, there is a queue of length at least $\ell \ge k (1-\la)^{-1}$.  In order for such a queue to be reduced to
length $k$, there must be at least $\frac12 \ell$ departures from the queue, and this is unlikely to occur before
time $t = \frac12 \ell n$.

We now show that mixing actually takes place faster if we start from a ``good'' state, i.e., a state in $\cN = \cN^\eps$.

\begin{lemma}
\label{lem.coalesce-adj-N}
Let $x,y$ be a pair of adjacent states
in $\cN^\eps$, with $x(j_0) = y(j_0)-1$ for some queue $j_0$, and $x(j) = y(j)$ for $j \not = j_0$.  Consider coupled copies $(X^x_t)$ and
$(X^y_t)$ of the $(n,d,\la)$-supermarket process.
For all times $t \ge 0$, we have
\begin{eqnarray*}
\E \|X^x_t - X^y_t\|_1 &=& \Pr (X^x_t \not = X^y_t) \\
&\le& e^{-\frac14 \log^2 n} + 4\exp \left( - \frac{t}{1600 k d^{k-1} n}\right).
\end{eqnarray*}
\end{lemma}

\begin{proof}
The proof is nearly identical to that of Lemma~\ref{lem.coalesce-adj}.  Here, instead of starting by running the two copies of the
process together until some time $T^*$, we make use of Theorem~\ref{thm.start-in-N-stay-in-N}, which tells us that, with
probability at most $1 - 2 (k+5) / s_0$, both $X^x_t$ and $X^y_t$ remain within $\cH^{6\eps}$ throughout the interval
$0 \le t \le s_0$.  We may thus repeat the proof of Lemma~\ref{lem.coalesce-adj} with $T^*$ and $q=q(\ell,g)$ replaced by~0, and we obtain the
result stated.
\end{proof}

Exactly as before, we can use this result to deduce the following.

\begin{theorem}
\label{thm.coalesce-general-N}
Let $(X^x_t)$ and $(X^y_t)$ be two copies of the $(n,d,\la)$-supermarket process with starting states $x$ and $y$ in $\cN^\eps$.
Then, for $t \ge 0$, we have
$$
\E \| X^x_t - X^y_t \|_1 \le n \left(e^{-\frac14 \log^2 n} + 4\exp \left( - \frac{t}{1600 k d^{k-1} n}\right) \right).
$$
\end{theorem}

Note that the conclusion is independent of $\eps$, and the hypothesis is weakest when $\eps$ is as large as possible, namely $\eps = 1/60$.

\begin{proof}
Take any two queue-lengths vectors $x$ and $y$ in $\cN^\eps$.  By Lemma~\ref{lem.path-connected}, there is a path
between $x = z_0z_1 \cdots z_m =y$ in $\cN^\eps$ of length $m \le 4n (1-\la) (\la d)^{k-1} \le n$ between $x$ and $y$.
The result now follows as in the proof of Theorem~\ref{thm.coalesce-general}.
\end{proof}

%\medbreak

As before, since $Y_0$ lies in $\cH^\eps \subseteq \cN^\eps$ with probability at least $1-e^{-\frac14\log^2 n}$, by Corollary~\ref{cor.eqm},
we may now deduce that the total variation distance $d_{TV}(\cL(X^x_t),\cL(Y_t))$ is at most
$$
e^{-\frac14\log^2 n} + n \left(e^{-\frac14 \log^2 n} + 4\exp \left( - \frac{t}{1600 k d^{k-1} n}\right) \right)
$$
$$
\le n \left( 2 e^{-\frac14 \log^2 n} + 4\exp \left( - \frac{t}{1600 k d^{k-1} n}\right) \right)
$$
whenever $x \in \cN^\eps$.  This result, with $\eps = 1/60$, is exactly the statement of Theorem~\ref{thm.rapid-mixing}.

Theorem~\ref{thm.rapid-mixing} shows that, from states $x \in \cN^\eps$, we have mixing to equilibrium in time of order $k d^{k-1} n \log n$.
We now indicate why this bound is approximately best possible, for any value of $\eps$ such that $(n,d,\la,k,\eps)$
satisfies the conditions of Theorem~\ref{thm.technical}, where also $\eps \le 1/30$.

Note that there is a state $z$ in $\cI^{3\eps} \subseteq \cH^{3\eps} \subseteq \cN^{3\eps}$ with $Q_k(z) \le (1-9\eps) n (1-\la) (\la d)^{k-1}$.
However, we know from Corollary~\ref{cor.eqm} that $\Pr (Y_t \in \cH^\eps) \ge 1 - e^{-\frac14 \log^2 n}$, so in order for $d_{TV}(\cL(X^z_t),\Pi)$ to
be small, we need that $Q_k(X^z_t) \ge (1-5\eps) n (1-\la) (\la d)^{k-1}$ with high probability.  Set $t = n (\la d)^{k-1}$.

For $x \in \cH^{3\eps}$, we obtain from Lemma~\ref{lem.drifts}, with a calculation almost exactly as in Lemma~\ref{lem.qk}, that
\begin{eqnarray*}
(1+\la) \Delta Q_k(x) &\le& (1-\la)(1+ \eps/2) - \frac{Q_k(x)}{n (\la d)^{k-1}} + \exp \left( - dQ_k(x) / kn \right)  \\
&\le&  (1-\la) (1+ \eps/2 - (1 - \eps)) + e^{-\frac{177}{50} \log n} \\
&\le& 2 \eps (1-\la),
\end{eqnarray*}
so $\Delta Q_k(x) \le 2 \eps (1-\la)$ also.  We know from Theorem~\ref{thm.start-in-I-stay-in-H} that, with probability at least
$1-(k+5)/s_0$, $X^z_s \in \cH^{3\eps}$ for all $s = 0, \dots, t-1$, and we also have that $Q_k(x) \le kn$ for every state $x$.  It follows that
\begin{eqnarray*}
\E Q_k(X^z_t) & = & Q_k(z) + \sum_{s=0}^{t-1} \E \left( \E ( \Delta Q_k(X^z_s) \mid \cF_s ) \right) \\
&\le& (1-9\eps) n (1-\la) (\la d)^{k-1} + 2 \eps t (1-\la) + kn \frac{k+5}{s_0} \\
&\le& (1-6\eps) n (1-\la) (\la d)^{k-1}.
\end{eqnarray*}

A result from~\cite{lmcd06} (adapted for discrete time) states that, for some absolute constant $c$, for any 1-Lipschitz function $f$,
any starting state $z$, any $t>0$ and any $u \ge 0$,
$$
\Pr (|f(X^z_t) - \E f(X^z_t)| \ge u) \le n e^{-cu^2/(t+u)}.
$$
Applying this with $f = Q_k$, $t = n (\la d)^{k-1}$ and $u = \eps t (1-\la)$, we find that
\begin{eqnarray*}
\lefteqn{\Pr (Q_k(X^z_t) > (1-5\eps) n (1-\la) (\la d)^{k-1})} \\
&\le& \Pr ( Q_k(X^z_t) - \E Q_k(X^z_t) > \eps n (1-\la) (\la d)^{k-1}) \\
&\le & n e^{-c\eps^2 n (1-\la)^2 (\la d)^{k-1} /2} \\
&\le & n e^{- 270 c k^2 \log^2 n \, d^{k-2}},
\end{eqnarray*}
using (\ref{ineq:8}) and (\ref{ineq:12}).
Therefore the mixing time is at least $t = n (\la d)^{k-1}$.

\end{document}